\documentclass[11pt]{amsart}
\usepackage[margin=2.23cm]{geometry}
\usepackage{amsmath}
\usepackage{amssymb}
\usepackage{amsfonts}
\usepackage{mathrsfs}
\usepackage{amsthm}
\usepackage{epsfig,epic}
\usepackage[english]{babel}
\usepackage[T1]{fontenc}

\usepackage{graphics}
\usepackage{graphicx}

\usepackage{eucal}
\usepackage{amscd}
\usepackage{amssymb}
\usepackage{verbatim}
\usepackage[hang]{subfigure}
\usepackage{hyperref}
\usepackage[capitalize]{cleveref}

\usepackage{eucal}

\usepackage{todonotes}

\newcommand{\col}[1]{\mathrm{\nu_{col}}(#1)}

\usepackage{textcmds}
\usepackage{enumerate}
\usepackage[figuresleft]{rotating}
\usepackage{caption}

\newtheorem{thm}{Theorem}[section]
\newtheorem{lemma}[thm]{Lemma}
\newtheorem{prop}[thm]{Proposition}
\newtheorem{cor}[thm]{Corollary}
\theoremstyle{definition}
\newtheorem{rem}[thm]{Remark}

\newtheorem{defn}{Definition}[section]
\numberwithin{equation}{section}

\newcommand{\be}{\begin{equation}}
\newcommand{\ee}{\end{equation}}
\newcommand{\bes}{\begin{equation*}}
\newcommand{\ees}{\end{equation*}}

\newcommand{\ud}{\mathrm{d}}
\newcommand{\BS}[2]{S_{#2}(#1)}

\newcommand{\N}{\mathbb{N}}

\newcommand{\RL}[1]{B_{#1}}

\newcommand{\R}{\mathcal{R}}

\newcommand{\Z}{\mathcal{Z}}
\newcommand{\RLp}[1]{B^{(#1)}}

\newcommand{\cA}{\mathcal{A}}

\newcommand{\ov}[1]{\overline{#1}}
\newcommand{\vep}{\varepsilon}
\newcommand{\Leb}[1]{Leb\left( #1\right)}

\newcommand{\MDCo}[1]{\ensuremath{{M D C}\left( #1 \right)}}
\newcommand{\MDCt}[3]{\ensuremath{{M D C}\left( #1,#2, #3 \right)}}

\newcommand{\RDCt}[3]{\ensuremath{{R D C}\left( #1,#2, #3   \right)}}
\newcommand{\RDCo}[1]{\ensuremath{{R D C}\left( #1  \right)}}

\newcommand{\triple}{\hat{T}}

\newcommand{\hatmu}{\ensuremath{\hat{\mu}}}
\newcommand{\muR}{\ensuremath{\mu}}

\newcommand{\modulo}[1]{| #1 |}

\title[Multiple mixing in area preserving flows]{Multiple mixing and parabolic divergence in  smooth area-preserving flows on higher genus surfaces.}

\author{Adam Kanigowski \and Joanna Ku\l{}aga-Przymus \and Corinna Ulcigrai}

\date{}

\makeindex

\begin{document}
\bibliographystyle{siam}
\maketitle

\begin{abstract}
We consider typical area preserving flows on higher genus surfaces and prove that the flow restricted to mixing minimal components is mixing of all orders, thus answering affimatively to Rohlin's multiple mixing question in this context. The main tool is a variation of the Ratner property originally proved by Ratner for the horocycle flow, i.e.\ the \emph{switchable} Ratner property introduced by Fayad and Kanigowski for special flows over rotations. This property, which is of independent interest, provides a quantitative description of the parabolic behaviour of these flows and has implication to joinings classification. The main result is formulated in the language of special flows over interval exchange transformations with asymmetric logarithmic singularities. 
We also prove a strengthening of one of Fayad and Kanigowski's main results, by showing that Arnold's flows are mixing of all oders for almost every location of the singularities.  

\end{abstract}
\tableofcontents
\section{Introduction and main results.}

In this paper we give a contribution to the ergodic theory of area-preserving flows and, more in general to the study of parabolic dynamical systems. Since the origins of the study of dynamics, with Poincar{\'e}, flows on surfaces have been one of the basic examples of dynamical systems. We consider smooth flows which preserve a smooth area form, also known as locally Hamiltonian flows (see Section~\ref{sec:locHam}).  In this context, we  address Rokhlin question on multiple mixing (see Section~\ref{sec:RokhlinMixing}) and prove a version of Ratner's property on parabolic divergence ((see Section~\ref{sec:ParDiv}).

\subsection{Locally Hamiltonian flows}\label{sec:locHam}
Denote by   $S$ a smooth closed connected orientable surface of genus $g\geq 1$, endowed with the standard area form $\omega$ (obtained as pull-back of the area form $\ud x \ud y$ on $\mathbb{R}^2$). We will consider a smooth flow $(\varphi_t)_{t\in\mathbb{R}}$ on $S$ which preserves a measure $\mu$ given  integrating a smooth density with respect to $\omega$. We will assume that the area is normalized so that  $\mu(S) =1$. As explained in Section~\ref{sec:locHam}, smooth area preserving flows  are in one to one correspondence with smooth closed real-valued differential $1$-forms and are  \emph{locally Hamiltonian flows}, also known as \emph{multi-valued Hamiltonian flows}. A lot of interest in the study of multi-valued Hamiltonians and the associated flows -- in particular, in their ergodic and mixing properties -- was sparked 
by Novikov \cite{No:the} in connection with problems arising in solid-state
physics (i.e.\ the motion of an electron in a metal under the action of a magnetic field) 
and in pseudo-periodic topology (see e.g.\ the survey by Zorich \cite{Zo:how}).

When $g\geq 2$, the  (finite) set of fixed points of $(\varphi_t)_{t\in\mathbb{R}}$ is always non-empty.  A \emph{generic} locally Hamiltonian flow (in the sense of Baire category, with respect to the topology given by considering perturbations of closed smooth $1$-forms by 
 (\emph{small}) {closed} smooth $1$-forms)  
 has only non-degenerate fixed points, i.e.\ \emph{centers} (see Figure \ref{island}) and \emph{simple saddles} (see Figure \ref{simplesaddle}), as opposed to degenerate \emph{multi-saddles} which have $2k$ separatrixes for $k>2$ (see Figure \ref{multisaddle}). From the point of view of topological dynamics (as proved independently  by Maier \cite{Ma:tra}, Levitt \cite{Le:feu} and Zorich \cite{Zo:how}), each smooth area-preserving flow can be decomposed into  {periodic components} and {minimal components}: a \emph{periodic component} is a subsurface (possibly with boundary) on which all orbits are closed and periodic (see for example Figure \ref{island} or Figure \ref{cylinder}); \emph{minimal components} (there are not more than $g$ of them) are subsurfaces (possibly with boundary) on which the flow is \emph{minimal} in the sense that all semi-infinite trajectories are dense (see Figure \ref{decomposition}).

 \begin{figure}[h!]
  \subfigure[ \label{island}]{
  \includegraphics[width=0.23\textwidth]{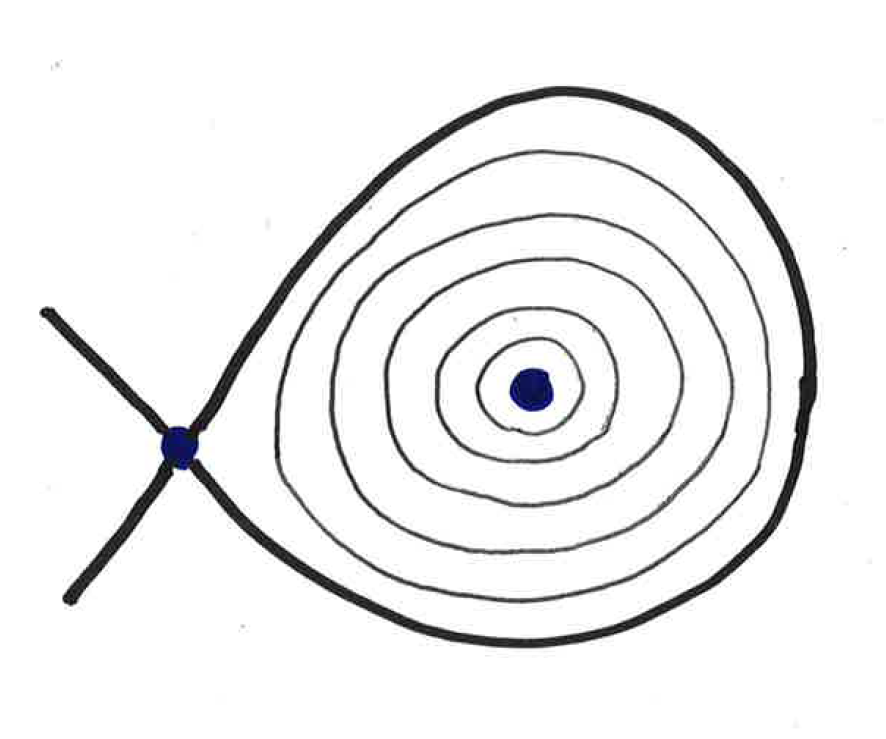} 	} \hspace{3mm} \subfigure[ \label{simplesaddle}]{ \includegraphics[width=0.18\textwidth]{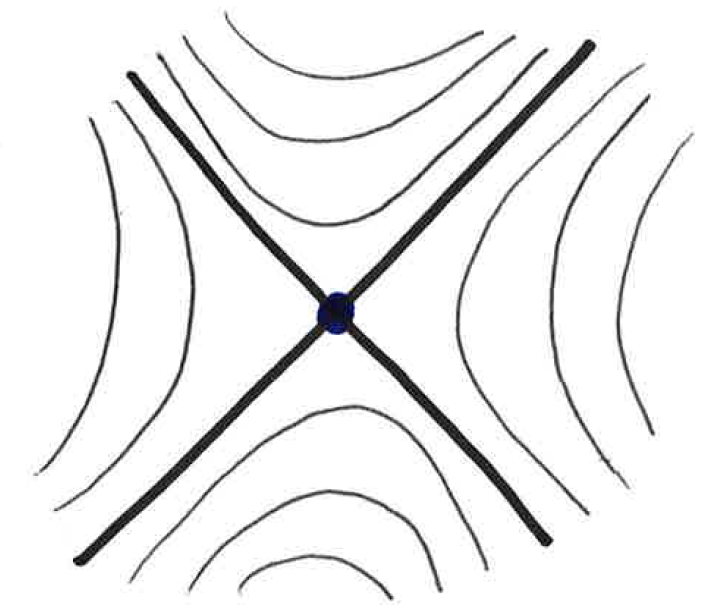}}\hspace{3mm}
\subfigure[\label{multisaddle}]{
\includegraphics[width=0.18\textwidth]{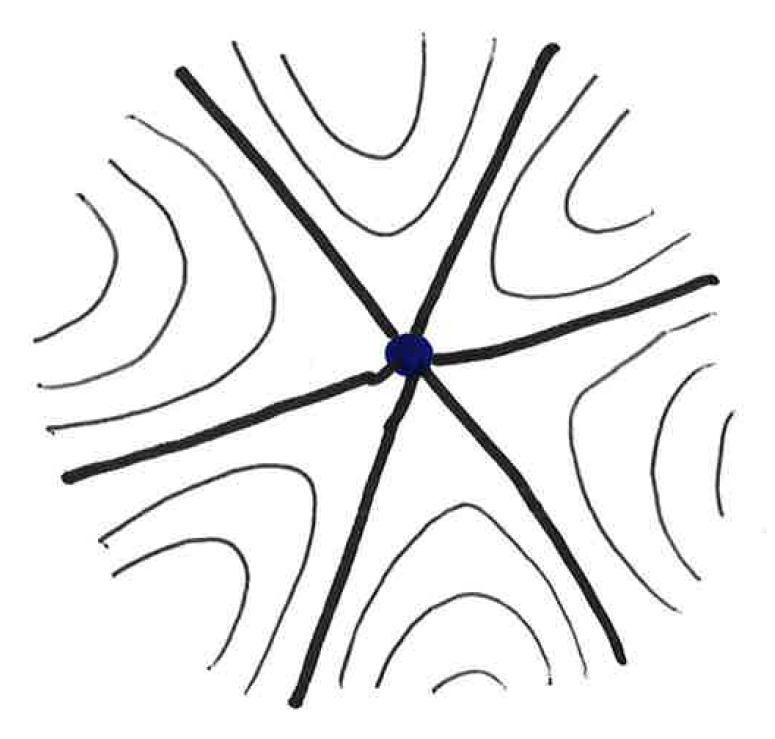} }	\hspace{3mm}  \subfigure[ \label{cylinder}]{	\includegraphics[width=0.23\textwidth]{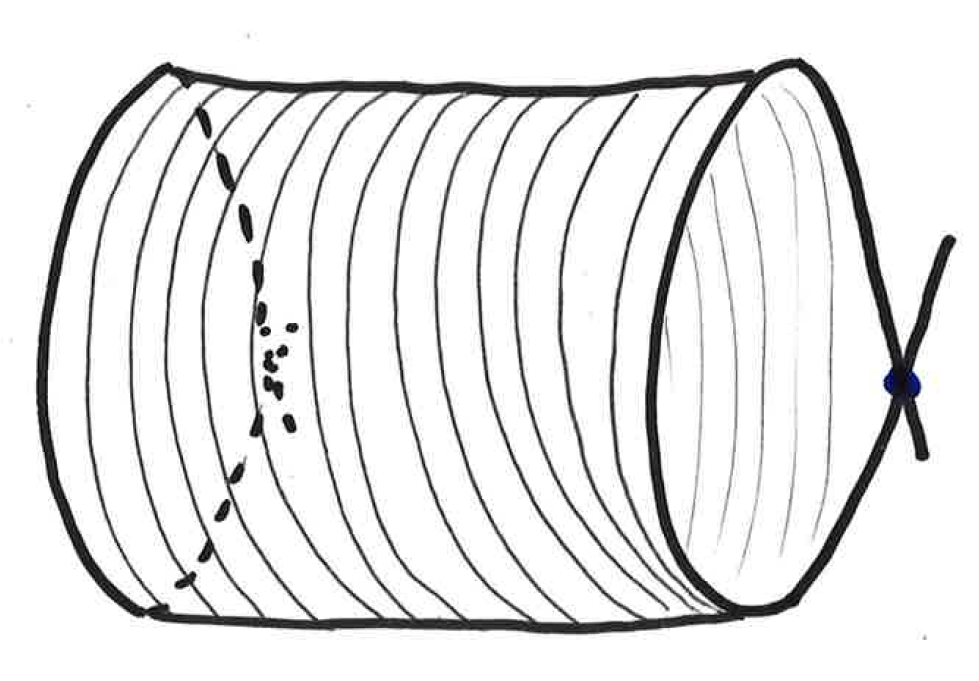}}
 \caption{Type of fixed points and periodic components in an area-preserving flow.\label{saddles}}
\end{figure}

\begin{figure}[h!]
  \includegraphics[width=0.6\textwidth]{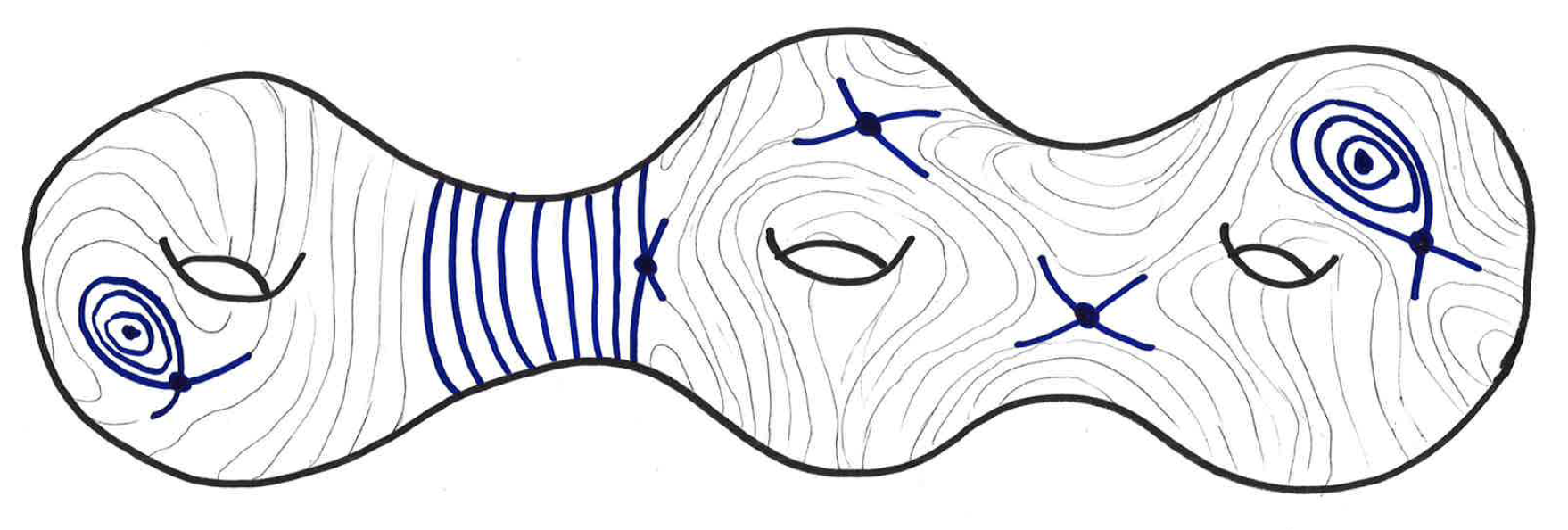}
\caption{Decomposition in periodic components filled by closed orbits (two islands around centers and a cylinder in blue in the Figure) and minimal components (one of of genus one and one of genus  two in the example).\label{decomposition}}
\end{figure}

We will focus on the ergodic properties of a \emph{typical} locally Hamiltonian flow in the sense of measure-theory. In particular, one can define a \emph{measure class} on locally Hamiltonian flows (sometimes called \emph{Katok fundamental class}, see  Section~\ref{sec:locHam} for the definition): when we say \emph{typical}, we mean full measure with respect to this measure class. One can divide locally Hamiltonian flows into two open sets (in the topology given by  perturbations by closed smooth $1$-forms, see Section~\ref{sec:locHam} for more details): in the first open set, which we will denote by $\mathscr{U}_{min}$, the typical flow is \emph{minimal} (in particular there are no centers and there is a unique minimal component) and \emph{ergodic} (i.e.\ there are no measureable flow-invariant sets $A \subset S$ such that $\mu(A)\notin \{0,1\}$). On the other open set that we will call $\mathscr{U}_{\neg min}$ there are periodic components (bounded by saddle loops homologous to zero), but minimal components of typical flows are still minimal and uniquely ergodic. Both results can be deduced from the representation of minimal components as special flows described in Section \ref{sec:special_flows_rep} below and the classical results by Keane \cite{Ke:int} and Masur and Veech (see \cite{Ma:int, Ve:gau}) respectively concerning of minimality and ergodicity of typical interval exchange transformations.

\subsection{Mixing, Rokhlin's question and multiple mixing}\label{sec:RokhlinMixing}
Stronger chaotic properties than ergodicity are mixing and multiple (or higher order) mixing. A flow $(\varphi_t)_{t\in\mathbb{R}}$ preserving a probability measure $\mu$ is \emph{mixing} (or strongly mixing) if any pair of measurable sets $A,B$ become asymptotically independent under the flow, i.e.\ $\lim_{t \to \infty} \mu (\varphi_t (A) \cap B)= \mu(A) \mu(B)$. More in general, $(\varphi_t)_{t\in\mathbb{R}}$ is 
\emph{mixing of order $k\geq 2$} if for any $k$ measurable sets $A_1, A_2, \dots, A_k $, 
\bes
 \lim_{t_2, \dots, t_k \to \infty} \mu ( A_1 \cap \varphi_{t_2}(A_2) \cap \varphi_{t_2+t_3}(A_3) \cdots \cap \varphi_{t_2+\dots +t_k}(A_k)  = \mu (A_1) \mu(A_2) \dots \mu(A_k).
\ees 
Clearly, for $k=2$ the above definition reduces to \emph{mixing}. If a flow is mixing of order $k$ for any $k\geq 2$, we say that it is \emph{mixing of all orders}.

Arnold in  \cite{Ar:top} conjectured that when $g=1$ and there is at least one periodic component (i.e.\ one is in  $\mathscr{U}_{\neg min}$), the typical locally Hamiltonian flow restricted to the minimal component is mixing. Following \cite{FK:mul}, we  will call \emph{Arnold flow} the restriction of such a flow to its minimal component. 
 This conjecture was proved by  Khanin and Sinai~\cite{SK:mix}  (see also further works by Kochergin \cite{Ko:nonI, Ko:nonII, Ko:som, Ko:wel}). On the other hand, in  $\mathscr{U}_{min}$, when $g=1$, the typical flow is not mixing (this can be deduced either from the work \cite{Ko:abs} by Kochergin or by classical KAM theory).   Mixing on surfaces of higher genus (i.e.\ when $g\geq 2$) was investigated by the last author. She showed~\cite{Ul:abs} that in the open set  $\mathscr{U}_{min}$, the typical flow, which is minimal and ergodic, is \emph{not} mixing (see also \cite{Sch:abs} for an independent proof of the same result when $g=2$), even though it is \emph{weakly mixing} \cite{Ul:wm}. On the other hand, Ravotti~\cite{Ra:mix}, by generalizing  the main  result proved by the last author in~\cite{Ul:mix} in the context of special flows, recently showed that typical flows in 
  $\mathscr{U}_{\neg min}$ have mixing minimal components and provided quantitative bounds on the speed of mixing for smooth observables (showing that mixing happens at subpolynomial rates).   Let us also recall that in the $1970s$ Kochergin \cite{Ko:mix} proved mixing when there are degenerate saddles (that is, in a non-generic case) and  that very recently Chaika and Wright~\cite{ChW:mix} showed the existence of mixing flows in  $\mathscr{U}_{min}$ (which by \cite{Ul:abs} consistute a measure zero exceptional set). 
 
\smallskip
A famous and still widely open problem in ergodic theory is the \emph{question by Rokhlin}  
whether mixing implies mixing of all orders \cite{MR0030709}. In the context of area-preserving flows, Fayad and the first author recently proved in \cite{FK:mul} that when $g=1$ flows with a mixing minimal compoment (as well as some mixing Kochergin flows with degenerate saddles) are indeed mixing of all orders, thus verifying Rokhlin's question in this context. 

\smallskip
Our main result is that mixing implies mixing of all orders for typical smooth area-preserving flows on surfaces of any genus.
\begin{thm} \label{thm:main}
For any fixed genus $g\geq 1$, consider locally Hamiltonian flows on a surface $S$ of genus $g$ with  non-degenerate fixed points and at least one periodic component. For a typical flow $(\varphi_t)_{t \in \mathbb{R}}$ in an open and dense set, the restriction of $(\varphi_t)_{t \in \mathbb{R}}$ to any of its minimal components  is  mixing of all orders.
\end{thm}
More precisely, the open and dense set of flows with at least one periodic component in the statement is the same set in $\mathscr{U}_{\neg min}$ for which one can also prove that typically minimal components are mixing (see \cite{Ra:mix}).  
In particular, since  typical flows on $\mathscr{U}_{min}$ are \emph{not} mixing by \cite{Ul:abs}, 
it follows that for a typical locally Hamiltonian flow mixing (of one of its minimal components) implies mixing of all orders. 

\smallskip
This result is deduced from a more general result in the language of \emph{special flows} (see Theorem \ref{prop:main_special_flows} below).  
Consider a segment transverse to a minimal component of an area-preserving  flow and the associated Poincar{\'e} first return map (i.e.\ the map that sends a point to the first point along its flow orbit that hits the same segment again). Poincar{\'e} maps of smooth area preserving flows, in suitable coordinates, are \emph{interval exchange transformations} (for short, IETs), which are piecewise isometries of an interval $I$ (see Section~\ref{IETsdefsec} for the definition). Given an IET  $T\colon I \to I $  which occurs as a Poincar{\'e} section of the surface and the return time function (which is an integrable  function $f\colon I \to \mathbb{R}$ defined at all but finitely many points) one can recover the flow as follows. Let  $X_f$ consist of points below the graph of $f$, that is  $X_f \doteqdot \{ (x,y) \in \mathbb{R}^2 :  x \in I, \, 0\leq y < f(x) \}$. Under the action of the \emph{special flow} $(\varphi_t)_{t \in \mathbb{R}}$ \emph{over the map} $T$ \emph{under the roof} $f$ a point $(x,y) \in X_f$ moves  with unit velocity along the vertical line up to the point $(x,f(x))$, then jumps instantly to the point $\left( T(x),0 \right)$, according to the base transformation and afterward it continues its motion along the vertical line until the next jump and so on. The formal definition is given in Section~\ref{suspflowsdefsec}. Since $T$ preserves the Lebesgue measure, any special flow over $T$ preserves the restriction of the two dimensional Lebesgue measure to $X_f$. It is well known that the original flow on the minimal component is measure-theoretically isomorphic to the described special flow and hence has the same ergodic and mixing properties.

Each minimal component of a locally Hamiltonian flow can be represented as a special flow over an IET. The corresponding roof function is not defined at a subset of discontinuities of the IET which correspond to points that hit a saddle before their first return, see Section~\ref{sec:reduction}. Since the flow is smooth, these discontinuities are singularities of the roof function (we have $f(x) \to +\infty$ as $x$ approches one or both sides of such a discontinuity). It turns out that non-degenerate simple saddles give rise to \emph{logarithmic singularities} of $f$, i.e.\ points where $f$ blows up as the fuction $|\log(x)|$ near zero, in a sense made precise in Section~\ref{sec:reduction} (while degenerate saddles give rise to \emph{polynomial singularities}, which are the type of singularities considered by Kochergin in \cite{Ko:mix} and also in part of \cite{FK:mul}).  Furthermore, for a typical flow in the open set $\mathscr{U}_{\neg min}$, logarithmic singularities are \emph{asymmetric}, see Section~\ref{sec:reduction}.

Our main result in the set up of special flows is the following. Recall that an IET $T$ is given by a combinatorial datum $\pi $ and a lenght datum $\underline{\lambda}$ which describe respectively the order and lenghts of the exchanged subintervals (see Section~\ref{IETsdefsec}). We say that a result holds for \emph{almost every } IET if it holds for any irreducible $\pi$ and Lebesgue almost every choice of $\underline{\lambda}$ (see Section~\ref{IETsdefsec}). 
\begin{thm} \label{prop:main_special_flows} For {almost every} interval exchange transformation  $T\colon I \to I$ and every roof function $f\colon I \to \mathbb{R}^+$ with \emph{asymmetric logarithmic singularities}  at the discontinuities of $T$ (in the sense of Definition \ref{def_LogAsym}), the special flow $(\varphi_t)_{t\in\mathbb{R}}$ over $T$ under $f$ is mixing of all orders. 
\end{thm}

\subsection{Parabolic divergence and Ratner properties}\label{sec:ParDiv}

The result that we use as a crucial tool to prove multiple mixing is that the flows that we consider in Theorem~\ref{thm:main} and Theorem~\ref{prop:main_special_flows} satisfy a variation of the so called \emph{Ratner property of parabolic divergence}. 
We believe that this is a result of independent interest, since it describes a central feature which shows the parabolic behaviour of flows we study. 
Ratner introduced in \cite{Ra:ho} a property that she called $\mathcal{H}_p$-property and was later known as \emph{Ratner property} (\cite{Thou}). This property, whose formal definition we omit since it is rather technical (see Section~\ref{sec:Ratner} for a more general definition) describes in a precise  quantative way how fast nearby trajectories diverge. Ratner first verified this property for horocycle flows and used it to deduce some of the main rigidity properties of horocycle flows (such as very specific properties of joinings and measure rigidity). The horocycle flow can be considered as the main example in the class of \emph{parabolic flows}, i.e.\ it is a continuous dynamical systems in which nearby orbits diverge polynomially. The Ratner property as originally defined by Ratner holds by virtue of this polynomial divergence.


It is reasonable to expect that some quantitative form of parabolic divergence similar to the Ratner property should hold  and be crucial in proving analogous rigidity properties for other classes of parabolic flows.  Thus, the natural question arose, whether there are parabolic flows satisfying the Ratner property beyond the class of horocycle flows. A positive answer to this question was given by
K. Fr\k{a}czek and M. Lema{\'{n}}czyk in \cite{FL}. The authors showed that a  \emph{variant}  of Ratner's property is satisfied in the class of special flows over irrational rotations of constant type and under some roof functions of bounded variation ($f(x)=ax+b$, $a,b>0$ being the most important example). The new property, called  the {\em finite Ratner's property} in \cite{FL} (see Section~\ref{sec:Ratner} for the definition) was shown to imply  the same joining consequences as the original one.  The finite Ratner property was further weakend by the two authors in \cite{FL2} to {\em weak Ratner's property} (see Section~\ref{sec:Ratner}), which was shown two hold in the class of special flows over two--dimensional rotations and some roof functions of bounded variation ($f(x,y)=ax+by+c$, $a,b,c>0$ being one example). All the dynamical consequences of the original Ratner's property hold also for the weak Ratner property, \cite{FL2}. The assumption on the roof being of bounded variation was crucial in \cite{FL} and \cite{FL2} and unfortunately this assumption is not verified for special flow representations of Arnold flows and more in general locally Hamiltonian flows in higher genus (since the roof function has logarithmic singularities and hence is not of bounded variation).

 In presence of singularities such as the fixed points of smooth area-preserving flows, the Ratner property in its classical form, as well as the weaker versions defined by K. Fr\k{a}czek and M. Lema{\'{n}}czyk is currently expected to fail. 
 The heuristic problem for Arnold flows and more generally smooth area-preserving flows to enjoy the Ratner property (or its weaker versions) is that Ratner-like properties require a (polynomially) controlled way of divergence of orbits  of nearby points. If the orbits of two nearby points get too close to a singularity, their distance explodes in an uncontrolled manner. This intuition was shown to be correct in \cite{FK:mul} (see Theorem 1 in the Appendix B in \cite{FK:mul}), where the first author showed that special flows over irrational rotation of constant type, under a roof function of the form $f(x)= x^{\gamma}, -1<\gamma<0$ or $f(x)=-\log(x)$ do not satisfy the weak Ratner property. 
 
To deal with this issues,  in \cite{FK:mul}, B.~Fayad and the first author introduced a new modification of the weak Ratner property, so called \emph{SWR-property} (the acronym standing  for \emph{Switchable Weak Ratner}), according to which one is allowed to choose whether we see polynomial divergence of orbits in the future or in the past, depending on points (see Definition~\ref{def:SWR} in Section~\ref{sec:Ratner}). All the dynamical consequences of the Ratner property are still valid for SWR-property. In particular, a mixing flow with SWR-property is mixing of all orders (see  Section~\ref{sec:Ratner}). The main result in  \cite{FK:mul} in the language of special flows is the following. 

\begin{thm}\label{mainthm:FK}
Consider the special flow  $(\varphi_t)_{t\in\mathbb{R}}$  over a rotation $R_\alpha(x) = x+ \alpha \mod 1$ and under  a roof function $f\colon I \to \mathbb{R}^+$ with \emph{one asymmetric logarithmic singularity} at the zero. For almost every $\alpha \in [0,1]$, $(\varphi_t)_{t\in\mathbb{R}}$ satisfies the SWR-property and hence is mixing of all orders.  Furthermore, the same result holds if $f$  has several asymmetric logarithmic singularities, under a non resonance condition (of full Hausdorff dimension) between the positions of the singularities and the base frequency $\alpha$. 
 \end{thm}
In particular, since for Arnold flows on tori (i.e. the restriction of a smooth area preserving flow on a surfaces of genus one to its minimal component), the base automorphism in the special flow representation of an  Arnold flow is an irrational rotation and the roof function has asymmetric logarithmic singularities,  it follows from the above Theorem that typically Arnold flows with one fixed point satisfy the SWR-property 
 and hence are typically mixing of all orders. 

\smallskip
In this paper, we prove that a generalization of the Ratner property  holds for minimal components of typical smooth area preserving flows in  $\mathcal{U}_{\neg min}$ for surfaces of \emph{any genus}. More precisely, we consider a stronger property, the {\em SR-property} (acronym for \emph{Switchable Ratner}, without the W for  \emph{Weak} in SWR). This property trivially implies SWR-property (see the definitions in Section~\ref{sec:SWRRatner}). 
Our main result in the language of special flows is then the following.


\begin{thm}\label{thm:Ratner_special_flows}
For almost every IET $T\colon I \to I$ and every roof function $f\colon I \to \mathbb{R}^+$ with \emph{asymmetric logarithmic singularities}  at the discontinuities of $T$ (in the sense of Definition \ref{def_LogAsym}), the special flow $(\varphi_t)_{t\in\mathbb{R}}$ over $T$ and under $f$ has the SR-property.
\end{thm}
As a Corollary (see Section~\ref{sec:conclusions}), we have the following. 
\begin{cor}\label{cor:Ratner_flows}
For any fixed genus $g\geq 1$, consider locally Hamiltonian flows on a surface $S$ of genus $g$ with  non-degenerate fixed points and at least one periodic component. For a typical flow $(\varphi_t)_{t \in \mathbb{R}}$ in an open and dense set, the restriction of $(\varphi_t)_{t \in \mathbb{R}}$ to any of its minimal components  has the SR-property.
\end{cor}
It is from this result that we deduce Theorem \ref{prop:main_special_flows} on multiple mixing, since the SWR-property (and hence in particular the SR-property) allows us to automatically upgrade mixing to mixing of all orders and mixing for these flows is known by \cite{Ul:mix, Ra:mix}.  

Theorem  \ref{thm:Ratner_special_flows} and Corollary \ref{cor:Ratner_flows} have also implications on joining rigidity for the corresponding flows (since Ratner properties restrict the class of self-joinings for these flows, see Section~\ref{sec:Ratner}). Most crucially,  it shows the power and generality of the modification of the Ratner property introduced by Fayad and the second author in capturing the quantitative divergence behaviour for a  large class of parabolic flows with singularities.
A special case of Theorem \ref{prop:main_special_flows} also implies 
 (as shown in Section~\ref{sec:conclusions}) the following notable strenghtneing of the main result in \cite{FK:mul}, which was stated here as Theorem \ref{mainthm:FK}.
\begin{cor}\label{cor:genFK}
Consider the special flow  $(\varphi_t)_{t\in\mathbb{R}}$  over a rotation $R_\alpha(x) = x+ \alpha \mod 1$ and under  a roof function $f\colon I \to \mathbb{R}^+$ with \emph{asymmetric logarithmic singularities} at $0< x_1< \dots < x_d<1$ (in the sense of Definition \ref{def_LogAsym}).  For almost every $\alpha \in [0,1]$ and almost every choice of $(x_1, \dots, x_d)$ with respect to the Lebesgue measure on $[0,1]^d$, $(\varphi_t)_{t\in\mathbb{R}}$ satisfies the SR-property. Hence, in particular, it is mixing of all orders.  
 \end{cor}
We remark that the above Corollary generalizes Theorem \ref{mainthm:FK} in two directions: first of all, it shows that the flows considered in Theorem \ref{mainthm:FK} have the stronger SR-property instead than the SWR-property. Secondly, at  most importantly, our result holds for almost every location of the singularities, while in Corollary~\ref{cor:genFK} the location of the singularities $(x_1, \dots, x_d)$ was restricted to a subset of full Hausdorff dimension but Lebesgue measure zero. Notice though, that  the full measure set in Corollary~\ref{cor:genFK}  is not explicit, while the resonance condition in Theorem \ref{mainthm:FK} (see Definition 1.3. and Remark 1.4. in \cite{FK:mul} for details) may allow to construct explicit examples.


\subsection{Outline and structure of the paper}\label{sec:outline}
In this section we present an outline and some heuristic ideas used in the proof of Theorem \ref{thm:Ratner_special_flows}, namely parabolic divergence (in the precise form of the switchable Ratner property SR) for special flows over IETs under roof functions which have asymmetric logarithmic singularities (in the sense of Definition \ref{def_LogAsym}), since this is the key result from which the other results mentioned in the introduction are then deduced (see \cref{sec:conclusions}).  
Two of the  main ingredients used in the proof of Theorem \ref{thm:Ratner_special_flows} are a suitable full measure \emph{Diophantine condition on the IET} on the base and precise \emph{quantitative estimates on Birkoff sums of the derivatives} of the roof functions. We will first comment on these two parts. 

\smallskip
\emph{Diophantine conditions on IETs} are given through the \emph{Rauzy-Veech algorithm}, which can be thought of a generalization of the continued fraction algorithm for rotations (since rotations can be seen as IETs of two intervals). This a powerful and well studied tool to prove fine properties of IETs, which was developed by Rauzy \cite{Ra:ech} and Veech \cite{Ve:gau} and has been used very fruitfully in the past thirty years, for example, just to mention a few highlights,  to prove the main results in \cite{AF:wea, AGY:exp, AV:sim, Bu:dec, Bu:lim, Ch:dis, MMY:coh, MMY:lin, Ul:abs, Zo:dev} and many more.  The Rauzy-Veech algorithm associates to an IET of $d$ intervals a sequence of $d\times d$ integer valued matrices $A_\ell$ which can be thought of as entries of a multi-dimensional continued fraction algorithm. As Diophantine conditions for rotations are conditions on the growth of the continued fraction entries of the rotation number, Diophantine conditions for IETs can be expressed in terms of the growth of the norm of the matrices $A_\ell$. It is fruitful to consider accelerations of the original algorithm which are \emph{positive} (i.e. the matrices $A_\ell$ have strictly positive entries) and \emph{balanced} (i.e. times when the Rohlin towers in the associated dynamical representation of the initial IET as suspension  over an induced IET have approximately the same heights and widths). 

 One of the main points of this paper is the definition of a new Diophantine Condition for IETs, that we call the \emph{Ratner Diophantine Condition} (or Ratner DC for short). This implies by definition the Mixing DC and it was inspired by the Diophantine Condition for rotations introduced by Fayad and the first author in \cite{FK:mul} (see Remark \ref{FK_DC}). The proof that the Ratner DC is satisfied for a full measure set of IETs exploits subtle properties of Rauzy-Veech induction and its positive balanced accelerations, in particular a quasi-Bernoulli type of property and the full strenght of the deep exponential tails estimates given by Avila, Gouezel and Yoccoz in \cite{AGY:exp}.

\smallskip

Let us now give an intuitive explanation of why \emph{Birkhoff sums of derivatives} play an important role in both the proof of mixing and parabolic divergence for special flows over  IETs under roofs with logarithmic asymmetric singularities. Since $T $ is a piecewise isometry (hence $T'=1$ almost everywhere), by the chain rule we have that
$$\frac{\ud }{\ud x} S_r(f)(x) = S_r(f\rq{})(x) \text{ for \ a.e.\ }x, \text{\quad for each }r\in\mathbb{Z}.$$
Consider a small horizontal segment $J=[a,b] $ which undergoes exactly $r$ jumps when flowing for time $t$ under the roof $f$ and which is a continuity interval for $T^r$. By  calculating the explicit expression for the special flow iterates (see \cref{sec:special_flows_rep}), one can see that the image of the segment $J$ after time $t$ is given by 
$(T^r(x), S_r(f)(x)  )$ for  $ x \in J$. 
Thus, the Birkhoff sums $S_r(f\rq{})(x)$ for $x \in J$ describe the vertical slope of the image of $J$ after time $t$ under the flow. This slope contains information on the \emph{shearing phenomenon} which is crucial both to mixing and to parabolic divergence. For an heuristic explanation of how mixing can be deduced by shearing, we refer the reader to the outline of \cite{Ul:mix} or \cite{Ra:mix}; a reformulation of the SR-Property using estimates on Birkhoff sums of derivatives is presented in Section~\ref{sec:SWRRatner} and was already used  in~\cite{Ka-Ku} and in a special case in \cite{FK:mul}.

Note that if $f$ has logarithmic singularities, the derivative $f'$ is not in $L^1(\ud x)$, since it has singularities of type $1/x$, which are not integrable. Thus, one cannot  apply the Birkhoff ergodic theorem (which, for a function $g \in L^1(\ud x)$ guarantees that $S_r(g)/r $ converge pointwise almost everywhere to a constant and thus that $S_r(g)$ grows as $r$). One can indeed prove that, for a typical IET the growth of $S_r(f')$ when $f$ has asymmetric logarithmic singularities is of order $C r\log r$ where $C$ is a constant which describes the asymmetry of the singularities. This additional $\log r$ factor 
is responsible for the \emph{shearing phenomenon} at the base of mixing and parabolic divergence. Unfortunately, to control the growth of $S_r(f')$ precisely, one needs to throw away a set of initial points which changes with $r$: more precisely, if $r$ is between $q_{\ell}$ and $q_{\ell+ 1}$, where $q_\ell$ denotes the maximal heights of towers at step $n_\ell$ of the Rauzy-Veech acceleration, one needs to remove a set $\Sigma_\ell\subset [0,1]$ whose measure goes to zero as $\ell$ grows (see Proposition \ref{mpd} for the precise statement). 

We use these sharp estimates on Birkhoff sums of derivatives (in the form of Lemma \ref{prty}) to prove the SR-property of parabolic divergence. We face the problem, though, that  while mixing is an asymptotic property, and hence requires only that shearing can be controlled for arbitrarily large times $r$ outside of a set whose measure goes to zero (so it is enough to use that the measure of $\Sigma_\ell$ goes to zero with $\ell$), to control Ratner properties one needs to have shearing for \emph{all} arbitrarily large times for most points (i.e. on a set of arbitrarily large measure). If the series of the measures of $\Sigma_\ell$ were summable, tails would have arbitrarily small measures and thus one could throw away the union of the sets $\Sigma_\ell$ for $\ell $ large. Unfortunately, one can check that the measures of $\Sigma_\ell$ are \emph{not} summable. 

This is where the \emph{Ratner DC} helps, since, if $\widetilde{K}_T \subset \mathbb{N}$ denotes the set of induction times $\ell$ such that finite blocks of cocycle matrices starting with $A_\ell$ are not too large (not larger than a power of $\log q_\ell$, see \eqref{DCcon} in the Ratner DC definition for details), the Ratner DC  guarantees that times in $\mathbb{N}\setminus \widetilde{K}_T $ where this fails  are \emph{rare}, and hence it can be used to show that the sum of the   
measures of $\Sigma_\ell$ for $\ell \notin \widetilde{K}_T$ is finite (see Corollary \ref{summabilitySigmas}). Thus, these sets can be thrown for large $\ell$. One is then left to estimate the times $\ell \in \widetilde{K}_T$. This is where one exploits the versatility of the \emph{switchable} Ratner property, according to which, if the desired quantification of parabolic divergence does not hold for  \emph{forward} Birkhoff sums (see (i) in Definition \ref{def:SWR}), one can \emph{switch} the direction of time, i.e. prove quantiative divergence estimates on \emph{backward} Birkhoff sums (see (ii) in Definition \ref{def:SWR}).  Using properties of balanced times in Rauzy induction, we show that if an orbit of a point of lenght $q_\ell$ gets too close to a singularity in the future (where too close is of order $q_{\ell+L}$ for a fixed $L$), then it did not come that close to a singularity in the past (this is proved in Proposition \ref{forbac}).   
 Thus, using that if $\ell \in \widetilde{K}_T$ the norms of the cocycle matrices $A_\ell \cdots A_{\ell+L}$ is not too large (and throwing away  additional sets  of bad points whose measures are summable), one can show that if Birkhoff sums are not controlled in the future, they are controlled in the past (Lemma \ref{prty}). Thus, the control required by the switchable Ratner property holds for all times. 

\smallskip

\noindent {\bf Structure of the paper.} The following sections are organized as follows. In  Section~\ref{sec:background} we review some background material, in particular we give the precise definition of locally Hamiltonian flows (see Section~\ref{sec:locHam}) and of special flows over IETs (see Section~\ref{sec:special_flows_rep}) and explain the reduction of the former to the latter (in Section~\ref{sec:reduction}). In \cref{sec:reduction} we also give the precise definition of asymmetric logarithmic singularities (see  \cref{def_LogAsym}). We then define Ratner properties, in particular the SR-property we use (see Definitions \ref{def:SWR} and \ref{def:SWR} in \cref{sec:Ratner}). Finally, in \cref{Se:RV} we recall basic properties of the Rauzy-Veech algorithm and the definition of the associated cocycles. We then describe the acceleration that we use (see Section~\ref{RVacc}) and in Section~\ref{sec:exptails} we recall the exponential tail estimates given by \cite{AGY:exp}.

\smallskip
In Section~\ref{se4}, we define the Diophantine conditions on IETs which we use in this paper, in particular we first recall the Diophantine condition under which mixing was proved in \cite{Ul:mix} and  \cite{Ra:mix} (see the Definition~\ref{def:mixingDC} of Mixing DC in Section~\ref{existencebalancedtimessec}), then we define the Ratner DC  (see Definition~\ref{def:ratnerDC} in Section~\ref{sec:RatnerDC}) under which we prove multiple mixing and the SR-Ratner property. The main result of this section is that, for a suitable choice of parameters, the Ratner DC is satisfied by a full measure set of IETs (see Proposition \ref{RDCfullmeasure}, which is proved in Section~\ref{sec:fullmeasure} using the exponential tail estimates recalled in Section~\ref{sec:exptails} and the consequences of the QB-property of compact accelerations of the Rauzy-Veech cocycle proved in Section~\ref{sec:QB}). 

\smallskip
Birkhoff sums and their growth are the main focus of Section~\ref{sec:BS}. In Section~\ref{sec:SWRRatner}, we first recall a criterium (from \cite{FK:mul} and \cite{Ka-Ku}), which allows to reduce the proof of the SR-property for some special flows to the quantative study of Birkhoff sums of the roof function. In Section~\ref{sec:derivatives} we first state the estimates on Birkhoff sums of the derivatives proved in \cite{Ul:mix, Ra:mix} under the Mixing DC and then deduce estimates in form which will be convenient for us to prove the SW-Ratner property (see Lemma \ref{prty}). Finally, in Section~\ref{sec:summability} we exploit the Ratner DC for suitable  parameters to prove that the sets $\Sigma_l$ with $l \notin \widetilde{K}_T$ (see the above outline) have summable measures (see the Summability Condition in Definition \ref{RatDC2} and Corollary~\ref{summabilitySigmas}).

\smallskip
The proof of the switchable Ratner property and of the other results presented in this introduction are all given in Section \ref{sec:proof}. First, in Section~\ref{sec:backward_forward}, we prove  Proposition \ref{forbac} which allows to control the distance of orbits of most points from the singularities either in the past or in the future. This Lemma, together with the Diophantine conditions and estimates on Birkhoff sums, is the last ingredient needed for the proof of  Theorem \ref{thm:Ratner_special_flows} (i.e.~the SR-property for special flows), which is presented in \cref{sec:deducingRatner}. The proof, which is rather technical, is preceeded by an outline at the beginning of \cref{sec:deducingRatner}. The other results in this introduction are then proved in \cref{sec:conclusions}. 

\smallskip
The Appendix contains in Section~\ref{appendix:Ratner} the proof that Ratner properties are an isomorphism invariant and in Section~\ref{appendix:IETs} for convenience of the reader, the proof of a property of balance Rauzy-Veech acceleration times that was proved in \cite{HMU:lag} and used in Section~\ref{sec:backward_forward}.

\section{Background material}\label{sec:background}

\subsection{Locally Hamiltonian flows}\label{sec:locHam}

Let $(S, \omega)$ be a two-dimensional symplectic manifold, i.e.\ $S$ is a  closed connected orientable smooth surface of genus $g\geq 1$ and $\omega$ a fixed smooth area form.  Any smooth area preserving flow on $S$ is given by a smooth closed real-valued differential $1$-form $\eta$ as follows.   Let $X$ be the vector field determined by $\eta = i_X \omega =\omega( \eta, \cdot )$ and consider the flow $(\varphi_t)_{t\in\mathbb{R}}$ on $S$ associated to $X$. Since $\eta$ is closed, the transformations $\varphi_t$, $t \in \mathbb{R}$, are  area-preserving. Conversely, every smooth area-preserving flow can be obtained in this way. 
The flow $(\varphi_t)_{t\in\mathbb{R}}$  is known as the \emph{multi-valued Hamiltonian} flow associated to $\eta$. Indeed, the flow $(\varphi_t)_{t\in\mathbb{R}}$ is \emph{locally Hamiltonian}, i.e.\ \emph{locally} one can find coordinates $(x,y)$ on $S$ in which $(\varphi_t)_{t\in\mathbb{R}}$ is given by
 the solution to the  equations $\dot{x}={\partial H}/{\partial y}$, $\dot{y}=-{\partial H}/{\partial x}$ for some smooth  real-valued Hamiltonian function $H$.  A \emph{global}  Hamiltonian $ H$ cannot be in general be defined (see \cite{NZ:flo}, Section 1.3.4), but one can think of  $(\varphi_t)_{t\in\mathbb{R}}$ as globally given by a \emph{multi-valued} Hamiltonian function.

One can define a \emph{topology} on locally Hamiltonian flows by considering perturbations of closed smooth $1$-forms by smooth closed $1$-forms. We assume that $1$-form $\eta$  is \emph{Morse}, i.e.\ it is locally the differential of a Morse function.  Thus, all zeros of $\eta$ correspond to either centers or simple saddles.   This condition is generic (in the Baire cathegory sense) in the space of perturbations of closed smooth $1$-forms by closed smooth $1$-forms.  A \emph{measure-theoretical notion of  typical} is defined as  follows by using the \emph{Katok fundamental class} (introduced by Katok in \cite{Ka:inv}, see also \cite{NZ:flo}), i.e.\ the cohomology class of the 1-form $\eta$  which defines the flow.  Let $\Sigma$ be the set of fixed points of $\eta$  and let $k$ be the cardinality of $\Sigma$. Let $\gamma_1, \dots, \gamma_n$ be a base of the relative homology $H_1(S, \Sigma, \mathbb{R})$, where $n=2g+k-1$. The image of  $\eta$ by the period map $Per $ is $Per(\eta) = (\int_{\gamma_1} \eta, \dots, \int_{\gamma_n} \eta) \in \mathbb{R}^{n}$. The pull-back $Per_* Leb$ of the Lebesgue measure class by the period map gives the desired measure class on closed $1$-forms. When we use the expression \emph{typical} below, we mean full measure with respect to this measure class.



 Let us recall that a \emph{saddle connection} is a flow trajectory from a saddle to a saddle and a \emph{saddle loop} is a saddle connection from a saddle to the same saddle (see Figure \ref{island}). Let us remark that if the flow  $(\varphi_t)_{t\in\mathbb{R}}$ given by a closed $1$-form $\eta$ has a \emph{saddle loop homologous to zero} (i.e.\ the saddle loop is a \emph{separating} curve on the surface), then the saddle loop is persistent under small pertubations (see Section~2.1 in \cite{Zo:how} or Lemma 2.4 in \cite{Ra:mix}). In particular, the set of locally Hamiltonian flows which have at least one saddle loop is open. The open sets $\mathscr{U}_{min}$ and  $\mathscr{U}_{\neg min}$ mentioned in the introduction are defined  respectively as the open set $\mathscr{U}_{\neg min}$ that contains all locally Hamiltonian flows with saddle loops homologous to zero and the interior  $\mathscr{U}_{min}$ (which one can show to be non-empty) of the complement, i.e.\ the set of locally Hamiltonian flows without  saddle loops homologous to zero\footnote{Note that saddle loops non homologous to zero (and saddle connections) vanish after arbitrarily small perturbations and neither the set of 1-forms with saddle loops non homologous to zero (or saddle connections) nor its complement is open.} (see \cite{Ra:mix} for details).
 
 Let us recall from the introduction the topological decomposition of an area-preserving flow into   \emph{minimal components} and  \emph{periodic components}.  Unless the surface is of genus one and consists of a unique component, each component is bounded by saddle connections. Periodic components are  \emph{elliptic islands} around a center (see Figure \ref{island}) or \emph{cylinders} filled by periodic orbits (see Figure \ref{cylinder}). We remark that if the flow is \emph{minimal}, fixed points can be only saddles, since if there is a center, it automatically produces an island filled by periodic orbits and hence a periodic component.   In the open set $\mathscr{U}_{min}$ with  no saddle loops homologous to zero, a typical flow has no saddle connections and this  implies minimality  by a result of  Maier \cite{Ma:tra} (or, in the language of suspension flows introduced in the next section, by the result of Keane~\cite{Ke:int} on IETs). In the open set  $\mathscr{U}_{\neg min}$, periodic components are typically bounded by saddle loops. After removing all periodic components, one is typically left with components without saddle connections on which the flow is minimal (for example, in Figure \ref{decomposition}, after removing a cylinder and two island, one is left with two minimal components one of genus one and one of genus two).

\subsection{Special flows over IETs}\label{sec:special_flows_rep}
As we mentioned in the introduction, smooth flows on higher genus surfaces can be represented as special flows over interval exchange transformations. Let us first recall the definition of IETs and of special flows.

\paragraph{\textbf{Interval exchange transformations.}}\label{IETsdefsec}
Let $I= I^{(0)}=[0,1)$ be the unit interval. An \emph{interval exchange transformation} (IET) of $d$ subintervals $T\colon I \to I$  is determined by a \emph{combinatorial datum}\footnote{We are using here the notation for IETs introduced by Marmi-Moussa-Yoccoz in \cite{MMY:coh} and subsequentely used by most recent references and lecture notes.} $\pi=(\pi_t, \pi_b)$ which consists of a pair $(\pi_t, \pi_b)$ of bijections from $\mathcal{A}$ to $\{1,\dots,d\}$, where $d \geq 2$ and $\mathcal{A}$ is a finite set with $d$ elements ($t,b$ stay here for \emph{top} and \emph{bottom} permutations) and a \emph{length vector} ${\lambda}$ which belong to the simplex  $\Delta_{d}$ of vectors ${\lambda}\in \mathbb{R}_+^\mathcal{A}$ such that $\sum_{\alpha\in\mathcal{A}}\lambda_\alpha=1$. 
Informally, the interval $I^{(0)}$ is decomposed into $d$ disjoint intervals  $I_\alpha$ of lenghts given by $\lambda_\alpha$ for  $\alpha\in\mathcal{A}$. The  \emph{interval exchange transformation} $T$ given by  $({\lambda}, \pi)$ is a piecewise isometry that rearranges the subintervals of lengths given by ${\lambda}$ in  the order determined by $\pi$, so that the intervals \emph{before} the exchange, from left to right, are $I_{\pi_t^{-1}(1)}, \dots, I_{\pi_t^{-1}(d)} $, while the order from left to right \emph{after} the exchange is  $I_{\pi_b^{-1}(1)}, \dots, I_{\pi_b^{-1}(d)} $.  Formally, $T$, for which we shall often use the notation $T=({\lambda}, \pi)$,  is the map $T\colon I^{(0)}\rightarrow I^{(0)}$ given by
$$
Tx=x-\sum_{\pi_b(\beta)<\pi_b(\alpha)}\lambda_\beta + \sum_{\pi_t(\beta)<\pi_t(\alpha)}\lambda_\beta \quad \text{ for }x\in I_\alpha^{(0)}=[l_\alpha,r_\alpha), 
$$
where $l_\alpha=\sum_{\pi_t(\beta)<\pi_t(\alpha)}\lambda_\beta$ and $r_\alpha=l_\alpha+\lambda_\alpha$ for $\alpha\in\mathcal{A}$
(the sums in the definition are by convention zero if over the empty set, e.g.~for $\alpha$ such that $\pi_t(\alpha)=1$).

We say that $T$ is \emph{minimal} if the orbit of all points  are dense.  We say that $\pi=(\pi_t,\pi_b)$ is \emph{irreducible} if $\{1,\dots,j\}$ is invariant under $\pi_b\circ \pi_t^{-1}$ only for $j=d$. Irreducibility is a necessary condition for minimality. Recall that $T$ satisfies the \emph{Keane condition} if the orbits of all \emph{discontinuities} $l_\alpha$ for $\alpha$ such that $\pi_t(\alpha)\neq 1$ are infinite and disjoint. If $T$ satisfies this condition, then $T$ is minimal \cite{Ke:int}.

\paragraph{\textbf{Special flows.}}\label{suspflowsdefsec}
Let $T\colon I \to I$ be an IET.\footnote{One can define in the same way special flows over any measure preserving transformation $T$ of a probability space $(M, \mathscr{M}, \mu)$, see e.g.\ \cite{CFS:erg}.} Let $f\in L^1 (I, dx)$ be a strictly positive function 
with $\int_{I} f(x)\, dx =1$.  Let $X_f \doteqdot \{ (x,y) \in \mathbb{R}^2 :  x \in I^{(0)}, \, 0\leq y < f(x) \}$ be the set of points below the graph of the roof function $f$ and $\mu$ be the restriction to $X_f$ of the Lebesgue measure $d x\,d y$. 
Given  $x\in I$ and $r\in \mathbb{N}^+$ we denote by  
\be\label{def:BS}
\BS{f}{r}( x)  \doteqdot \begin{cases} \sum_{i=0}^{r-1} f(T^i(x)) & \text{if} \ r>0; \\ 0 & \text{if} \ r=0;  \\ -\sum_{i=r}^{-1} f(T^{i}(x)) & \text{if} \ r<0; \end{cases}
\ee
the $r^{th}$ non-renormalized \emph{Birkhoff sum} of $f$ along the trajectory of $x$ under $T$. 
Let $t>0$. Given $x\in I$, denote by $r(x,t)$ the integer uniquely defined by
$r(x,t)\doteqdot \max \{ r\in \mathbb{N} :  \BS{f}{r}(x) < t \}$.

The \emph{special flow built over}  $T$ \emph{under the roof function f} is a one-parameter group $(\varphi_t)_{t\in \mathbb{R}}$ of $\mu$-measure preserving transformations of $X_f$  
whose action is given, for $t>0$, by 
\be \label{flowdef}
\varphi_t(x,0) = \left( T^{r(x,t)}(x), t- \BS{f}{r(x,t)}(x)\right).
\ee
For $t<0$, the action of the flow is defined as the inverse map and $\varphi_0$ is the identity.  
 The integer $r(x,t)$ gives the number of \emph{discrete iterations} of the base transformation $T$ which the point $(x,0)$ undergoes when flowing up to time $t>0$. %

\subsection{Locally Hamiltonian flows as special flows over IETs}\label{sec:reduction}

Locally Hamiltonian flows can be represented as special flows over IETs under  roof functions with logarithmic singularities. We recall now some of the details of this reduction; for more information see~\cite{Ra:mix}.


\begin{defn}\label{def_LogAsym}
Let $T\colon I\to I$ be an IET. We say that $f\colon I \to \mathbb{R} \cup \{+\infty\}$ has \emph{logarithmic singularities} and we write $f \in LogSing(T)$ if: 
\begin{enumerate}[(a)]
\item
$f$ is defined on all of $I\setminus \{l_\alpha : \alpha\in\mathcal{A}\}$;
\item
$f\in \mathcal{C}^2(I \setminus \{l_\alpha : \alpha\in\mathcal{A}\})$;
\item
$f$ is bounded away from zero;
\item\label{logD}
there exist $C_\alpha^{+}, C_\alpha^{-}\geq 0$, $\alpha\in\mathcal{A}$ such that
$$
\lim_{x \to l_\alpha^{+}} \frac{f\rq{}\rq{}(x)}{(x -l_\alpha)^{-2}} = C_\alpha^{+}, \qquad \lim_{x \to r_\alpha^{-}} \frac{f\rq{}\rq{}(x)}{(r_\alpha - x)^{-2}} = C_\alpha^{-}.
$$
\end{enumerate}
Let $C^{+} \doteqdot  \sum_\alpha C_\alpha^{+}$ and $C^{-}\doteqdot  \sum_\alpha C_\alpha^{-}$; if $C^{+} \neq C^{-}$, we say that $f$ has \emph{asymmetric logarithmic singularities} and we write $f \in AsymLogSing(T)$.
\end{defn}
We remark that it follows from~\eqref{logD} that the local behaviour of $f$ close to the singularities is $f = C_\alpha^{+} \modulo{\log (x-l_\alpha)} + \rm{o}(1)$ for $x \to l_\alpha^{+}$ and $f = C_\alpha^{-} \modulo{\log (r_\alpha-x)} + \rm{o}(1)$ for $x \to r_\alpha^{-}$, hence we speak of logarithmic singularities. We remark that we allow the possibility that some $C_\alpha^+$ or $C_\alpha^-$ are zero, so $f$ could have a finite one-sided limit at some $l_\alpha$ or $r_\alpha$, but we assume that at least one of the singularities is indeed logarithmic.  

\smallskip

Let $S'$ be a minimal component of the flow $(\varphi_t)_{t\in\mathbb{R}}$ determined by $\eta$. 
Then we can find a segment $I$ transverse to the flow containing no critical point and suitable coordinates, such that the first return map $T\colon I\to I$ of $(\varphi_t)_{t\in\mathbb{R}}$ to $I$  is an interval exchange transformation $T=(\lambda, \pi)$ exchanging $d$ intervals, where $d$ is the number of saddle points of $\eta$ restricted to the minimal component. Since $S'$ is a minimal component, $\pi$ is irreducible. 

The following remark is useful to show that if a property holds for almost every IET on $d$ intervals, it holds for the flow given by a typical $\eta$ on $S$. 
\begin{rem}\label{rk:coordinates}
One can choose the transverse segment $I$ so that  
the lenght of each interval $I_\alpha$ exchanged by $T$ appears as one of the coordinates of $Per(\eta)$, where  we recall that $Per $ denotes the period map defined in Section~\ref{sec:locHam}. Furthermore, if $S_1, \dots, S_{\kappa}$ are distinct minimal components of the flow determined by $\eta$,  one can choose transverse segments $I_1, \dots, I_{\kappa}$ on each $S_i$ and coordinates in which  the first return maps $T_i \colon I_i\to I_i$ of $(\varphi_t)_{t\in\mathbb{R}}$ to $I_i$ are interval exchanges such that the lenghts of the intervals exchanged by $T_i$, for $1\leq i \leq \kappa$, all appear as distinct coordinates of $\Theta(\eta)$.
\end{rem}


The first return time function $f$ on $I$ (i.e.\ the roof function in the special flow representation of $(\varphi_t)_{t\in\mathbb{R}}$) has logarithmic singularities. Moreover, the condition that the singularities of $f$ are asymmetric 
 is open and dense: more precisely, one can show (see~\cite{Ra:mix}) that there exists an open and dense subset $\mathcal{U}'_{\neg min} \subset \mathcal{U}_{\neg min}$ such that all minimal components of locally Hamiltonian flow in  $\mathcal{U}'_{\neg min} \subset \mathcal{U}_{\neg min}$ can be represented as special flows under a roof in $AsymLogSing(T)$.

\subsection{Ratner properties of parabolic divergence} \label{sec:Ratner}
We recall now the definition of the switchable Ratner property. We state a more general definition which also includes the weak switchable Ratner property and then comment on the differences  with the original Ratner property (see Remark \ref{rk:Ratner_hist} below).  
The definition is rather technical and is followed by an intuitive explanation of its heuristic meaning (see Remark \ref{rk:Ratner_expl}).

\smallskip
Let $(X,d)$ be a $\sigma$-compact metric space, $\mathcal{B}$ the $\sigma$-algebra of Borel subsets of $X$, $\mu$ a Borel probability measure on $(X,d)$.  Let $(T_t)_{t\in\mathbb{R}}$ be an ergodic flow acting on $(X,\mathcal{B},\mu)$.
\begin{defn}[SWR-Property, see \cite{FK:mul}]\label{def:SWR}
Fix a compact set $P\subset \mathbb{R}\setminus\{0\}$ and $t_0>0$. We say that the flow $(T_t)_{t\in\mathbb{R}}$ has {\em $sR(t_0,P)$-property} if:

\smallskip
\noindent for every $\vep>0$ and $N\in\N$ there exist $\kappa=\kappa(\vep)$, $\delta=\delta(\vep,N)$ and a set $Z=Z(\vep,N)$ with $\mu(Z)>1-\vep$, such that:

\smallskip
\noindent for every $x,y\in Z$ with $d(x,y)<\delta$ and $x$ not in the orbit of $y$, there exist $p=p(x,y)\in P$ and $M=M(x,y),$ $L=L(x,y)\geq N$ such that $\frac{L}{M}\geq \kappa$ and  
at least one of the following holds:
\begin{enumerate}[(i)]
\item
$\frac{1}{L}\left|\{n\in [M,M+L] : d(T_{nt_0}(x),T_{nt_0+p}(y))<\vep\}\right| >1-\vep$,
\item
$\frac{1}{L}\left|\{n\in [M,M+L] : d(T_{(-n)t_0}(x),T_{(-n)t_0+p}(y))<\vep\}\right| >1-\vep$.
\end{enumerate}
We say that $(T_t)_{t\in\mathbb{R}}$ has the \emph{switchable weak Ratner property}, or, for short, the {\em SWR-property} (with the set $P$) if 
$\{t_0>0: (T_t)_{t\in\mathbb{R}} \text{ has }sR(t_0,P)\text{-property}\}$ is uncountable.
\end{defn}

\begin{defn}[SWR-Property, see \cite{FK:mul}]\label{def:SR}
 We say that the flow $(T_t)_{t\in\mathbb{R}}$ has \emph{switchable Ratner property}, or, for short, the {\em SR-property}, if 
$(T_t)_{t\in\mathbb{R}}$ has  the { SWR-property} with the set $P=\{1,-1\}$.
\end{defn}

\begin{rem}\label{rk:Ratner_expl}
Intuitively, the SR-property (or the SWR-property) mean that, for a large set of choices of nearby initial points (i.e. pairs of points in the set $Z$ which are $\delta$ close), the orbits of the  two points \emph{either in the past, or in the future} (according to wheather (i) or (ii) hold), diverge and then, after some arbitrarily large time ($M t_0$ or $-M t_0$) \emph{realign}, so that $T_{nt_0}(x)$ is close to a a \emph{shifted} point $T_{nt_0+p}(y)$ of the orbit of $y$ (where $p\in P$ denotes the temporal \emph{shift}), and the two orbits then \emph{stay close} for a \emph{fixed proportion} $\kappa$ of the time $M$. One can see that this type of phenomenon is possible only for parabolic systems, in which orbits of nearby points diverge with polynomial or subpolynomial speed.  
\end{rem}

\begin{rem}\label{rk:Ratner_hist}
The original definition of the Ratner property  differs from Definition \ref{def:SR} only in that for all $x,y\in Z$ (i) has to be satisfied. The possibility of \emph{choosing}, for a given pair of points, whether (i) or (ii) holds, is the reason why the property was called \emph{switchable} by B.~Fayad and the first author in \cite{FK:mul}: one can \emph{switch} between either considering the \emph{future} trajectories of the points (if (i) holds), or the \emph{past}  (if (ii) holds).

Let us also stress that in the Ratner property (some times also called \emph{two-point Ratner property})  $P=\{1,-1\}$. The generalizations given by K. Fr\k{a}czek and M. Lema{\'{n}}czyk in \cite{FL} and \cite{FL2} mentioned in the introduction, i.e. the \emph{finite Ratner property} and the \emph{weak Ratner property}, amounted to allowing $P$ to be any finite set or respectively any compact set $P\subset \mathbb{R}\setminus\{0\}$. Thus the weak Ratner property \cite{FL2} is analogous to Definition \ref{def:SWR} but with the restriction that for all $x,y\in Z$ (i) has to be satisfied. 
\end{rem} 

All the variants of the Ratner properties are defined so that the results in the following Theorem \ref{thm:Ratmix} and Remark \ref{rk:joinings}  hold. 
\begin{thm}[\cite{FK:mul}]\label{thm:Ratmix}
Let $(X,d)$ be a $\sigma$-compact metric space, $\mathcal{B}$ the $\sigma$-algebra of Borel subsets of $X$, $\mu$ a Borel probability measure on $(X,d)$. Let $(T_t)_{t\in\mathbb{R}}$ be a  flow acting on $(X,\mathcal{B},\mu)$. If $(T_t)_{t\in\mathbb{R}}$ is mixing and has the SWR-property, then it is mixing of all orders. 
\end{thm}
\begin{rem}\label{rk:joinings}
More precisely, one can show that if $(T_t)_{t\in\mathbb{R}}$  has the SWR-property (and hence in particular if it has the SR-property),  then it has a property the \emph{finite extension of joinings property} (shortened as FEJ property), \cite{FK:mul}, which is a rigidity property that restricts the type of self-joinings that $(T_t)_{t\in\mathbb{R}}$  can have \cite{Ryz-Tho,FL}. 
We refer the reader to \cite{Gla, Ryz-Tho} for the definition of joinings and FEJ. Furthemore, it is well known that if $(T_t)_{t\in\mathbb{R}}$  is mixing \emph{and} has the FEJ property, then it is automatically mixing of all orders, \cite{Ryz-Tho}.
\end{rem}

One can show that the SR-property, as well as other Ratner properties (with the set $P$ being finite), are an isomorphism invariant. We include the proof of this fact in the Appendix \ref{appendix:Ratner} (see Lemma \ref{lemma:iso_inv}). Let us remark that it is not known whether the weak Ratner property is an isomorphism invariant.

\subsection{Rauzy-Veech induction} \label{Se:RV}

The Rauzy-Veech algorithm and the associated Rauzy-Veech cocycle were originally introduced and developed in the works by Rauzy  and Veech \cite{Ra:ech, Ve:int, Ve:gau} and proved since then to be a powerful tool to study IETs. 
If $T=({\lambda}, \pi)$  satisfies  the Keane condition recalled in Section \ref{IETsdefsec}, which holds  for a.e.~IET by \cite{Ke:int},  the Rauzy-Veech algorithm produces a sequence of IETs  which are induced maps of $T$ onto a sequence of nested subintervals contained in $I^{(0)}$. The intervals are chosen so that the induced maps are again  IETs of the same number $d$ of exchanged intervals. 
For the precise definition of the algorithm, we refer e.g.\ to the recent lecture notes by Yoccoz \cite{Yo:con} or Viana \cite{Vi:IET}. We recall here only some basic definitions and properties needed in the rest of this paper. 

Let us denote by $|\cdot |$  the vector norm $|{\lambda}|=\sum_{\alpha\in \cA}\lambda_\alpha$. If $I'\subset I^{(0)}$ is the subinterval associated to one step of the algorithm and $T'$ is the corresponding induced IET, the \emph{Rauzy-Veech map}  $\R $ associates to $T$ the IET $\R (T)$  obtained by renormalizing $T'$ by $\Leb{I'}$ so that the renormalized IET is again defined on an unit interval.
The natural domain of definition of the map $\R$ is a full Lebesgue measure subset of the space $X\doteqdot  \Delta_{d} \times \R(\pi)$,  where $\R(\pi) $ is the \emph{Rauzy class} of $\pi$ (i.e.\ the subset  of all pairs of bijections $\pi'=(\pi'_t,\pi'_b)$ from $\cA$ to $\{1,\dots,d\}$ which appear as combinatorial data  of an IET $T'=({\lambda}', \pi')$ in the orbit under $\R$ of some IET $({\lambda}', \pi)$  with initial pair of bijections $\pi=(\pi_t,\pi_b)$). We will denote by $\Delta_\pi\doteqdot  \Delta_d \times \{\pi\}$ the copy of the simplex indexed by $\pi$. 

Veech proved in \cite{Ve:gau} that $\R$ admits an invariant measure $\mu= \mu_{\R}$ (we will usually simply write $\mu$ unless we want to stress is the invariant measure for $\R$ and not any of its accelerations defined below) which 
 is absolutely continuous with respect to Lebesgue measure, but infinite. 
Zorich showed in \cite{Zo:fin} that one can induce the map 
$\R $ in order to obtain an accelerated map $\Z$, which we call \emph{Zorich map}, that admits a  \emph{finite} invariant measure $\mu_{\Z}$. Both these measures have an absolutely continuous density with respect to the restriction of the Lebesgue measure on $\mathbb{R}^{d}$ to each copy $\Delta_\pi$ of the simplex $\Delta_{d}$, which we will denote by $Leb_{X}$. 
Let us also recall that both $\R$ and its acceleration $\Z$ are \emph{ergodic} 
 with respect to $\mu= \mu_{\R}$ and $\mu_{\Z}$ respectively \cite{Ve:gau}. 
 

\paragraph{Rauzy-Veech (lengths) cocycle.}   We will now recall the definition of the cocycle associated by the algorithm to the map $\R$.
For each $T= ({\lambda},\pi)$ for which $\R (T) $ is defined, we define the matrix $B=B(T)\in SL(d, \mathbb{Z})$  such that ${\lambda} = B \cdot  {\lambda}'$, where ${\lambda}' $ satisfies $\R (T)  = ({\lambda}'/|{\lambda}'|, \pi')$. In particular, $|{\lambda}'|$ is the length $Leb(I')$ of the inducing interval $I'$  on which $\R (T)$ is defined.   
 The map  $B^{-1}\colon X \rightarrow SL(d,\mathbb{Z})$ is a cocycle over $\R$, known as the \emph{Rauzy-Veech cocycle}, that describes how the lengths transform.  If $T = ({\lambda}, {\pi})$ satisfies the Keane condition so that  its Rauzy-Veech orbit $(\R^n (T))_{n\in \mathbb{N}}$ is infinite, we denote by  $T^{(n)}\doteqdot  \R^n (T) $ the IET  obtained at the $n^{th}$ step of Rauzy-Veech algorithm and by $(I^{(n)})_{n\in\mathbb{N}}$ the sequence of nested subintervals so that  $T^{(n)}$ is the first return map of $T$ to the interval $I^{(n)} \subset I^{(0)}$.  By construction, $T^{(n)}$ is again an IET of $d$ intervals; let $\pi^{(n)} \in \R(\pi)$ and $\lambda^{(n)} \in \Delta$ be the sequence of combinatorial and lengths data such that   $T^{(n)}=({\lambda}^{(n)}/|\lambda^{(n)}|,\pi^{(n)})$, where $| {\lambda}^{(n)}|= Leb(I^{(n)})$. 
If we define $\RL{n} = \RL{n} (T) \doteqdot  B(\R^n (T))$ and $\RLp{n}= \RLp{n}(T) \doteqdot \RL{0} \cdot \ldots \cdot \RL{n-1}$  and iterating the lengths relation, we get   
\begin{equation} \label{lengthsrelation}
 {\lambda} ^{(n)} = \left( \RLp{n} \right) ^{-1 }  {\lambda},  \text{ where  }  \R ^n (T) \doteqdot  \left(\frac{{\lambda}^{(n)}}{|\lambda^{(n)}|}, \pi^{(n)}\right) .
\end{equation}
For more general products with $m<n$, we use the notation $B^{(m,n)} \doteqdot \RL{m}\cdot   \RL{m+1}  \cdot\ldots\cdot \RL{n-1}$. The entries of $B^{(m,n)}$ have the following \emph{dynamical interpretation}: $B^{(m,n)}_{\alpha \beta}$ is equal to the number of visits of the orbit of any point $x\in I^{(n)}_\beta$ to the interval $I^{(m)}_\alpha$ under the orbit of $T^{(m)}$ up to its first return to $I^{(n)}$. In particular, $ \sum_{\alpha\in\cA} \RLp{n}_{\alpha\beta}$ gives the first return time of   $x\in I^{(n)}_\beta$ to $I^{(n)}$ under $T$.


\subsubsection*{Rohlin towers and  heights cocycle.}\label{towers} 
The action of the initial IET $T=T^{(0)}$ can be seen in terms of Rohlin towers over $T^{(n)}= \R^n(T)$ 
 as follows. Let ${h}^{(n)} \in \mathbb{N}^\cA$ be the vector such that $h^{(n)}_\beta$ gives the return time of any $x\in I^{(n)}_\beta$ to $I^{(n)}$, $\beta\in\mathcal{A}$. 
 By the above dynamical interpretation,   $h^{(n)}_\beta= \sum_{\alpha\in\cA} \RLp{n}_{\alpha\beta}$ is the norm of the $\beta^{th}$ column of $\RLp{n}$. 
Define the sets 
\begin{equation}\label{def:towers}
Z^{(n)}_\alpha \doteqdot \bigcup _{i=0}^{h^{(n)}_\alpha-1} T^i I^{(n)}_\alpha,\ \alpha\in\mathcal{A}.
\end{equation}
Each $Z^{(n)}_\alpha$ can be visualized as a tower over $I^{(n)}_\alpha\subset I^{(n)}$, of height  $h^{(n)}_\alpha$, 
whose floors are $T^i I^{(n)}_\alpha$.  Under the action of $T$ every floor but the top one, i.e.~if $0\leq i <h^{(n)}_\alpha-1$, moves one step up, while the image by $T$ of the last floor, corresponding to $i=h^{(n)}_\alpha-1$,   is $ T^{(n)} I^{(n)}_\alpha$. 

The \emph{height vector} ${h}^{(n)}$ which describes return time and heights of the Rohlin towers at step $n$ of the induction can be obtained by applying the \emph{dual cocycle} $B^T$, that we will call \emph{lenghts cocycle}, i.e., if ${h}^{(0)}$ is the column vector with all entries equal to $1$,  
\be\label{heightsrelation}
{h}^{(n)} = (B^T)^{(n)} {h}^{(0)}.
\ee
Let us denote by $\phi^{(n)}$ be the partition of $I^{(0)}$ into floors of step $n$, i.e.~intervals of the form $T^{i} I^{(n)}_\alpha$. When $T$ satisfies the Keane condition, the partitions $\phi^{(n)}$ converge as $n$ tends to infinity to the trivial partitions into points (see for example \cite{Yo:con, Vi:IET}).

\subsubsection*{Natural extension of the Rauzy-Veech induction.}
The \emph{natural extension} $\hat{\R}$ of the map $\R$ is an invertible map defined on a domain $\hat{X}$ (which admits a geometric interpretation in terms of the space of zippered rectangles,  see for example  \cite{Yo:con, Vi:IET}) 
such that there exists a projection $p \colon \hat{X} \rightarrow X$ for which $p \hat{\R} = \R p$. More precisely, for any $\pi$ in the Rauzy class, let $\Theta_\pi \subset \mathbb{R}_+^d$ be the set of vectors ${\tau}$ such that
$$
\sum_{\pi_t(\alpha)<j}\tau_j >0 \text{ and }\sum_{\pi_b(\alpha)<j}\tau_j < 0\text{ for } 1\leq j\leq d-1.
$$
Vectors in  $\Theta_\pi$ are called \emph{suspension data}. 
Points in $\hat{X}$ are triples 
$\hat{T}=(\tau, \lambda, \pi)$ such that 
\be\label{area_condition}
\sum_{\alpha}\lambda_\alpha h_\alpha =1, \text{ where }  h_\alpha \doteqdot  \sum_{\substack{ \pi_t (\beta)<\alpha,\\\pi_b (\beta)> \alpha } } \tau_\beta -  \sum_{\substack{ \pi_t (\beta)>\alpha,\\\pi_b (\beta)< \alpha } } \tau_\beta.
\ee
To each such triple  $\hat{T}=( \tau, \pi, \lambda)$  one can associate  a geometric object known as \emph{zippered rectangle}. We refer to  \cite{Yo:con, Vi:IET} for details.  The vector $h= (h_\alpha)_{\alpha \in \mathcal{A}}$ gives the heights of the rectangles and the vector $\lambda = (\lambda_\alpha)$ gives their lenghts (while $\tau$ contain information about how to \emph{zip} the vertical sides of the rectangles together). Thus,  
the above normalization condition \eqref{area_condition} guarantees that the associated zippered rectangle has area one. 

The projection $p$ is defined by  $p(\tau,\pi, \lambda) = (\pi, \lambda)$. 
The natural extension $\hat{\R}$ preserves a natural invariant measure $\hatmu$, whose push-forward $p_*{\hatmu}$  by the projection $p$ (i.e.\ the measure such that $p_*{\hatmu}(E)= {\hatmu}(p^{-1}E)$ for any measurable set on $X$) equals $\muR$.
 
Both  cocycles $B^{-1}$ and $B^T$ can be extended to  cocycles over $(\hat{X},  {\mu}_{\hat{\R}},\hat{\R})$ (for which we will use the same notation $B^{-1}$, $B^T$) by setting $B(\tau, \lambda, \pi) \doteqdot  B(\lambda, \pi)$ for any $(\tau, \lambda, \pi) \in \hat{X}$, i.e.\ the extended cocycles are constant on the fibers of $p$.

\subsubsection*{Cylinder sets.}
Let us define symbolic cylinders for the Rauzy-Veech map $\R$ and for its natural extension $\hat{\R}$. 
We will say that a finite sequence of matrices $B_0, \dots, B_{n}$ is a \emph{sequence of Rauzy-Veeech matrices} or, more precisely, a sequence of Rauzy-Veech matrices   \emph{starting at $\pi$} if there exists $T = ({\lambda}, \pi)$ for which $B_i= B_i( T) $ for all $0\leq i \leq n$. We will say that a matrix $B$ is a \emph{Rauzy-Veech product} (\emph{at} $\pi$) if  $B = B_0\cdot \ldots\cdot B_{n} $ where  $B_0, \dots, B_{n}$ is a sequence of Rauzy Veech matrices (starting at $\pi$). Furthermore, we will say that two Rauzy-Veech products $C,D$ \emph{can be concatenated} if $CD$ is also a Rauzy-Veech product.

We will say that $\Delta_B \subset \Delta_\pi$ is a \emph{Rauzy-Veech cylinder} if  $B$ is a Rauzy-Veech product at $\pi$ and
\bes
\Delta_{B } = \left\{ (\lambda',\pi) : \lambda'=  \frac{ B {\lambda} }{|B  {\lambda}|} , {\lambda} \in \Delta_d \right\} \subset \Delta_\pi.
\ees
One can see that  any $T=(\lambda, \pi)$ where $\lambda \in \Delta_B $  satisfies $ B_i( T) = B_i  $ for all $0\leq i < n$. 
Thus, $\Delta_B$  is a cylinder set for the symbolic coding of Rauzy-Veech induction given by the sequence of Rauzy-Veech matrices. 

One can analogously define symbolic cylinders for the natural extension $\hat{\R}$.  Let us first define the set $\Theta_C $ associated to a Rauzy-Veech product $C=C_1\cdot\ldots \cdot C_n$ starting at $\pi$ and ending at $\pi'$  to be the subset of suspension data $\tau \in \Theta_{\pi'}$ implicitely defined by
\bes
B^T \Theta_C = \Theta_{\pi}.
\ees
In other words, if $(\tau, \pi', \lambda ) \in \hat{X}$ belongs to  $\Theta_{\pi} \times \Delta_C  $, the past $n$ Rauzy-Veech matrices are prescribed by $C$, i.e. for $-n \leq i \leq -1$ we have $B_{i}(\tau, \pi', \lambda) =  C_{i+n+1}$. 


Cylinders in the space $\hat{X}$ have then the form $\Theta_C \times \Delta_D  \cap \hat{X}$,  
where $C$ and $D$ are Rauzy-Veech products that can be concatenated. Let us remark that vectors $\tau \in \Theta_\pi$ are \emph{not} normalized, while points $(\tau, \pi, \lambda )$ in $\hat{X}$ are such that the normalization condition \eqref{area_condition} holds. Thus, $\Delta_D \times \Theta_C$ is not contained in $\hat{X}$ and to obtain a cylinder for $\hat{R}$ one needs to intersect it with $\hat{X}$. We will use the notation
\bes
(\Theta_C \times \Delta_D )^{(1)}\doteqdot ( \Theta_C  \times \Delta_D ) \cap \hat{X}  
\ees
for cylinders to avoid explicitly writing the intersection with $\hat{X}$. 

It follows from the definitions that if $C= C_1\cdot\ldots\cdot C_n$ and $D=D_0 \cdot\ldots\cdot D_m$ where  $C_1, \dots, C_n$, $D_0, D_1,\dots, D_m$ is a sequence of Rauzy-Veech matrices, $(\pi, \lambda, \tau) \in   \Theta_C \times \Delta_D$ if and only if the cocycle matrices $B_i= B_i (\pi, \lambda, \tau) $ as $i$ ranges from $-m$ to $n$ are in order $C_1, \dots, C_n, D_0,D_1 \dots D_m$ (in other   words, $B_i=D_{i}$ for $0\leq i \leq m$ and $B_{i}=C_{i+n+1}$ for $-n\leq i \leq -1$), thus $(\Theta_C \times \Delta_D )^{(1)}$ are indeed symbolic cylinders for the natural extension.

Remark also that by definition we have   
\be\label{cylindersprop}
 (\Theta_C  \times \Delta_{\pi'})^{(1)}  = \hat{\R}^{-n} ( \Theta_{\pi}  \times \Delta_C )^{(1)}.
\ee 




\subsubsection*{Hilbert distance and projective diameter.}
Let us say that a matrix $C$ is \emph{positive} (resp. non negative) and let us write $C>0$ (resp. $C\geq 0$)
if all its entries  are strictly positive (resp. non negative). 

Consider on the simplex $\Delta_{d} $ the \emph{Hilbert distance} $d_H$, defined as follows.
\bes
d_H(\lambda, \lambda') \doteqdot \log \frac{\max_{i=1,\dots , d}\frac{\lambda_i}{\lambda'_i}}{\min_{i=1 , \dots ,d}\frac{\lambda_i}{\lambda'_i}}.
\ees
One can see that for any negative $d \times d$ matrix $A\geq 0$, the associated projective transformation $\psi_A(\lambda) = {A \lambda }/{|A \lambda |}$ of $\Delta_{d}$ is a contraction of the Hilbert distance, i.e.  $d_H(\widetilde{A} \lambda, \widetilde{A} \lambda' ) \leq d_H(\lambda, \lambda')$ for any $\lambda, \lambda' \in \Delta_d$. Furthermore, if $A>0$, then  it is a strict contraction.

Let us define the \emph{projective diamater} $diam_H(A)$ of $A\geq 0$ as the diameter with respect to $d_H$ of the image of $\psi_A$, namely
\be \label{diam_def}
diam_H(A) \doteqdot \sup_{\lambda, \lambda' \in \Delta_d} d_H( \psi_A( \lambda), \psi_A (\lambda' ) ) = 
\sup_{\lambda, \lambda' \in \Delta_d} d_H( A \lambda, A \lambda' ) ,
 \ee
 where the last equality follows from the definition of $d_H$.    

\begin{rem}\label{rk:finitediam} 
 We remark that $diam_H (A )$ is finite exactly when $A$ is a positive matrix, since $A>0$ is equivalent to $ \psi_A \left( \Delta_{d} \right)$ being pre-compact, which means that its closure is contained in $\Delta_{d}$ and its diameter with respect to $d_H$ is finite.
\end{rem}

\begin{rem}\label{rk:diam}
Notice that if, given two positive matrices $A, B$, if $\Delta_B \subset \Delta_A$ then clearly from the definition $diam_H(B) \leq diam_H(A)$. In particular,  since by definition of cylinders $\Delta_{AB} \subset \Delta_A$, we have that $diam_H(AB) \leq diam_H(A)$. Furthermore, since $\Delta_{AB}$ is the image of $\Delta_B$ by the  projective transformation $\psi_A(\lambda) = {A \lambda }/{|A \lambda |}$ which is a projective contraction, we also have that $diam_H(AB) \leq diam_H(B)$. 
\end{rem}

\subsection{Rauzy-Veech accelerations}\label{RVacc}
Let $T^{(n)}\doteqdot \R^n(T)$, $n\geq 0$, be the Rauzy-Veech orbit of $T=T^{(0)}$ satisfying the Keane condition. 

\subsubsection*{Accelerations of the Rauzy-Veech map.}
Given an increasing sequence $\{n_\ell\}_{\ell\in\N}$ of natural numbers, we can consider the corresponding \emph{acceleration} $\tilde{\R}$ of the Rauzy-Veech map, defined on $\{T^{(n_\ell)} : \ell\in\N\}$ by $\tilde{\R}(T^{(n_\ell)})=T^{(n_{\ell+1})}$, $\ell\in\N$.
 In other words, $\tilde{\R}^{(\ell)}(T)=\R^{(n_\ell)}(T)$, $\ell\in\N$. We will refer to $\{n_\ell\}_{\ell\in\N}$ as a \emph{sequence of induction times} for $T$. 

The sequence $\{n_\ell\}_{\ell\in\N}$ can be chosen e.g.\ by considering Poincar{\'e} first return map of  the Rauzy-Veech induction as follows. Fix a subset $Y \subset X = \Delta_d \times \mathcal{R}(\pi)$ of positive measure. By the ergodicity of $\R$, for $Leb_{\Delta_d}$ almost every ${\lambda}$ and for $\pi'\in\R(\pi)$, the corresponding IET $T=({\lambda},\pi')$ visits $Y$ under $\R$ infinitely often and this gives us immediately a sequence $\{n_\ell\}_{\ell\in\N}$ for a typical $T$. The corresponding acceleration of $\R$ in this case will be denoted by $\R_Y$ and is a map $\R_Y : Y \to Y$ defined a.e. on $Y$. 

\begin{rem}\label{rk:formJac}
Let us assume that $Y = \Delta_B$ is a Rauzy-Veech cylinder. One can see that $\R_Y $ is piecewise defined and locally given by maps of the form ${\lambda} \mapsto D {\lambda}/ |D {\lambda} |$ where $D$ is a matrix of the form $D = B C$ for some non-negative $C \in SL(d, \mathbb{Z})$.  The Jacobian of a map of this form  is $J_D ({\lambda}) =  |D {\lambda} |^{-d}$  (see Veech~\cite{Ve:int}, Proposition 5.2).
\end{rem}

Given an acceleration $\R_Y$ one can correspondingly define a cocycle $A_{Y}$ over $\R_Y$ obtained by accelerating the Rauzy-Veech cocycle. This cocycle is a.e. defined by setting
$A_{Y} (T) \doteqdot  B^{(n_Y(T))}(T)$, where  $n_Y(T)$ is the first return time of $T$ to $Y$. 





\subsubsection*{Accelerations of the Rauzy-Veech natural extension.}
Similarly, one can accelerate the natural extension  $\hat{\R}$ of $\R$. Given $\hat{T}\in\hat{X}$ and an increasing sequence $\{n_\ell\}_{\ell\in\N}$ of natural numbers,\footnote{If we want to accelerate also the backward iterations of $\hat{\R}$, we need an increasing sequence of integers, indexed by $\mathbb{Z}$.} we define $\tilde{\hat{\R}}$ on $\{\hat{T}^{(n_\ell)} : \ell\in\N\}$ by $\tilde{\hat{\R}}(\hat{T}^{(n_\ell)})=\hat{T}^{(n_{\ell+1})}$, $\ell\in\N$. In other words, $\tilde{\hat{\R}}^{(\ell)}(\hat{T})=\hat{\R}^{(n_\ell)}(\hat{T})$. The corresponding accelerated cocycle $A=A(\hat{T})$ is given by $A^{(\ell,\ell+1)}(\hat{T})\doteqdot B^{(n_\ell,n_{\ell+1})}(\hat{T})$, $\ell\in\N$, where $B$ is the cocycle associated to~$\hat{\R}$.

As before, the sequence $\{n_\ell\}_{\ell\in\N}$ can be chosen e.g.\ by considering Poincar{\'e} first return map of natural extension of the Rauzy-Veech induction: given a subset $\hat{Y}\subset \hat{X}$ of positive measure, $\{n_\ell\}_{\ell\in\N}$ is defined as the sequence of visits of $\hat{T}$  to $\hat{Y}$ under $\hat{\R}$ (and it is well defined for a typical $\hat{T}$). The corresponding acceleration will be denoted by $\hat{\R}_{\hat{Y}}$ and the corresponding accelerated cocycle will be denoted by $A_{\hat{Y}}$ and is explicitly given by
\begin{equation}\label{acccocycle_def}
A_{\hat{Y}} (\hat{T}) \doteqdot  B^{\left(n_{\hat{Y}}(\hat{T})\right)}(\hat{T}),
\end{equation}
where $n_{\hat{Y}}(\hat{T})$ is  the first return time of $\hat{T}$ to $\hat{Y}$ (this is well defined for  ${\mu}_{\hat{\R}}$ almost every $T \in \hat{Y}$).

\begin{defn}
Let us say that an acceleration $\hat{\R}_{\hat{Y}}$ of the natural extension $\hat{\R}$ is a \emph{cylindrical acceleration} if
 $\hat{Y}$ is a \emph{finite union of  Rauzy-Veech cylinders} for the natural extension $\hat{\R}$.
\end{defn}

\begin{rem}\label{rk:sequence_indep_lift}
Let $T=( {\lambda}, \pi)$ be an IET and  $(\tau, {\lambda}, \pi )$ any of its lifts in $p^{-1}(\pi, {\lambda}) \in \hat{X}$. If $\hat{\R}_{\hat{Y}}$ is a cylindrical acceleration, the sequence $\{n_\ell\}_{\ell\in\N}$  of first returns to $\hat{Y}$ of the orbit of  $(\pi, {\lambda}, {\tau})$ under $\hat{\R}$ depends on $T$ only, apart from possibly finitely many initial terms. More precisely, if each cylinder $\Delta_{F_i} \times \Theta_{E_i} $  in $\hat{Y}$ is such that $E_i$ is product of at most $\ell_0$  Rauzy-Veech matrices, then $n_\ell$ for $\ell \geq \ell_0$ is uniquely determined by $T$. Indeed the sequence  $B_n = B(\hat{\R}^n (\pi, {\lambda}, {\tau}) )$ of Rauzy-Veech matrices for  $n \in \mathbb{N}$ depends on 
$T$ only since by definition of extended Rauzy-Veech cocycle and natural extension $B(\hat{\R}^n (\pi, {\lambda}, {\tau}) )= B({\R}^n (\pi, {\lambda}) )$. Remark now that to decide whether $\hat{\R}^{n_\ell}(\pi, {\lambda}, {\tau})$ belongs to $\hat{Y}$, by definition of $\ell_0$ as maximal cylindrical lenght, it is enough to know the matrices $B_{n_\ell - \ell_0}, \cdots B_{n_\ell + \ell_0}$ and if $\ell \geq \ell_0$, since $n_{\ell}\geq \ell \geq \ell_0$, these matrices are uniquely determined by $T$. 
\end{rem}

\subsection{Positivity, balance, pre-compactness and exponential tails}\label{sec:exptails}
Let $T^{(n)}\doteqdot \R^n(T)$, $n\geq 0$, be the Rauzy-Veech orbit of $T=T^{(0)}$ satisfying the Keane condition. 

\begin{defn}[Positive times]
Sequence $\{n_\ell\}_{\ell \in \mathbb{N}}$ is called a \emph{positive sequence of induction times} for $T$ if for any $\ell \in \mathbb{N}$ all entries of $B^{(n_\ell, n_{\ell+1})}=B^{(n_\ell, n_{\ell+1})}(T)$ are strictly positive: $B^{(n_\ell, n_{\ell+1})}>0$.
\end{defn}

\begin{rem}\label{positvegrowth}
It follows from \eqref{lengthsrelation} that along a sequence $\{n_\ell\}_{\ell\in\N}$ of positive times, we have ${\lambda}^{(n_{\ell})} \geq d^{k}  \lambda^{(n_{\ell+k})}$, $\ell\geq 1$.
\end{rem}


\begin{defn}[Balanced times]\label{balti}
If, for some $\ell\geq 1$ and $\nu>0$, we have
\be
\frac{1}{\nu} \leq \frac{\lambda_\alpha^{(n_\ell)}}{\lambda_\beta^{(n_\ell)}} \leq \nu, \qquad
\frac{1}{\nu} \leq \frac{h_\alpha^{(n_\ell)}}{h_\beta^{(n_\ell)}} \leq \nu, \quad  \forall \alpha,\beta\in\cA,
\ee
we say that $n_\ell$ is $\nu$-balanced. If $n_\ell$ is $\nu$-balanced for each $\ell\geq 1$ (with the same $\nu>0$), we say that $\{n_\ell\}_{\ell\in \mathbb{N}}$ is a $\nu$-\emph{balanced} (or simply a \emph{balanced}) \emph{sequence of induction times} for $T$.
\end{defn}
\begin{rem}\label{balancedcomparisons}
 If $n_\ell$ is $\nu$-balanced then lengths and heights of the Rohlin towers
are approximately of the same size, i.e.\ 
$$
 \frac{1}{d\nu }{\lambda^{(n_\ell)} } \leq  \lambda^{(n_\ell)}_\alpha\leq \lambda^{(n_\ell)}\text{ and }
\frac{1}{\nu\lambda^{(n_\ell)} }  \leq  h^{(n_\ell)}_\alpha \leq \frac{\nu}{ \lambda^{(n_\ell)} } \text{ for each }\alpha\in\cA. %
$$
\end{rem}

\subsubsection*{Pre-compactness, cylindrical accelerations and bounded distorsion.}
We will consider a special class of accelerations, which are cylindrical and \emph{precompact} in the following sense.
\begin{defn}
We say that an acceleration ${\R}_{{Y}}$ of $\R$ is \emph{pre-compact} whenever ${Y}$ is pre-compact in $X$ and we say that an acceleration $\hat{\R}_{\hat{Y}}$  of the natural extension $\hat{\R}$ is \emph{pre-compact} whenever $\hat{Y}$ is pre-compact in $\hat{X}$.
\end{defn}

A Rauzy-Veech cylinder $\Delta_C$ 
is pre-compact in $X$ 
if and only if $C$ is a \emph{positive} matrix. 
Thus,  a \emph{cylindrical acceleration} is \emph{precompact} if and only if each of the cylinders in the finite union defining the acceleration is given by a positive Rauzy-Veech product. 
Similarly,  there are simple conditions on  Rauzy-Veech product matrices $C, D$ that guarantee that a cylinder $\Delta_D \times \Theta_C$ for $\hat{\R}$ is pre-compact (see for example the notion of strongly positive matrix in \cite{AGY:exp}).

\begin{rem}[see, e.g., \cite{Ul:mix}]\label{rk:automcylin}
If $\hat{\R}_{\hat{Y}}$ is a pre-compact acceleration then for any $\hat{T}$, for which the corresponding sequence of induction times $\{n_\ell\}_{\ell\in\N}$ is well-defined, $\{n_\ell\}_{\ell\in\N}$ is automatically balanced. Furthermore, if the pre-compact acceleration is \emph{cylindrical} (i.e.\ the inducing set $\hat{Y}$ is a finite union of pre-compact cylinders for $\hat{\R}$), then each resulting $\{n_\ell\}_{\ell\in\N}$ is automatically positive for the corresponding $\hat{T}$.
\end{rem}

One of the reason why pre-compact accelerations are important is that they enjoy the \emph{bounded distorsion property}. More precisely, if we set $Y=\Delta_C$ where $C$ is a \emph{positive} Rauzy-Veech product and consider the corresponding cylindrical acceleration, $\R_Y$ is \emph{strictly expanding} and has \emph{bounded distorsion}, i.e.\ there exists a constant $\nu_Y$ such for any inverse branch of $\R_Y$, which is a map of the form ${\lambda} \mapsto D{\lambda}/ |D {\lambda} |$, where $D$ is a matrix of the form described in Remark \ref{rk:formJac},  the Jacobian $J_D ({\lambda})  $ of the inverse branch satisfies 
$| {J_D (\lambda)} |/|{ J_D (\lambda')}| \leq \nu_Y$ for all  $\lambda, \lambda'  \in Y$. 
This property follows from a remark by Veech (see \cite{Ve:int}, Section~5) and can be found for example in \cite{Mo:ren} (see Lemma 3.4) or \cite{AGY:exp} (see Lemma 4.4). 

To control distorsion, it is useful to introduce the following quantity (see for example Section~2 in \cite{Bu:dec}).  Given a $d \times d$ positve matrix $C$, let us define $\col{C}$ to be 
\begin{equation}\label{def:col}
\col{C}\doteqdot  \max_{1\leq i,j,k\leq d} \frac{C_{ij}}{C_{ik}}.
\end{equation}
Let us remark that (see also Proposition 2 in \cite{Bu:dec}) if $C$ is a $d \times d$ matrix with non negative entries and $D$ is a $d\times d$ matrix with positive entries (so that in particular $CD$ has positive entries and $\col{CD}$ is well defined) one has
\begin{equation}\label{Colproduct}
\col{CD}\leq \col{D}.
\end{equation} 
Then one has the following lemma (see also Corollary 1.7 in \cite{Ke:sim} and equation (15) in \cite{Bu:dec}). 

\begin{lemma}[distorsion]\label{rk:bounded_distorsion}  
If $C$ is  a $d \times d$ positive matrix,  the Jacobian $J_C$ of the map $\lambda \mapsto C \lambda / |C \lambda|$ satisfies
$$
\sup_{\lambda, \lambda' \in \Delta_d} \left|  \frac{J_C (\lambda)}{ J_C (\lambda')} \right| \leq \col{C}^d .
$$
\end{lemma}
\begin{proof}
Since  $\Delta_d$ is generated by its vertices, whose image under the map $\lambda \mapsto C\lambda$ are the columns of $C$, we have
$$
\sup_{\lambda, \lambda' \in \Delta_d} \frac{|C\lambda'|}{|C\lambda|} = \max_{1\leq j ,k\leq d } \frac{\sum_{i=1}^d C_{ij} }{\sum_{i=1}^d C_{ik} }\leq \max_{1\leq j,k \leq d }  \frac{\sum_{i=1}^d \left(\frac{C_{ij}}{C_{ik}}\right) C_{ik} }{\sum_{i=1}^d C_{ik} }\leq \col{C} . 
$$
Thus, the estimate follows from the explicit form of $J_C (\lambda)$ (see Remark \ref{rk:formJac}).
\end{proof}

\subsubsection*{Exponential tails.} 
The main technical tool for us is the result proved  by Avila, Gou\"{e}zel and Yoccoz in \cite{AGY:exp} (in order to show exponential mixing of the Teichmueller flow), i.e.\ the existence of pre-compact accelerations for which the return time has exponential tails.  Using the  terminology introduced so far, one can rephrase the main result proved in \cite{AGY:exp} as follows.
\begin{thm}[Theorem 4.10 in \cite{AGY:exp}]\label{AGY}
For every $\delta>0$, there exists a cylindrical pre-compact acceleration  $\R_{\hat{Y}_\delta}$ 
(corresponding to returns to a set $\hat{Y}_\delta\subset \hat{X}$  which is finite union of cylinders for $\hat{\R}$) 
 such that the corresponding accelerated cocycle $A_{\hat{Y}_\delta}$ given by  \eqref{acccocycle_def} satisfies
\be \label{integrabilityAGY}
\int_{\hat{Y}_\delta  } \| A_{\hat{Y}_\delta} (\hat{T}) \|^{1-\delta} \, d {\mu}_{\hat{\R}}(\hat{T}) < \infty.
\ee
\end{thm}
Note that by \cref{rk:automcylin} we immediately obtain that the times in the sequence $\{n_\ell\}_{\ell\in\N}$ corresponding to the acceleration $\hat{\R}_{\hat{Y}_\delta}$ are positive and balanced.
The original statement of Theorem $4.10$  in \cite{AGY:exp} claims the integrability of $e^{(1-\delta)r_{\hat{Y}_\delta}(\hat{T})}$, where $r_{\hat{Y}_\delta}(\hat{T})$ is the first return time of $\hat{T} \in \hat{Y}_\delta$ under the Veech flow. For the reduction to this formulation, see \cite{Ul:mix} and recall the notation introduced above.
We recall that Bufetov, by different techniques, obtained in \cite{Bu:dec} a result analogous to (\ref{integrabilityAGY}) for {some} $\delta\in(0,1)$. We will need the full strenght of the result of \cite{AGY:exp}, i.e.\ for any $\delta\in (0,1)$.

\section{Diophantine conditions for IETs}\label{se4}

As mentioned in the introduction, 
Diophantine conditions for IETs can be expressed in terms of the growth behavior of the Rauzy-Veech cocycle matrices along a sequence of positive times. In this section, we first define a \emph{Diophantine condition} for IETs (in Definition~\ref{def:mixingDC}) which holds for a full measure set of IETs and that was used by the third author in \cite{Ul:mix} and by Ravotti in \cite{Ra:mix} to prove mixing  for special flows over IETs (more precisely, this condition allows to prove that the Birkhoff sums $\BS{f}{r}$  of a  function $f$ with asymmetric logarithmic singularities and its derivatives satisfy precise quantitative estimates, see Proposition~\ref{mpd}).  We then define a stronger Diophantine Condition (see Definition~\ref{def:ratnerDC} in Section~\ref{sec:RatnerDC}) which will allow us to prove a quantitative form of parabolic divergence and, as a consequence, mixing of all orders for the same type of special flows whose base IET enjoys this stronger Diophantine property. The main result  of this section is that this condition is satifsied by a full measure set of IETs (see Proposition \ref{RDCfullmeasure}).

\subsection{Mixing Diophantine condition} \label{existencebalancedtimessec}
 The following Diophantine condition was introduced by the third author in \cite{Ul:mix} to show mixing for the class of special flows under a roof function with only one asymmetric singularity. Ravotti in~\cite{Ra:mix}  extends this result and shows that, in fact, the same condition implies mixing also when the roof function has several asymmetric logarithmic singularities (see Theorem \ref{DCmixing} below).

\begin{defn}[Mixing DC, see \cite{Ul:mix}]\label{def:mixingDC}
We say that an IET $T$ satisfies the \emph{mixing Diophantine condition} (or, for short, satisfies the \emph{mixing DC}) with integrability power $\tau$  if $1< \tau < 2$ and there exist $\overline{\ell} \in \mathbb{N}$, $\nu > 1$ and a sequence $\{ n_\ell\}_{\ell\in \mathbb{N}}$ of \emph{balanced} induction times such that:
\begin{itemize}
\item
the subsequence  $\{ n_{\overline{\ell} k} \}_{k \in \mathbb{N}}$ is \emph{positive}, 
\item
the matrices $ B^{(n_{\ell\phantom{\overline{\ell}}},n_{\ell+\overline{\ell}})}$ have  \emph{uniformly bounded diameter} with respect to the Hilbert metric (recall the definition in \eqref{diam_def}), i.e. there exists $D>0$ such that $diam_H(B^{(n_{\ell\phantom{\overline{\ell}}},n_{\ell+\overline{\ell}} )})\leq D$ for any $\ell\geq 1$.
\item setting $A_\ell\doteqdot B^{(n_\ell,n_{\ell+1})}$, the following \emph{integrability} condition holds:
\be \label{integrability} 
\lim _{\ell \rightarrow + \infty} \frac {\|A_\ell \|}{\ell^{\tau}} =  \lim _{\ell \rightarrow + \infty} \frac {\|B^{(n_\ell,n_{\ell+1})} \|}{\ell^{\tau}} =0 .
\ee
\end{itemize}
We denote by $\MDCt{\tau}{\overline{\ell}}{\nu}$ the set of IETs which satisfy the {mixing DC} with integrability power $1< \tau<2$ and parameters $\overline{\ell} \in \mathbb{N}$ and $\nu > 1$ and we denote by  $\MDCo{\tau}$ the IETs which satisfy the mixing DC with integrability power $\tau$, that is the  union over $\overline{\ell} \in \mathbb{N}$ and $\nu > 1$ of $\MDCt{\tau}{\overline{\ell}}{\nu}$.
If $T \in \MDCt{\tau}{\overline{\ell}}{\nu}$ for some $1< \tau<2$, $\overline{\ell} \in \mathbb{N}$ and $\nu > 1$, we simply say that $T$ satisfies the mixing DC. 
\end{defn}
\begin{rem}[\cite{Ul:mix}]
We remark that (\ref{integrability}) implies for $d=2$ the Diophantine condition used for rotations in \cite{SK:mix} to prove mixing for typical minimal components in genus one (i.e.\ $k_\ell = \rm{o}( \ell^{\tau})$, where $\{ k_\ell \}_{\ell \in \mathbb{N}}$ are the entries of the continued fraction and the exponent $\tau$ satisfies the same assumption $1 < \tau <2$).  
\end{rem}


The following result was proved by the third author in \cite{Ul:mix} when $f$ has only one singularity, then extended to several singularities by Ravotti \cite{Ra:mix}. 

\begin{thm}[\cite{Ul:mix, Ra:mix}]\label{DCmixing}
If  $T\colon I \to I$ satisfies the Mixing DC, for  every roof function $f\colon I \to \mathbb{R}^+$ with \emph{asymmetric logarithmic singularities}  at the discontinuities of $T$ (in the sense of Definition \ref{def_LogAsym}), the special flow $(\varphi_t)_{t\in\mathbb{R}}$ over $T$ under $f$ is mixing.
\end{thm}

The following Proposition is proved by the third author in \cite{Ul:mix} (see the proof of Proposition 3.2\footnote{In the statement of Proposition 3.2 in \cite{Ul:mix}
it is only claimed that for $1< \tau < 2$ the set of IETs which satisfies the Mixing DC with integrability power $\tau$, i.e. what we here call  $\MDCo{\tau}$ has full measure. By reading the actual proof of Proposition 3.2 in \cite{Ul:mix}, though, one can see that  $\overline{\ell}\in\mathbb{N}$ and $\nu>1$ and chosen at the beginning of the proof and the full measure set of IETs constructed all share the same parameters $\overline{\ell}$ and $\nu$, i.e. the proof does show indeed that the result here cited holds.}).
\begin{prop}[\cite{Ul:mix}]\label{existencebalancedtimes} 
For any $1< \tau < 2$, there exists  $\overline{\ell}\in\mathbb{N}$ and $\nu>1$ such that the set $\MDCt{\tau}{\overline{\ell}}{\nu}$ has full measure, i.e.  for each irreducible combinatorial datum $\pi$ and for Lebesgue a.e.\  $\underline{\lambda} \in  \Delta_{d}$, the corresponding IET $T= (\underline{\lambda}, \pi) $ belongs to $\MDCt{\tau}{\overline{\ell}}{\nu}$. 
\end{prop}
\begin{rem}\label{rk:sequencemixingDC}
The key ingredient in the proof of the above result in~\cite{Ul:mix} is the main estimate on exponential tails from \cite{AGY:exp} which we recall in \cref{AGY}.  More specifically, the sequence $\{ n_{\ell} \}_{\ell \in \mathbb{N}}$ from the definition of the mixing DC is constructed from the sequence of induction times corresponding to the finite union of cylinders $\hat{Y}_\delta$, $\delta=1-\tau^{-1}$ from~\cref{AGY} (recall from Remark \ref{rk:sequence_indep_lift} that this sequence essentially depends on the IET only). The first two conditions in Definition \ref{def:mixingDC} follow easily from the fact that this acceleration is cylindrical and positive and the integrability condition  can be deduced from the exponential tails condition in  \cref{AGY} (see the proof of Proposition 3.2 in \cite{Ul:mix}).  
\end{rem}

\subsection{Ratner Diophantine condition}\label{sec:RatnerDC}
In this section we introduce a Diophantine condition which we will later use (see in particular in sections \ref{sec:summability} and \ref{sec:deducingRatner})  to quantify parabolic divergence and prove the switchable Ratner property for suspension flows with asymmetric logarithmic singularities. For this we will need some notation. Let $T$ be an IET satisfying the Keane condition. Recall that  $B^{-1}$ denotes the Rauzy-Veech cocycle and that $\underline{h}^{(n)}=(h^{(n)}_\alpha)_{\alpha\in\cA}$ is the vector of heights (see Section~\ref{Se:RV}). Given a sequence $\{n_\ell\}_{\ell\in\N}$ of induction times, we define
\begin{equation}\label{CF_like_notation}
A_\ell = A_\ell(T) \doteqdot  B^{(n_\ell, n_{\ell+1})} (T),\qquad  q_\ell \doteqdot \max_{\alpha\in\cA} h^{(n_\ell)}_\alpha \text{ for $\ell\geq 1$}
\end{equation}
(equivalently, $q_\ell$ is the norm of the largest column of the transpose of the matrix $A^{(\ell)}=B^{(n_\ell)}$). The above notation is chosen this way to resemble the standard notation related to the continued fraction expansion algorithm, since $A_\ell$ and $q_\ell$   play  an analogous role to entries and denominators of the convergents of the continued fraction algorithm, in the context of the Rauzy-Veech multidimensional continued fraction algorithm.  We remark that since we extended the Rauzy-Veech cocycle $B^{-1}$ to a cocycle over the natural extension $\hat{\R}$ of $\R$ (see Section~\ref{Se:RV}), we can define $A_\ell$ also for \emph{negative} indexes $\ell$, by setting:
\begin{equation}\label{CF_like_notationRhat} A_\ell = A_\ell(\hat{T}) \doteqdot B^{(n_\ell,n_{\ell+1})}(\hat{T}), \qquad \ell\in\mathbb{Z}, \quad  \hat{T}=( \tau, \pi, \lambda) \in \hat{X}.
\end{equation}

From now on we assume that $\overline{\ell}, \nu$ are so that the conclusion of Proposition \ref{existencebalancedtimes} holds.
\begin{defn}[Ratner DC]\label{def:ratnerDC}
We say that an IET $T$ satisfies the \emph{Ratner Diophantine condition} if $T$ satisfies the Mixing DC $\MDCt{\tau}{ \overline{\ell}}{ \nu}$ along a subsequence $(n_\ell)_{\ell\in\mathbb{N}}$ for some  $\nu>1$ and $\overline{l}\in \mathbb{N}$ and 
there exists $\xi<1, \eta<1$ such that $A_\ell$  and $q_\ell $ defined in \eqref{CF_like_notation} satisfy 
\be \label{RatnerDC:eq}
\sum_{\ell\in\N \ \text{s.t.\ } \|A_\ell\| \|A_{\ell+1}\|\ldots\|A_{\ell+L}\|> \ell^\xi } \frac{1}{(\log q_\ell)^\eta } < + \infty, \text{ where }  L=\overline{\ell}(1+ \left[ \log_d (2 \nu^2) \right]).
\ee
(Here $[ \cdot ]$ denotes the integer part).  In this case, we say that $T \in \RDCt{\tau}{ \xi}{ \eta}$. We also write $\RDCo{\tau}$ for the union of all $\RDCt{\tau}{ \xi}{ \eta}$ over $\xi<1$ and  $\eta<1$. 
\end{defn}



We remark that in the definition of $\RDCt{\tau}{ \xi}{ \eta}$ we do not record  the explicit  dependence on $\nu>1$ and $\overline{l}\in \mathbb{N}$, it is sufficient that the Ratner DC holds with respect to some such $\eta$ and $\overline{l}$. In the rest of the paper, we will use the Ratner DC   $\RDCt{\tau}{ \xi}{ \eta}$ only for values $1< \tau < 16/15$. 

\begin{rem}\label{FK_DC}
The Ratner DC should be compared with the Diophantine Condition introduced by Fayad and the first author in \cite{FK:mul}. In \cite{FK:mul} the authors define

$$\mathcal{E}\doteqdot \left\{\alpha\in [0,1) :  \sum_{i\notin K_\alpha}\frac{1}{\log^{7/8}q_i}<+\infty\right\},$$
where $K_\alpha=\{i\in \mathbb{N} :  q_{i+1}< q_i\log^{7/8}q_i\}$, and show that $\lambda(\mathcal{E})=1$ (see Proposition 1.7 in \cite{FK:mul}). This corresponds to Ratner DC with $\xi=\eta=7/8$.

Notice that if an IET $T$ is of bounded type (which form a $0$ measure set) i.e. for all $\ell\in\mathbb{N}$, $\|A_\ell(T)\|<C$, then Ratner DC is automatically satisfied (we sum only finitely many terms). Ratner DC means that the times $\ell$, where $A_\ell(T)$ is large are not too frequent. In a sense if an IET satisfies Ratner DC, it behaves like an IET of bounded type modulo some error with small density (as a subset of $\mathbb{N}$), but, as Proposition \ref{RDCfullmeasure} shows, this relaxation allows the property to hold for a full measure set of IETs. 
\end{rem}

The main result of this section is that Ratner DC has full measure for a suitable choice of the parameters.
\begin{prop}[full measure of Ratner DC]\label{RDCfullmeasure}
For any $\tau\in(1,16/15)$, $\xi\in(11/12,1)$ and $\eta> \frac{1}{2}$, the set $\RDCt{\tau}{ \xi}{ \eta}$ has full measure, so in particular
for each irreducible combinatorial datum $\pi$ and for Lebesgue a.e.\  $\underline{\lambda} \in  \Delta_{d}$, the corresponding IET $T= (\underline{\lambda}, \pi) $ belongs to $\RDCt{\tau}{\xi}{\eta}$ and hence satisfies the Ratner DC. 
\end{prop}
\noindent The proof of this result will be presented in Section~\ref{sec:fullmeasure}. In the next Section~\ref{sec:QB} we introduce some tools needed in the proof.
%

\subsection{Quasi-Bernoulli property}\label{sec:QB}
The bounded distorsion property of pre-compact accelerations of the Rauzy-Veech map $\R$ (see Remark \ref{rk:bounded_distorsion}) guarantees a quasi-Bernoulli kind of property (see also Corollary 1.2 in \cite{Ke:sim} or Proposition 3 in \cite{Bu:dec}).


\begin{lemma}[QB-property]\label{lemma:preQB}
For every $d_0>0$ there exists a constant $\nu=\nu(d_0)>1$ such that for any two positive  Rauzy-Veech products $E, F$ which can be concatenated (i.e. such that $EF$ is  also a Rauzy-Veech product) and such that  $\col{E}\leq d_0$ and the projective diameters $diam_H(E), diam_H(F)$ are bounded by $d_0$, we have 
\bes
 \frac{1}{\nu} \, \muR(\Delta_E)  \muR(\Delta_F) < \muR (\Delta_{E F})< \nu \, \muR(\Delta_E)  \muR(\Delta_F).
\ees
\end{lemma}
\begin{proof}
Remark first that $\Delta_{E} $ is by definition the image of $\Delta_{d}$ under the map $\psi_E$ 
given by $\psi_E (\lambda) = E \lambda / |E \lambda|$.  Thus, by the change of variable formula (recalling that we denote by $Leb_X$ the measure which coincides with the restriction of the Lebesgue measure on $\mathbb{R}^d$ to  the simplex $\Delta_d$ for each of the copies $\Delta_\pi = \Delta_{d} \times \{\pi\}$) we have that  
$Leb_X(\Delta_{E}) = \int_{\Delta_{d}} {J_E (\lambda)}\, d \lambda$, where $J_E$ denotes the Jacobian of the map $\psi_E$. Similarly, since $\Delta_{EF} $   is the image under  $\psi_E$ of  $\Delta_F$,    $Leb_X(\Delta_{EF}) = \int_{\Delta_{F}} {J_E (\lambda)}\, d \lambda$. 
Thus, from  Lemma \ref{rk:bounded_distorsion}, we have that
\bes
 \frac{1}{\col{E}^d} \, \frac{ Leb_X (\Delta_{E})} { Leb_X (\Delta_{d})} < \frac{ Leb_X (\Delta_{EF})} { Leb_X (\Delta_F)}<  \col{E}^d \, \frac{ Leb_X (\Delta_{E})} { Leb_X (\Delta_{d})}. 
\ees
The claim of the lemma hence follows   by  remarking that $\muR$ is absolutely continuous with respect to $Leb_X$ with density which is bounded on compact sets. Indeed, 
 since $\Delta_E$, $\Delta_F$ and $\Delta_{EF}$ (which is contained in $\Delta_E$) are pre-compact and of diameter bounded by $d_0$, $\muR(E)$, $\muR(F)$ and $\muR(EF)$ are comparable to $Leb_X(E)$, $Leb_X(F)$ and $Leb_X(EF)$ respectively with constants which depend on $d_0$ only. 
\end{proof}

The technical results that we need in order to prove that IETs with the Ratner DC have full measure are the following Lemma and consequent Corollary, which are both applications of the above quasi-Bernoulli property. They show that correlations between events described by prescribing matrices of an acceleration of the Rauzy-Veech induction can be estimated if the acceleration is pre-compact.

\begin{lemma}\label{QBforA}
Let $A=A_{\hat{Y}}$ be a pre-compact cylindrical acceleration of the Rauzy-Veech natural extension $\hat{\R}$. Then there exists a constant $c=c(\hat{Y})>1$ such that for any integers $0\leq l<m<n$ and any matrices $C_0,C_1 \in SL(d,\mathbb{Z})$ we have that
\begin{multline}\label{qb1}
\hat{\mu}\left( \hat{T}\in \hat{Y} :\  A_l A_{l+1}\cdots A_{m-1}(\hat{T})=C_0 ,\ A_m A_{m+1} \cdots A_{n}(\hat{T})=C_1 \right) \\
\leq c\cdot \hat{\mu}\left(\hat{T}\in \hat{Y} : \ A_l A_{l+1}\cdots A_{m-1}(\hat{T})=C_0 \right)\cdot \hat{\mu}\left(\hat{T}\in \hat{Y} :\ A_m A_{m+1} \cdots A_{n}(\hat{T})=C_1 \right).
\end{multline}
\end{lemma}

\begin{cor}\label{QBforAcor}
For any pre-compact acceleration $A=A_{\hat{Y}}$ and any $N \in \mathbb{N}$  there exists $c=c(\hat{Y}, N)$ such that for any choice of integers $l_1< l_2 <  \cdots < l_N$ and $i_1, \dots, i_N $ we have that
\begin{multline}\label{qb1} 
\hat{\mu}\left( \hat{T}\in \hat{Y} : \|A_{l_1}(\hat{T}) \|=i_1,\ \|A_{l_2}(\hat{T})\|=i_2 , \dots, \ \|A_{l_N}(\hat{T})\|=i_N  \right) \\
\leq c \ \cdot \hat{\mu}\left(\hat{T}\in \hat{Y} : \|A_{l_1}(\hat{T})\|=i_1\right)\cdots \hat{\mu}\left(\hat{T}\in \hat{Y} : \|A_{l_N}(\hat{T})\|=i_N\right).
\end{multline}
\end{cor}

Let us first give the proof of Corollary \ref{QBforAcor}, then the one of Lemma \ref{QBforA}.

\begin{proof}[Proof of Corollary \ref{QBforAcor}]
Let $C_1, \dots, C_N$ be Rauzy-Veech matrices and $D_1, \dots, D_{N-1} $ Rauzy-Veech products such that for some $\hat{T} \in \hat{Y}$ we have
$$A_{l_i}(\hat{T}) = C_i\quad  \text{for}\, 1\leq i \leq N, \qquad D_i = A_{l_i+1}\cdots A_{l_{i+1}-1} \quad \text{for}\, 1\leq i \leq N-1, $$ 
So that in particular $A_{l_1}A_{l_1+1}\cdots A_{l_N}(\hat{T}) = C_1 D_1 C_2D_2 \cdots C_{N-1} D_{N-1}C_{N}$. By applying Lemma \ref{QBforA} $2N-1$ times in order to split up the product into a product of matrices $C_i$ and $D_i$ (more precisely applying it to $C_i$ and $D_{i}\cdots C_N$ or $D_i$ and $C_{i+1} \cdots C_N$ for $i=1,\dots, N-1$) we have that
\begin{multline}\nonumber
\hat{\mu}\left( \hat{T}\in \hat{Y} : A_{l_1}A_{l_1+1}\cdots A_{l_N}(\hat{T}) = C_1 D_1 C_2D_2 \cdots C_{N-1} D_{N-1}C_{N} \right) \\
\leq c(\hat{Y})^{2N-1} \cdot \Pi_{i=1}^N\hat{\mu}\left(\hat{T}\in \hat{Y} : A_{l_i}(\hat{T})=C_i\right)\, \Pi_{i=1}^{N-1} \hat{\mu}\left(\hat{T}\in \hat{Y} : A_{l_i+1} \cdots  A_{l_{i+1}-1}(\hat{T})=D_i  \right).
\end{multline}
For brevity let us denote by $\mathcal{C}_i(D_i)$ the event $\{ \hat{T}\in \hat{Y} : A_{l_i+1} \cdots  A_{l_{i+1}-1}(\hat{T})=D_i \}$. Then, summing over all possible choices of the matrices $D_1, \dots, D_{N-1} $ and using that the events $\mathcal{C}_i(D_i))$ for $1\leq i\leq N-1$ are disjoint, we get that
$$
\sum_{D_1, \dots, D_{N-1}} \Pi_{i=1}^{N-1} \hat{\mu}( \mathcal{C}_i(D_i)) \leq  \Pi_{i=1}^{N-1} \left( \sum_{D_1, \dots, D_{N-1}} \hat{\mu}( \mathcal{C}_i(D_i)) \right) = \Pi_{i=1}^{N-1} \mu \left( \cup_{i=1}^{N-1} \mathcal{C}_i(D_i) \right) \leq \hat{\mu}(Y)^{N-1}
$$ 
Summing up over all possible choices of matrices $C_1, \dots, C_N$ and $D_1, \dots, D_{N-1} $ as above and such that $\|C_{j}\|=i_j$ for  $j=1,\dots, N$ 
we get \eqref{qb1} for $c=c(\hat{Y},N)\doteqdot c(\hat{Y})^{2N-1} \hat{\mu}(\hat{T})^{N-1}$.
\end{proof}

The proof of  Lemma \ref{QBforA} is based on the remark that in a pre-compact cylindrical acceleration $A=A_{\hat{Y}}$,  every return of the orbit of ${\triple}$ to $\hat{Y}$, i.e. every $l$ such that $\hat{\R}_{\hat{Y}}^l (\triple) \in \hat{Y}$, corresponding to visits of $\hat{\R}_{\hat{Y}}^l (\triple)$ to a cylinder $(\Theta_C \times \Delta_D )^{(1)}$ where $C,D$ are positive matrices with uniformely bounded Hilbert diameter and $\nu_{col}$. Equivalently, this means that one sees the block $CD$ appearing in the cocycle products centered at time $l$, i.e. the cocycle matrices $ A_{l+L} \cdots A_{l-2} A_{l-1}$ end with $C$  and  $A_l A_{l+1}, \dots A_{l+L}$ start with $D$ for some $L$\footnote{One can make special choices of cylindrical accelerations (see for example Avila-Gouezel-Yoccoz \cite{AGY:exp} which guarantee that $A_l$ (resp. $A_{l-1}$) is sufficiently long so that it has to start (resp. end) with some specific matrices $D$ (resp. $C$). In general, it might happen though that to see an occurrence of $C$ or $D$ one needs to consider several successive steps, which complicates the writing of the proof)}. This hence  allows to use the QB-property. 

\begin{proof}[Proof of Lemma \ref{QBforA}]
Remark first that by definition $$A_l A_{l+1}\cdots A_{m-1} ({\triple} ) = A_0 A_{1}\cdots A_{m-l-1} (\hat{\R}_{\hat{Y}}^{l} (\pi, \lambda, \tau) ), $$ so that,
since the measure $\hat{\mu}$ is invariant under $\R$ and hence its restriction to $\hat{Y}$ is invariant under the Poincar{\'e} map $\R_{\hat{Y}}$, it is enough to prove the Lemma statement for $l=0$ (where $0<m<n$ play the role of the former $0<m-l<n-l$).   
Since $A=A_{\hat{Y}}$  is a cylindrical acceleration of $\hat{\R}$, $\hat{Y}$ is the union of finitely many cylinders for $\hat{\R}$, that we will denote by $(\Theta_{E_k} \times \Delta_{F_k} )^{(1)}$ for $1\leq k\leq N$.  Furtheremore  since $A$ is precompact all matrices $E_i$ and $F_i$ are positive and hence have finite Hilbert diameter by Remark \ref{rk:finitediam}. Let $d_{\hat{Y}}$ be the maximum of the projective diameters $diam_H(E_i), diam_H(F_i)$ for $1\leq i \leq N$ and of $\col{E_i}$ for $1\leq i \leq N$. Let $\nu_{\hat{Y}}\doteqdot  \nu(d_{\hat{Y}})$ be the constant  given by Lemma \ref{lemma:preQB}.  
Let $n_0$ be the maximal number of Rauzy-Veech  matrices produced to obtain any of the matrices $E_i$ or $F_i$ in the definition of the cylinders.   Notice that it is enough to prove the statement of the Lemma under the assumption that either $m\geq n_0$ or $n-m\geq n_0$,  since the possible Rauzy-Veech products of $n_0$ Rauzy-Veech  matrices are finitely many and  hence the finite number of possibilities with   $m\leq n_0$ and  $n \leq m + n_0 \leq 2 n_0$ only change the constant $c$. 

The cocycle $A=A_{\hat{Y}}$ associated to the first return map  $\hat{\R}_{\hat{Y}}$ to $\hat{Y}$ is locally constant and, when restricted to one of the cylinders $(\Theta_{E_k} \times \Delta_{F_k})^{(1)} $ in $\hat{Y}$, the set where it takes as a value a fixed given Rauzy-Veech product $C$ of $p$ matrices and, after $p$ iterations of $\hat{\R}_{\hat{Y}}$ one lends to $(\Theta_{E_l} \times \Delta_{F_l})^{(1)} $, i.e. the set
\bes
\mathcal{C}=\mathcal{C}(C,k,l )\doteqdot  \left\{ {\triple} \in (\Theta_{E_k} \times \Delta_{F_k})^{(1)} \ s.t. \ A ({\triple}) =C, \ \hat{\R}_{\hat{Y}}^p ({\triple}) \in (\Theta_{E_l} \times \Delta_{F_l})^{(1)}  \right\}
\ees
is a Rauzy-Veech cylinder of the form  $(\Theta_{D} \times \Delta_{CF_l} )^{(1)}$, where  $\Theta_D \subset  \Theta_{E_k} $ and $\Delta_{CF_k} \subset \Delta_{F_k}$ (which explicitly means that the product $CF_l $ \emph{starts with} $F_k$, i.e. it has the form $F_k C'$ for some non-negative $C'$, and the product $D$ \emph{ends with} $E_k$ i.e. it has the form $D' E_k$ for some non negative $D'$) and, since by definition of the cocycle and cylinders (see in particular \eqref{cylindersprop}),  recalling that $p$ is the number of Rauzy-Veech matrices produced to get $C$,   we have that

\bes
\hat{\R}_{\hat{Y}}^p \left((\Theta_{D} \times \Delta_{CF_l})^{(1)}\right) = (\Theta_{DC} \times \Delta_{F_l})^{(1)} \subset (\Theta_{E_l} \times \Delta_{F_l})^{(1)},
\ees
the matrix $DC$ ends with $E_l$: more precisely, either $D=E_k$ and hence $DC= E_k C = C''E_l $ for some non negative $C''$ so that $\Theta_{DC} \subset \Theta_{E_l}$ (which happens if the matrix $C$ is sufficently large so that $E_k C$ is a longer Rauzy-Veech product than $E_l$), or otherwise $\Theta_{D} $ is a strict subset of $\Theta_{E_k} $ chosen so that $DC=E_l$ and hence $\Theta_{DC}=\Theta_{E_l}$. 

Let $C_0, C_1$ be any two matrices in $SL(d, \mathbb{Z})$. Remark that we can assume that $C_0, C_1$ are Rauzy-Veech products  that  can be concatenated, i.e. that $C_0=A_0 A_1 \cdots A_{m-1}( {\triple} ) $ and $C_1=A_m\cdots A_{n-1}( {\triple} )$ for some (a positive measure set of) $\hat{T}=(\tau,  \lambda, \pi) \in \hat{Y}$), since otherwise in the statement either the RHS or both sides are zero and the statement is trivially true. 
Let $\mathcal{C}_0$ and $\mathcal{C}_1$  
 be the cylinders as described above on which the cocycle $A$ is locally equal to $C_0$ and $C_1$ respectively and such that $\hat{R}_{\hat{Y}}^m \mathcal{C}_0 \cap \mathcal{C}_1 \neq \emptyset$. In virtue of the remark above, we write for some $1\leq k_0,k_1,k_2\leq N$,  
\bes
\mathcal{C}_i = ( \Theta_{D_i} \times \Delta_{C_i F_{k_i}})^{(1)} \subset (\Theta_{E_{k_i}} \times \Delta_{F_{k_i}})^{(1)} \quad \text{for} \ i=0,1  
\ \text{where}
\ees
\be\label{D1ass}
\text{either} \ D_1=E_{k_1}, \qquad  \text{or} \ D_1 \ \text{ends\ with}\ E_{k_1} \ \text{and} \ D_1 C_1=E_{k_2}  . 
\ee
Similarly, since $C_0$ is the product of $m$ Rauzy-Veech matrices, one can see that 
\bes
\mathcal{C}_0 \cap \hat{R}_{\hat{Y}}^{-m} \mathcal{C}_1 = 
  (\Theta_{D} \times \Delta_{C_0C_1 F} )^{(1)}, 
\ees 
where (according to which between $C_1F_{k_2}$ or $F_{k_1}$  is a  longer Rauzy-Veech product) 
\be\label{C1Fass}
\text{either} \ F=F_{k_2} \ \text{and} \ C_1F \  \text{starts\ with}\ F_{k_1},\qquad  \text{or}  \ C_1F=F_{k_1}\ \text{and} \  F \ \text{starts\ with}\ F_{k_2}, 
\ee
and $D$ is such that
  \be\label{Dass}
	\begin{split}
 & \text{either} \ D = E_{k_0},   \ \text{and} \ DC_0  \ \text{and} \ DC_0C_1 \ \text{end \ with}\ E_{k_1} \ \text{and} \ E_{k_2} \ \text{respectively}, \\
&\text{or} \ DC_0 = E_{k_1}, \ \text{and} \ D  \ \text{and} \ DC_0C_1 \ \text{end \ with}\ E_{k_0} \ \text{and} \ E_{k_2} \ \text{respectively};
\end{split}
\ee 
(the last possibility, i.e.~that $DC_0C_1 = E_{k_2}$ is excluded since the assumption that either $m\geq  n_0$ or $n-m\geq n_0$ ensure that $C_0C_1$ is a sufficiently long Rauzy-Veech product to begin with $E_{k_2}$). 

Let us now compare the measures of $\mathcal{C}_0, \mathcal{C}_1 $ and $\mathcal{C}_0 \cap \hat{R}_{\hat{Y}}^{-m} \mathcal{C}_1$.  Using that $\hat{\mu}$ is $\hat{\R}$ invariant, the definition of cylinders (see in particular \eqref{cylindersprop}) and that $\hat{\mu}= p_* \muR$, we have that for $l=0,1$, if $D_l$ is a Rauzy-Veech product starting at $\pi_l$
\be\label{measures}
\hatmu (\mathcal{C}_l)
= \hatmu \left((\Theta_{D_l} \times \Delta_{C_l F_{k_{l+1}}})^{(1)} \right) 
= \hatmu \left( (\Theta_{\pi_l} \times \Delta_{D_lC_j F_{k_{l+1}}})^{(1)} \right) = {\muR} \left( \Delta_{D_{l}C_l F_{k_{l+1}}} \right).
\ee
Similarly, 
$\hatmu (\mathcal{C}_0 \cap \hat{R}_{\hat{Y}}^{-m} \mathcal{C}_1)  = 
\muR (\Delta_{DC_0 C_1 F})$. 
We can now get that 
\bes
\hatmu (\mathcal{C}_0 \cap \hat{R}_{\hat{Y}}^{-m} \mathcal{C}_1 )
  \leq \nu_{\hat{Y}} \muR(\Delta_{D C_0}) \, \muR(  \Delta_{C_1 F}). 
\ees
by applying the upper inequality  in Lemma \ref{lemma:preQB} 
and recalling the definition of $\nu_{\hat{Y}}$: the assumptions of the Lemma holds since $\col{D C_0}\leq \col{E_{k_1}}$ by \eqref{Dass} and~\eqref{Colproduct}, $diam_H(DC_0) \leq   \max \{ diam_H (E_{k_0}), diam_H(E_{k_1}) \}$  by \eqref{Dass} and  Remark \ref{rk:diam} and   $diam_H(C_1F)\leq diam_H(F_{k_1})$ by \eqref{C1Fass} and  Remark \ref{rk:diam}.  Now, dividing and multiplying by $\muR (\Delta_{F_{k_1}})$ and $\muR (\Delta_{D_1})$ and by remarking that by \eqref{D1ass} either $\Delta_{D_1}=\Delta_{E_{k_1}}$ or otherwise $\Delta_{D_1} \supset \Delta_{D_1 C_1}=\Delta_{E_{k_2}}$ and hence in both cases $\mu(D_1) \geq \inf_{1\leq k\leq N} \mu(\Delta_{E_k})$, we get that
\be\label{estimate:product} \hatmu (\mathcal{C}_0 \cap \hat{R}_{\hat{Y}}^{-1} \mathcal{C}_1 )  \leq \nu_{\hat{Y}} \frac{ \muR (\Delta_{DC_0}) \muR (\Delta_{F_{k_1}}) \muR (\Delta_{D_1})  \muR(  \Delta_{C_1F})}{\inf_{1\leq i\leq N} \muR(\Delta_{F_i} ) \inf_{1\leq i\leq N} \muR (\Delta_{E_i} )  }.
\ee
Applying now the lower inequality in Lemma  \ref{lemma:preQB} (which can be applied thanks to \eqref{D1ass}, \eqref{C1Fass} and the definition of $\nu_{\hat{Y}}$, which give that $\col{D_1}\leq \col{E_{k_1}}$ by \eqref{Colproduct} and allow to estimate diameters using  Remark \ref{rk:diam}), and then remarking that \eqref{C1Fass} implies that  $\Delta_{D_1 C_1 F} \subset \Delta_{D_1C_1F_{k_2}}$ and using \eqref{measures} for $l=1$,   we have that
\be\label{est2} 
\muR (\Delta_{D_1})  \muR(  \Delta_{C_1 F}) \leq  \nu_{\hat{Y}} \muR(\Delta_{D_1 C_1 F})\leq  \nu_{\hat{Y}} \muR(\Delta_{D_1 C_1 F_{k_2}}) =  \nu_{\hat{Y}} \hatmu(\mathcal{C}_1).
\ee
Similarly, again by the lower inequality in Lemma  \ref{lemma:preQB} (which can be used this time thanks to \eqref{Dass}, \eqref{Colproduct} and Remark \ref{rk:diam}, which in particular yield $\col{DC_0}\leq \col{E_{k_1}}$), reasoning as in \eqref{measures} and then remarking that $\Theta_{D} \subset \Theta_{E_{k_0}}$ by \eqref{Dass}, we also have that
\be\label{est1}
\muR(\Delta_{DC_0}) \muR (\Delta_{F_{k_1}}) \leq  \nu_{\hat{Y}} \muR(\Delta_{DC_0 {F_{k_1}}}) =  \nu_{\hat{Y}} \hatmu \left(( \Theta_{D}\times \Delta_{C_0 F_{k_1}})^{(1)}\right)  \leq  \nu_{\hat{Y}} \hatmu (\mathcal{C}_0). 
\ee
Combining  \eqref{estimate:product}, \eqref{est2} and \eqref{est1},  
 we finally get
\be\label{Cestimate}
\hatmu ( \mathcal{C}_0 \cap \hat{R}_{\hat{Y}}^{-m} \mathcal{C}_1 ) \leq   c \, \hatmu (  \mathcal{C}_0 ) \ \hatmu (\mathcal{C}_1 ), \qquad \text{where}\ c=c(\hat{Y})\doteqdot  {\nu_{\hat{Y}}^3}/({\inf_{1\leq i\leq N} \muR (\Delta_{F_i} ) \inf_{1\leq i\leq N} \muR(\Delta_{E_i} ) }).
\ee
One can now conclude the proof of the Lemma by 
summing over all possible choices of  symplexes $\mathcal{C}_0, \mathcal{C}_1$ as above, namely by summing over all choices of cylinders $\mathcal{C}(C_1,k_1,l_1 )$ and $\mathcal{C}(C_2,k_2,l_2 )$ for $1\leq k_1,k_2,l_1, l_2 \leq N$.  
\end{proof}

\subsection{Full measure of the Ratner DC}\label{sec:fullmeasure}
In this section we will prove \cref{RDCfullmeasure}, by showing that the Ratner Diophantine condition for a suitable choice of the parameters $\tau, \xi $ and $\eta$ is satisfied by a full measure set of interval exchange transformations.

\begin{proof}[Proof of \cref{RDCfullmeasure}]

Set $\delta\doteqdot 1-\tau^{-1}$ and consider the set $\hat{Y}\doteqdot \hat{Y}_\delta$ given by \cref{AGY}, so that the corresponding acceleration $A\doteqdot A_{\hat{Y}}$ has the exponential tails property (see \cref{AGY}).  Since by the choice of $\tau$ we in particular have that $1< \tau<2$,  by  \cref{existencebalancedtimes}  (see also \cref{rk:sequencemixingDC}), there exists a subset $\hat{Y}' \subset \hat{Y}$ with $\hatmu(\hat{Y}')= \hatmu(\hat{Y})$ such that for any $\hat{T} \in \hat{Y}'$ the sequence $(n_l)_l$ of returns of $\hat{T}$ to $\hat{Y}'$ satisfies  the properties in the Definition \ref{def:mixingDC} of the  Mixing DC property with integrability power $\tau$ for some fixed $\overline{\ell} \in \mathbb{N}$ and $\eta >1$. Recall that $L=\overline{\ell}(1+ \left[ \log_d (2 \nu^2) \right])$ (see \eqref{RatnerDC:eq}). Fix $k\in \N$.


\smallskip
\textbf{Claim.}
There exists a constant $c=c(L)>0$ such that for every $R>0$ and every $0\leq J\leq L$ we have 
\begin{equation}\label{upperbound2}
\hatmu({\triple} \in \hat{Y} : \|A_{k}({\triple})\|\ldots\|A_{k+J}({\triple}) \|>R)<\frac{c}{R^{1-\delta}}.
\end{equation}
\smallskip

The proof of the Claim goes by induction on $J$. For $J=0$ by the integrability condition \eqref{integrabilityAGY} of $A$  (given by Theorem \ref{AGY}) and by invariance of $\hatmu$ under $\hat{\R}$, for any $R \in \mathbb{N}$ we have that
$$R^{1-\delta}\  \hatmu({\triple} \in \hat{Y} \, s.t.\, \| A_{k}({\triple}) \|> R ) \leq \sum_{i> R} i^{1-\delta}\  \hatmu({\triple} \in \hat{Y} \, s.t.\, \| A({\triple}) \|=i ) \leq  \int_{\hat{Y}_\delta  } \| A  \|^{1-\delta} \, d {\hatmu}\doteqdot c_0.$$
Assume that the Claim holds for $J<L$. We will show that it holds for $J+1$.
By summing the  QB property for the cocycle $A$ proved in \cref{QBforA}  over the set of $j$ such that $j> R/i$, we have that
\begin{align}\label{splt}
\begin{split}
&\hatmu({\triple} \in \hat{Y}\; :\; \|A_{k}({\triple}) \|\ldots\|A_{k+J+1}({\triple}) \|>R)\leq \\
&\sum_{i=1}^R \hatmu({\triple} \in \hat{Y}\; :\; \|A_{k+J+1} ({\triple}) \|=i\text{ and }\|A_{k}({\triple})\|\ldots\|A_{k+J}({\triple})\|>R/i)
\leq\\
& c_{\hat{Y}}\ \sum_{i=1}^R \hatmu({\triple} \in \hat{Y}\; : \;  \|A_{k+J+1}({\triple})\|=i) \, \hatmu({\triple} \in \hat{Y}\; :\; \|A_{k}({\triple})\|\ldots\|A_{k+J}({\triple})\|>R/i).
\end{split}
\end{align}
By the induction assumption (i.e. using the Claim for $J$ and $R/i$, $i=1\ldots R$) we get that
$$
\hatmu({\triple} \in \hat{Y} \,:\, \|A_{k}({\triple})\|\ldots\|A_{k+J}({\triple})\|>R/i)\leq c \left(\frac{i}{R}\right)^{(1-\delta)}, \qquad \text{for}\ i=1,\ldots,R.
$$

Therefore, denoting by  $\hat{Y}_i$ the set of  $\hat{T}= (\tau, \lambda, \pi) \in \hat{Y}$ such that $ \|A(\hat{T})\|=i$, we have by the integrability condition \eqref{integrabilityAGY} of $A$  (given by Theorem \ref{AGY}) and by invariance of $\hatmu$ under $\hat{\R}$
\begin{align*}
\hatmu(\hat{T} : \|A_{k}(\hat{T})\|\ldots\|A_{k+J+1}(\hat{T})\|>R)&\leq \frac{c_{\hat{Y}} c}{R^{1-\delta}} \sum_{i=1}^R i^{1-\delta}\, \hatmu(\hat{T}  : \|A_{k+J+1}(\hat{T})\|=i)\\ & 
 =\frac{c_{\hat{Y}} c}{R^{1-\delta}}\sum_{i=1}^R \int_{\hat{Y}_i}i^{1-\delta} \leq \frac{c_{\hat{Y}} c}{R^{1-\delta}}\int_{\hat{Y}}\|A \|^{1-\delta} \leq \frac{c_{\hat{Y}} c^2}{R^{1-\delta}}.
\end{align*}
This finishes the proof of Claim. Using \eqref{upperbound2} for $J=L$ and $R=k^\xi$, we get in particular  
\begin{equation}\label{upperbound}
\hatmu({\triple} \in \hat{Y} : \|A_{k}({\triple})\|\ldots\|A_{k+L}({\triple}) \|>k^\xi)<\frac{c}{k^{\xi(1-\delta)}}.
\end{equation}

Now we will show that for $\hatmu$-almost every  $\hat{T}\in \hat{Y}$,
 in each interval of the form $[j^2, (j+1)^2]$ ($j$ suff. large) there are at most $2L$ indexes $k\in[j^2,(j+1)^2]$ such that $\|A_{k}(\hat{T})\|\ldots\|A_{k+L}(\hat{T}) \|>k^\xi$. This follows by Corollary \ref{QBforAcor} as the events $\|A_{k}(\hat{T})\|\ldots\|A_{k+L}(\hat{T})\|>k^\xi$, $\|A_{l}(\hat{T})\|\|A_{l+L}(\hat{T})\|>l^\xi$ are almost independent for $l\notin \{k-L,\ldots, k+L\}$, i.e. there exists $c=c(L)>0$ such that
\be\label{almostindep}
\begin{split}
& \hat{\mu}({\triple} \in \hat{Y}  :   \|A_{k}(\hat{T})\|\ldots\|A_{k+L}({\triple})\|>k^\xi,  \|A_{l}({\triple})\|\ldots\|A_{l+L}({\triple})\|>l^\xi)\\ & \leq c \hat{\mu}({\triple} \in \hat{Y} \,:\, \|A_{k}({\triple})(\hat{T})\|\ldots\|A_{k+L}({\triple})\|>k^\xi)  \hat{\mu}({\triple} \in \hat{Y} \,:\,  \|A_{l}({\triple})\|\ldots\|A_{l+L}({\triple})\|>l^\xi) .
\end{split}
\ee
To show this notice that we can decompose the LHS analogously to \eqref{splt} and then, since $l\notin \{k-L,\ldots,k+L\}$, use Corollary \ref{QBforAcor}.



For $k\geq 1$, denote by $\mathscr{E}_k$ the set of $\hat{T} \in \hat{Y}$ such that $\|A_{k}(\hat{T})\|\ldots\|A_{k+L}(\hat{T})\|>k^\xi$. We are interested in 
$$
\mathscr{F}_j = \{{\triple} \in \hat{Y} \, s.t.\,  {\triple} \in  \mathscr{E}_k \text { holds for at least $2L+1$ numbers $k$ in }[j^2,(j+1)^2] \}, j\geq 1.
$$
By the definition of the set $\mathscr{F}_j$, it follows that 
\begin{equation}\label{bmn}
\mathscr{F}_j\subset \sum_{m,n\in[j^2,(j+1)^2], n\notin\{m-L,\ldots,m+L\}}\mathscr{E}_m  \cap \mathscr{E}_n .
\end{equation}

\noindent Now by \eqref{upperbound} and \eqref{almostindep}, we have that for any $m$ and $n\notin\{m-L,\ldots,m+L\}$, 
$$
\hat{\mu}(\mathscr{E}_m  \cap \mathscr{E}_n)\leq \left(\frac{1}{mn}\right)^{\xi(1-\delta)}\leq \frac{1}{j^{4\xi(1-\delta)}}.
$$
Therefore, by \eqref{bmn}, we have 
$$
\hat{\mu}(\mathscr{F}_j)\leq \frac{4j^2}{j^{4\xi(1-\delta)}} \ll \frac{1}{j^{1+1/100}} \text{ for }\delta<\frac{1}{16}\text{ and }\xi\geq>11/12.
$$
Hence, $\sum_{j\geq 1}\hat{\mu}(\mathscr{F}_j)< +\infty$. Thus, it follows from the Borel-Cantelli Lemma that there exists a set $\hat{Y}''\in \hat{Y}$ with $\hatmu( \hat{Y}'') = \hatmu( \hat{Y})$ such that  any  ${\triple} \in \hat{Y}''$ belongs only a finite number of times to $\bigcup_{j\geq 1}\mathscr{F}_j$. Denote by $n_{\triple} \in \mathbb{N}$ the last visit of ${\triple}$ to $\bigcup_{j\geq 1}\mathscr{F}_j$. Notice, that by the definition of $\mathscr{F}_j$ and $\mathscr{E}_k$, for $j\geq n_{\triple}^\frac{1}{2}$ and writing $A_\ell $ for $A_\ell(\hat{T})$, we have
$$
\sum_{\ell\in[j^2,(j+1)^2] \ \text{s.t.\ } \|A_\ell\|... \|A_{\ell+L}\|> \ell^\xi } \frac{1}{(\ell)^\eta}\leq \frac{2L}{j^{2\eta}}.
$$
Hence, since $\eta>1/2$ 
$$
\sum_{\ell\geq n_{\triple} \ \text{s.t.\ } \|A_\ell\|... \|A_{\ell+L}\|> \ell^\xi } \frac{1}{\ell^\eta}\leq \sum_{j\geq n_{\triple}^{\frac{1}{2}}}\frac{2}{j^{2\eta}}<+\infty.
$$
Therefore and since $(q_\ell)_{\ell\in\N}$ grows exponentially, we have the following (for some constants $C_T,c_T>0$)
$$
\sum_{\ell\in\N \ \text{s.t.\ } \|A_\ell\|... \|A_{\ell+L}\|> \ell^\xi } \frac{1}{(\log q_\ell)^\eta }\leq
C_T+ \sum_{\ell\geq n_{\triple} \ \text{s.t.\ } \|A_\ell\|... \|A_{\ell+L}\|> \ell^\xi } \frac{c_T}{\ell^\eta}<+\infty
$$
and so that equation \eqref{RatnerDC:eq} in the Ratner DC  holds for any  ${\triple}\in \hat{Y}''$.




Thus,  we showed so far that for every   ${\triple} \in \hat{Y}' \cap \hat{Y}''$, along the sequence $(n_l)_l $  of returns to $\hat{Y}$, both the Mixing DC for  $T \in \MDCt{\tau}{\overline{\ell}}{\nu}$  and \eqref{RatnerDC:eq} in the Ratner DC holds. Thus,  since $\hatmu(\hat{Y}')=\hatmu(\hat{Y}'') = \hatmu(\hat{Y})$,  
 the Ratner Diophantine condition  (see Definition \ref{def:ratnerDC}) holds for $\hatmu$-almost every ${\triple} \in \hat{Y}$.  
Since by Remark \ref{rk:sequence_indep_lift} this sequence does not eventually depend on the lift $\hat{T}=(\tau, \lambda, \pi)$ of the IET $T=(\lambda, \pi)$, but on $T$ only and $p_* \hatmu = \mu$, it follows that the set $Y' \subset Y$ of IETs for which such a sequence exists 
 has measure $\mu (Y') = Y$. Consider the set  $\mathcal{DC}_R$ of $T\in X$ such that there exists $\overline{n}$, for which $\R^{\overline{n}} T\in Y' $. We claim that this is the set of full measure of IETs which satisfy the desired Ratner Diophantine condition.  To see that $\mathcal{DC}_R$ has full  measure, it is enough to use ergodicity of $\Z$ and the fact that $\mu_{\Z} (Y' )>0$, remarking  that $\Z$ orbits are subsets of $\R$ orbits.  If  $T\in \mathcal{DC}_R$, the sequence  $n_l \doteqdot \widetilde{n}_{l}+\overline{n}$, where $\widetilde{n}_l$ is the sequence associated to $\R^{\overline{n}}T$ clearly also satisfy the Ratner DC definition property. 
Finally, the formulation in Proposition \ref{existencebalancedtimes} follows by absolute continuity of $\mu_{\Z}$ w.r.t. Lebesgue. 
\end{proof}

\section{Birkhoff sums of roof functions with logarithmic singularities}\label{sec:BS}
In this section, we state precise estimates on the growth of Birkhoff sums of the derivative of functions with logarithmic asymmmetric singularities under the Diophantine conditions introduced in \cref{se4}.  These estimates, as explained in the outline in \cref{sec:outline}, are a crucial tool to prove mixing and parabolic divergence of the corresponding special flows. In  \cref{sec:SWRRatner}
 we first recall a criterium which allows to reduce the proof of the SR-property to a statement about Birkhoff sums of the roof function. In Section~\ref{sec:derivatives} we state the estimates on Birkhoff sums of the derivatives proved in \cite{Ul:mix, Ra:mix} under the Mixing DC and deduce estimates in form which will be convenient for us to prove the SW-Ratner property in the next section. Finally, in Section~\ref{sec:summability} we show that the Ratner DC for a certain range of parameters implies the convergence of a series (see the Summability Condition in Definition \ref{RatDC2}) which is useful when proving parabolic divergence estimates.

\subsection{Ratner properties for special flows (over IETs) via Birkhoff sums}\label{sec:SWRRatner}

In this section we recall a criterion which implies the SR-property in the class of special flows over an ergodic automorphism. It was studied in~\cite{FK:mul} in the case the base automorphism is an irrational rotation and, in the general case in~\cite{Ka-Ku}.
\begin{prop}\label{cocy}
Let $(X,d)$ be a $\sigma$-compact metric space, $\mathcal{B}$ the $\sigma$-algebra of Borel subsets of $X$, $\mu$ a Borel probability measure on $(X,d)$. Let $T$ be an ergodic automorphism acting on $(X,\mathcal{B},\mu)$ and let $f\in L^1(X,\mathcal{B},\mu)$ be a positive function bounded away from zero. Let $\mathcal{T}=(T_t^f)_{t\in\R}$ be the corresponding special flow. Let $P=\{-1,1\}$. Assume that
\begin{center}
\parbox{0.8\textwidth}{
for every $\vep>0$ and $N\in \N$ there exist $\kappa=\kappa(\vep)>0$, $\delta=\delta(\vep,N)>0$ and a set $X'=X'(\vep,N)$ with $\mu(X')>1-\vep$, such that for every $x,y\in X'$ with $0<d(x,y)<\delta$ there exist $M=M(x,y),L=L(x,y)\geq N$ with $\frac{L}{M}\geq \kappa$ and $p=p(x,y)\in P$
}
\end{center}
such that one of the following holds:
\begin{enumerate}[(i)]
\item\label{posi}
$d(T^nx,T^ny)<\vep\;\text{and }\; |S_{n}(f)(x)-S_{n}(f)(y)-p|<\vep\text{ for every }n\in [M,M+L]$,
\item\label{nega}
$d(T^{-n}x,T^{-n}y)<\vep\;\text{and }\;|S_{-n}(f)(x)-S_{-n}(f)(y)-p|<\vep\text{ for every }n\in [M,M+L]$.
\end{enumerate}
Then $\mathcal{T}$ has the SR-property.
\end{prop}


\noindent The following Lemma provides conditions to  verifying~\eqref{posi} and~\eqref{nega} in Proposition \ref{cocy} above in case of special flows over IETs.
\begin{lemma}\label{njz}
Let $T\colon I\to I$ be an IET. Fix $\vep>0$ and $N\in\N$ and assume that for some $x,y\in I$ with $x<y$, $y-x<\vep$ and some $M,L \geq N$ we have:
\begin{equation}\label{betai}
\{\ell_\alpha : \alpha\in\mathcal{A}\}\cap (\bigcup_{i=0}^{M+L}T^i[x,y])=\emptyset,
\end{equation}
\begin{equation}\label{cont}
\parbox{0.8\textwidth}{
$[M,M+L]\times[x,y]\ni (n,\theta)\mapsto sign (S_n(f')(\theta))$ is constant for every $\theta\in [x,y]$ and every $n\in[M,M+L]$,
}
\end{equation}
\begin{equation}\label{der}
(1-\vep)(y-x)^{-1}<|S_n(f')(\theta)|<(1+\vep)(y-x)^{-1}.
\end{equation}
Then $x,y$ satisfy \eqref{posi} (with $M,L,\vep$). Analogously, if 
\begin{equation}
\{\ell_\alpha : \alpha\in\mathcal{A}\}\cap (\bigcup_{i=1}^{M+L}T^{-i}[x,y])=\emptyset,
\end{equation}
\begin{equation}
\parbox{0.8\textwidth}{
$[M,M+L]\times[x,y]\ni (n,\theta)\mapsto sign (S_{-n}(f')(\theta))$ is constant for every $\theta\in [x,y]$ and every $n\in[M,M+L]$,
}
\end{equation}
\begin{equation}
(1-\vep)(y-x)^{-1}<|S_{-n}(f')(\theta)|<(1+\vep)(y-x)^{-1},
\end{equation}
then $x,y$ satisfy \eqref{nega}.
\end{lemma}

%
%
%
\begin{proof} We will assume that \eqref{betai}, \eqref{cont} and \eqref{der} hold and show \eqref{posi} (the proof of the other part of the assertion is analogous). Notice first that by \eqref{betai},
for every $n\in [M,M+L]$, 
$$
d(T^nx,T^ny)=d(x,y)=y-x<\vep,
$$
so the first part of \eqref{posi} holds. Moreover, by \eqref{betai} and the fact that $f\in \mathcal{C}^2(\mathbb{T}\setminus\{\ell_\alpha : \alpha\in\mathcal{A}\})$, for every $n\in [M,M+L]$, we have 
\begin{equation}\label{ivt}
S_{n}(f)(x)-S_{n}(f)(y)=(x-y)S_{n}(f')(\theta_n) \text{ for some }\theta_n\in[x,y].
\end{equation}
By~\eqref{cont}, we can assume WLOG that for every $n\in [M,M+L]$ and every $\theta\in[x,y]$, $S_{n}(f')(\theta_n)>0$ (the opposite case is analogous). Then, by \eqref{ivt}, for every $n\in[M,M+L]$, using \eqref{der}, we obtain
$$
|S_{n}(f)(x)-S_{n}(f)(y)+1|= |(y-x)S_{n}(f')(\theta_n)-1|\leq \vep,
$$
so \eqref{posi} holds with $p=-1\in P$.
\end{proof}

%

\subsection{Growth of Birkhoff sums of derivatives}  \label{sec:derivatives}
Throughout this section, let $T$  be an IET and  let $f \in AsymLogSing(T)$.  We will also assume that IET $T$ satisfies Mixing DC.
Before stating quantative results on the growth of the Birkhoff sums $S_r(f\rq{}\rq{})(x)$ of the derivative $f'$ of the function over $T$, we will introduce some notation and definitions. 
Define the following sequence $\{ \sigma_{\ell} \}_{\ell \in \mathbb{N}}$, used in the proof of Proposition \ref{mpd} below as a threshold to determine whether $r$ is closer to $q_l$ or to $q_{l+1}$.
Let $\tau'$ be such that $\tau/2 < \tau' < 1$, where $\tau$ is the Diophantine exponent in \eqref{integrability} given by Proposition \ref{existencebalancedtimes} and $\tau '$ is well defined since $\tau <2$. Let
\be \label{sigmadef}
\sigma_\ell = \sigma_\ell (T) \doteqdot \left( \frac{ \log \| A_{\ell} \|} {\log q_\ell}  \right)^{\tau'} ,  \qquad \frac{\tau}{2} < \tau' < 1,
\ee
Clearly $\sigma_{\ell}$ depends on the IET $T$ we start with, since the sequence $\{n_\ell\}_{\ell \in \mathbb{N}}$ does.

The following set is the set of points which one needs to \emph{throw away} in order to get estimates on Birkhoff sums for times between $q_\ell$ and $q_{\ell+1}$ (see Proposition \ref{mpd}). These are points that get too close to the singularities, so that the corresponding Birkhoff sums can be arbitrarily large. 
\begin{defn} Let $\Sigma^+_{\ell}=\Sigma^+_\ell (T)$ be the following set, where $[\cdot]$ denotes the fractional part:
\be \label{badsetforBSu}
\Sigma^+_\ell (T) 
\doteqdot \bigcup_{\alpha\in\cA} \bigcup_{i=0}^{[\sigma_{\ell} q_{\ell+1}]} T^{-i} [-\sigma_\ell I^{(n_\ell)}+\ell_\alpha,l_\alpha+ \sigma_\ell I^{(n_\ell)}].
\ee
\end{defn}
\begin{rem}\label{mes.sigma}
Notice that since $n_\ell$ is a $\nu$-balanced time (see Remark \ref{balancedcomparisons})
$$
\lambda(\Sigma^+_\ell)\leq 2|\cA|\nu\sigma_\ell^2\frac{q_{\ell+1}}{q_\ell}\leq 2|\cA|\nu^2\sigma_\ell^2\|A_\ell\|.
$$
\end{rem}

 For $\alpha\in \cA$, let  us denote by  $u_\alpha(x)=\frac{1}{x-r_\alpha}$ and by $v_\alpha(x)=\frac{1}{l_\alpha-x}$. Then the following functions $U(r,x)$ and $V(r,x)$ give the largest contribution in a Birkhoff sum of length $r$ starting from $x$ given respectively to visits of the orbit which are close to a singularity form the right or from the left respectively:
$$U(r,x)= \max_{\alpha\in\cA}\max_{0\leq i<r}u_\alpha(T^ix),
\;\;\;\; V(r,x)=\max_{\alpha\in\cA}\max_{0\leq i<r}v_\alpha(T^ix).$$

For points outside the set $\sigma_\ell$ one has the following precise estimates, which were proved by the third author in \cite{Ul:mix} for $f$ with one asymmetric logarithmic singularity and then extended by Ravotti in \cite{Ra:mix}  to general $f \in AsymLog(T)$.

\begin{prop}[Growths of derivatives, \cite{Ul:mix, Ra:mix}]\label{mpd} 
Let $T$ be an IET which satisfies the Mixing DC along a subsequence $(n_\ell)$ of induction times. Let $f\in AsymLogSing(T)$. Assume WLOG that $C^->C^+$. 
There exists $M>0$ such that for every $\vep>0$ there exists $\ell_1\in \N$  such that for every $\ell\geq \ell_1$
$$
q_{\ell}\leq r<q_{\ell+1}\text{    and    } x\notin \Sigma_\ell^+,
$$
we have
$$
(C^--C^+-\vep^2)r\log r\leq S_r(f')(x)\leq (C^--C^++\vep^2)r\log r +M(U(r,x)+V(r,x)).
$$
\end{prop}

We present now some estimates on Birkhoff sums (Lemma \ref{prty}) which can be deduced by Proposition \ref{mpd} and are given in a form which will  be convenient for us to prove the SW-Ratner property. In order to prove the quantitative estimates on  
Birkhoff sums in Lemma \ref{prty}, we need the following technical Lemma. 

\begin{lemma}\label{techest} 
If $\tau,\tau',\xi,\eta$ are such that
$$
   \tau\in \left(1,{16}/{15}\right),\;\;\;\tau'\in \left({15}/{16},1\right),\;\;\;
	 \eta\in \left(3/4,2\tau '-\tau \right),\;\;\;
	 \xi\in \left(\max(11/12,\tau ' \eta),\tau '\right),
$$
then  the following hold:
\begin{equation}\label{nr1}
 \lim_{\ell\to+\infty}\sigma_\ell(\log q_\ell)^\xi=0;
\end{equation}

\begin{equation}\label{nr2}
\lim_{\ell\to+\infty}\sigma_\ell^{2-\eta}\ell^\tau=0;
\end{equation}
\begin{equation}\label{nr3} 
 \lim_{\ell\to+\infty}\frac{\log \|A_\ell\|}{(\log q_\ell)^\xi\sigma^\eta_\ell}=0;
\end{equation}
\end{lemma}

\begin{proof} Notice first that there exists a constant $c>0$ such that 
\begin{equation}\label{exp.ql}
q_\ell\geq c^\ell.
\end{equation}
For sufficiently large $\ell\in \N$, by the integrability condition \ref{integrability}, $\|A_\ell\|\leq \ell^\tau$. Since $\xi<\tau'$ we get (see \eqref{sigmadef})
$$\sigma_\ell(\log q_\ell)^\xi\leq \frac{\tau \log \ell}{(\log q_\ell)^{\tau'-\xi}}.
$$
This and \eqref{exp.ql} give \eqref{nr1}.  To prove \eqref{nr2} notice that by \eqref{sigmadef}, the integrability condition \eqref{integrability} and \eqref{exp.ql} for sufficiently large $\ell$ we have (for some constant $C>0$)
\begin{equation}\label{firstest}
\sigma_\ell^{2-\eta}\ell^\tau \leq C \frac{(\tau \log \ell)^{(2-\eta)\tau'}}{\ell^{(2-\eta)\tau'-\tau}}.
\end{equation}
Since $\eta\tau'<\eta<2\tau'-\tau$ by the assumptions on $\eta,\tau,\tau'$, we have that $(2-\eta)\tau'-\tau>0$.
Thus, \eqref{nr2} follows from \eqref{firstest}.  
Next, notice that
$$
\frac{\log \|A_\ell\|}{(\log q_\ell)^\xi\sigma^\eta_\ell}\leq 
\frac{\tau \log \ell}{(\log q_\ell)^{\xi-\eta\tau'}},
$$
this finishes the proof of \eqref{nr3} since, by assumptions, 
$\xi-\eta\tau'>0$ and \eqref{exp.ql} holds.
\end{proof}

\smallskip
\noindent \textbf{Assumption.} From now on we make a standing assumption\footnote{We remark that the condition on $\eta$ we ask for here is on purpose more restrictive of the condition required in Lemma \ref{techest}, since this will be useful later.} on $\tau,\tau',\xi,\eta$, namely
\begin{equation}\label{asucons}
   \tau\in \left(1,\frac{16}{15}\right),\;\;\;\tau'\in \left(\frac{15}{16},1\right),\;\;\;
	 \eta\in \left(3/4,\tau'(2\tau '-\tau)\right),\;\;\;
	 \xi\in \left(\max(11/12,\tau ' \eta),\tau '\right).
\end{equation}
One can verify that all intervals are indeed non-empty, so that such a choice exists.
 Next lemma allows us to control forward (backward) Birkhoff sums of the derivative for points whose forwards (backwards) orbits do not come to close to singularities. It says that if orbit of a point stays away from singularity, then Birkhoff sums of the derivative are (up to a small error) equal to the main contribution (comming from sums along the orbit). The main tool is Proposition \ref{mpd}.  
\begin{lemma}[Growth of derivatives for SR-property]\label{prty} For every $\vep>0$ there exists $\ell_1\in \N$  such that for every $\ell\geq \ell_1$ we have the following: 
\begin{enumerate}[(A)]
\item if
\begin{equation}\label{fra}U(q_{\ell+1},x), V(q_{\ell+1},x)\leq 2 q_\ell(\log q_\ell)^\xi ,
\end{equation}
then for $q_\ell\leq r<q_{\ell+1}$ we have
\begin{equation}\label{est.ex}
(C^--C^+-\vep^2)r\log r\leq S_r(f')(x)\leq (C^--C^++\vep^2)r\log r;
\end{equation}
\item if
\begin{equation}\label{fra1}U(q_{\ell+1},T^{-q_{\ell+1}}x), V(q_{\ell+1},T^{-q_{\ell+1}}x)\leq 2 q_\ell(\log q_\ell)^\xi ,
\end{equation}
then for $h^{(n_l)}\leq r<h^{(n_{l+1})}$ we have
\begin{equation}\label{est.ex2}
(C^--C^+-\vep^2)r\log r\leq -S_{-r}(f')(x)\leq (C^--C^++\vep^2)r\log r.
\end{equation}
\end{enumerate}
\end{lemma}

\begin{proof}
Fix $\vep>0$. Let us first show (A). From Lemma \ref{techest}, for sufficiently large $\ell$ (depending on  $\vep$) we have 
\begin{equation}\label{jetd} (\log q_\ell)^\xi <
\frac{\vep}{\sigma_\ell}
\end{equation}
Let $M$ be the constant given by Proposition \ref{mpd}. Notice that for every $i\in [0, \sigma_\ell q_{\ell+1}]\subset [0,q_{\ell+1}]$ and every $\alpha\in\cA$, using \eqref{fra} and \eqref{jetd} (since $n_\ell$ is a balanced time and $\vep$ is small)
$$
d(T^ix,l_\alpha)\geq \frac{1}{2 q_\ell(\log q_\ell)^\xi}\geq \sigma_\ell I^{(n_\ell)}.
$$
Therefore  $x\notin \Sigma_\ell^+$ (see \eqref{badsetforBSu}).  
Moreover, if $\ell$ is sufficiently large (since $\xi<1$), using \eqref{fra}, we have for every $r\in [q_\ell,q_{\ell+1})$,
\begin{equation}\label{noi}
\frac{\vep}{2M}r\log r\geq\frac{\vep}{2M}q_\ell\log q_\ell\geq 2q_\ell(\log q_\ell)^\xi\geq U(q_{\ell+1},x)\geq U(r,x), 
\end{equation}
(the same estimates hold for $V(r,x)$).

 Since $x\notin \Sigma_\ell^+$,  we can use estimates in Proposition \ref{mpd} for $\frac{\vep}{2}$ and using \eqref{noi} we know that for $\ell\geq \ell_1$ ($\ell_1$ depending on $\vep$), 
$$
(C^--C^+-\vep)r\log r\leq S_r(f')(x)\leq (C^--C^++\vep)r\log r.
$$
This gives \eqref{est.ex}.

Now we show \eqref{est.ex2}. Fix $r\in [q_\ell,q_{\ell+1})$. We will show that $T^{-r}x\notin \Sigma_\ell^+$. For this aim note first that if $\ell$ is sufficiently large, then 
\begin{equation}\label{mjf}
\sigma_\ell q_{\ell+1}< q_\ell.
\end{equation}
Indeed, by the definition of $q_{\ell+1}$ it follows that for every $z\in\mathbb{T}$, 
\begin{equation}\label{jv} 
\{z,\ldots,T^{q_{\ell+1}}z\}\cap I^{n_{\ell+1}}\neq \emptyset.
\end{equation}
In particular, for $z=T^{-q_{\ell+1}}x$, by the fact that $n_{\ell+1}$ is a balanced time, \eqref{jv} and \eqref{fra1}, it follows that
$$
\nu q_{\ell+1}\leq \frac{1}{I^{(n_{\ell+1})}}\leq U(q_{\ell+1},T^{-q_{\ell+1}}x)\leq 2 q_{\ell}(\log q_{\ell})^\xi.
$$
This and \eqref{jetd} gives \eqref{mjf} for sufficiently large $\ell$. By \eqref{mjf} it follows that 
\begin{equation}\label{sim}
\Sigma_\ell^+\subset\bigcup_{\alpha\in\cA} \bigcup_{i=0}^{q_{\ell}-1}T^{-i}
[-\sigma_\ell I^{(n_\ell)}+l_\alpha,l_\alpha+\sigma_\ell I^{(n_\ell)}].
\end{equation}
Notice however, that by \eqref{fra1}, for every $i\in [0,q_\ell)$ and every $\alpha\in\cA$ ($q_{\ell+1}>r\geq q_\ell$), we have

\begin{multline*}d(T^i(T^{-r}x),l_\alpha)=d(T^{i-r}x,l_\alpha)
\geq \sup_{-q_{\ell+1}<s<0}d(T^sx,l_\alpha)\geq\\
 \frac{1}{\max(U(q_{\ell+1},T^{-q_{\ell+1}}x),
V(q_{\ell+1},T^{-q_{\ell+1}}x))}\geq \frac{1}{2q_{\ell}(\log q_{\ell})^\xi}\geq
\sigma_\ell I^{(n_\ell)},\end{multline*}
the last inequality by \eqref{jetd} (and balance).
Therefore, 
$$T^{-r}x\notin \bigcup_{\alpha\in\cA} \bigcup_{i=0}^{q_{\ell}-1}T^{-i}
[-\sigma_\ell I^{(n_\ell)}+l_\alpha,l_\alpha+\sigma_\ell I^{(n_\ell)}],
$$
and by \eqref{sim}, $T^{-r}x\notin \Sigma_\ell^+$. Notice that
\begin{equation}\label{sie1}
U(r,T^{-r}x)\leq U(q_{\ell+1},T^{-q_{\ell+1}}x)\text{  and  }V(r,T^{-r}x)\leq V(q_{\ell+1},T^{-q_{\ell+1}}x).
\end{equation}
Moreover, by the cocycle identity $-S_{-r}(f')(x)=S_r(f')(T^{-r}x)$. Therefore, we use Proposition \ref{mpd} for $r$ and $T^{-r}x$ and use \eqref{sie1} and \eqref{fra1}, to get
$$
(C^--C^+-\vep)r\log r\leq S_{r}(f')(T^{-r}x)\leq (C^--C^++\vep)r\log r.
$$
This finishes the proof of \eqref{est.ex2}.
\end{proof}

\subsection{Ratner Summability Condition}\label{sec:summability}
In this section we will deduce from the Ratner DC a summability condition, which is very convenient when studying Birkhoff sums since it allows to estimate the measure of the set of points which we need to throw away to get the control of the growth of Birkhoff sums to prove the SR-condition (see the heuristic discussion in the outline in \cref{sec:outline}).

\begin{defn}\label{RatDC2} We say that an IET $T$ that satisfies the mixing DC with power $\tau$ satisfies the {\em Ratner Summability Condition} (or for short, the \emph{summability condition}) with 
 exponents $(\tau', \xi', \eta')$ if 
\begin{equation}\label{DCcon}
\sum_{\ell\notin K_T}\sigma_\ell^{\eta'}<+\infty,\qquad \text{where} \quad  K_T\doteqdot \Big\{\ell\in\N:\; q_{\ell+L}\leq
\frac{q_{\ell}}{\sigma_\ell^{\xi'}}\Big\}.
\end{equation} 
\end{defn}
We remark that the dependence on $\tau'$ in the definition is through $\sigma_\ell$ which appears in the definition of $K_T$, since as one can see in \eqref{sigmadef}, $\sigma_\ell$ depends on $\tau'$.


\begin{lemma}\label{equiv:DC} 
Under the assumption \eqref{asucons} on the parameters 
 $\tau,\tau', \xi, \eta$, if 
$T \in \RDCt{\tau}{\xi}{\eta}$, 
 for any  $\xi'>\frac{\xi}{\tau'}$ and $\eta'>\frac{\eta}{\tau'}$, $T$ satisfies the Ratner Summability Condition with  exponents $(\tau', \xi', \eta')$. 
\end{lemma}
\begin{proof}
Notice first that, since $\log q_\ell\geq c\ell$ (for some constant $c>0$), for any $\epsilon>0$, $\log\|A_\ell \|/(\log q_\ell)^\epsilon$ tends to zero as $\ell $ grows by the integrability condition \eqref{integrability}. Thus,  
recalling the definition \eqref{sigmadef} of $\sigma_\ell$, we have that for any $\tau''< \tau'$, for any $\ell$ sufficiently large,
\begin{equation}\label{sigmaest}
{\sigma_\ell}= \frac{ {\left( \log \| A_{\ell} \|\right)}^{\tau'}}{(\log q_\ell)^{\tau'-\tau''}} \frac{1}{(\log q_\ell )^{\tau''}}
\leq 
   \frac{1}{(\log q_\ell)^{\tau''}}.
\end{equation}
Notice that $\xi'> \frac{\xi'\tau'+\xi}{2}> \xi$. Therefore and by \eqref{sigmaest} we have 
(for $\ell$ sufficiently large), we have
$$\frac{1}{\sigma_\ell^{\xi'}}\geq (\log q_\ell)^{\frac{\xi'\tau'+\xi}{2}}> \ell^\xi.
$$
Therefore, if $\ell\notin K_T$, then, writing $A_\ell$ for  $A_\ell(T)$ and taking $L$ as in \eqref{RatnerDC:eq}, we have that
\begin{equation}\label{sumovercomparision}
\|A_\ell \| \|A_{\ell+1} \|\cdots \|A_{\ell+L}\|\geq \frac{q_{\ell+L}}{q_{\ell}} > \frac{1}{\sigma_\ell^{\xi'}}> \ell^\xi.
\end{equation}
Moreover, since $\eta'\tau'>\eta$, thanks to \eqref{sigmaest}  we also have that $\sigma_\ell^{\eta'}<1/(\log q_\ell)^\eta$. Hence
\begin{equation}\label{finiteseries}
\sum_{\ell\notin K_T}\sigma_\ell^{\eta'}<
\sum_{\ell\in\N \ \text{s.t.\ } \|A_\ell\| \|A_{\ell+1}\|\cdots \|A_{\ell+L}\|> \ell^\xi } \sigma_\ell^{\eta'} <
\sum_{\ell\in\N \ \text{s.t.\ } \|A_\ell\| \|A_{\ell+1}\|\cdots \|A_{\ell+L}\|> \ell^\xi } \frac{1}{(\log q_\ell)^\eta } < + \infty.
\end{equation}
\end{proof}

The Ratner Summability Condition, as a corollary of Lemma \ref{techest}, implies in particular that the series of measures of the sets $\Sigma^+_\ell(T)$ of points not controlled in Proposition \ref{mpd} (see Definition \ref{badsetforBSu}) is \emph{summable} if one restricts the sum only to $ \ell \in \N$ such that  $\|A_{\ell} \| \| A_{\ell+1}\| \cdots\|A_{\ell+L}\| > \ell^\xi$, which will be useful to prove the SR-condition (see the outline in the introduction).

\begin{cor}\label{summabilitySigmas}
Under the assumptions of Lemma \ref{equiv:DC}, 
 we have that
\begin{equation}\label{sersum}
\sum_{\ell\notin K_T}  \lambda {\left(\Sigma_\ell^+(T) \right)}\leq 
\sum_{\ell\notin \widetilde{K}_T}
 \lambda {\left(\Sigma_\ell^+(T) \right)}<\infty,
\end{equation}
where $K_T$ is defined as in \eqref{DCcon} and $\widetilde{K}_T\doteqdot  \{ \ell \in \N:\; \|A_{\ell} \| \| A_{\ell+1}\| \cdots\|A_{\ell+L}\| \leq \ell^\xi\}$.
\end{cor}
\begin{proof} 
The inequality between the two series in \eqref{sersum} follows from \eqref{sumovercomparision}, which shows that 
if $\ell \notin K_T$ then 
$\ell \notin \widetilde{K}_T$.  Let us choose $\eta'$ such that, in addition to $\eta'>\eta/\tau'$, it also satisfies $ \eta' \in (3/4,2\tau '-\tau)$, which is possible since $\eta/\tau'< 2\tau'-\tau$ by the choice of $\eta$ (recall \eqref{asucons}).  Then,   
by \eqref{nr2} in Lemma \ref{techest} applied to $\eta'$, for sufficiently large $\ell$ we have
$\sigma_\ell^2\ell^\tau\leq {\sigma_\ell^{\eta'}}/{2}$.  
Therefore, since for sufficiently large $\ell$ we also have that $\|A_\ell\|\leq \ell^\tau$ by the Mixing DC (see Definition \ref{def:mixingDC}), by Remark \ref{mes.sigma}  there exists a constant $C>0$ such that 
$$ 
\sum_{\ell\notin \widetilde{K}_T} \lambda \left(\Sigma_\ell(T)\right) \leq  \sum_{\ell \notin \widetilde{K}_T}  2|\cA|\nu^2 \sigma_\ell^2\|A_\ell\| 
\leq C \sum_{\ell\notin \widetilde{K}_T} \sigma_\ell^{\eta'},
$$
which is finite by \eqref{finiteseries}. 
\end{proof}

One can show as a corollary of Lemma \ref{equiv:DC} that for a suitably chosen range of exponents $\tau'$,$\xi'$ and $\eta'$ the set of IET's satisfying the Summability condition  with exponents $(\tau',\xi',\eta')$ has full measure. 


\begin{cor}\label{RatvsPar} Let $\tau\in (1,16/15)$, $\tau'\in(15/16,1)$,  $\xi'>99/100$, $\eta'>3/4$.  For each irreducible combinatorial datum $\pi$ and for Lebesgue a.e.\  $\underline{\lambda} \in  \Delta_{d}$, the corresponding IET $T= (\underline{\lambda}, \pi) $ satisfies the Summability Condition with exponents $(\tau',\xi',\eta')$. 
\end{cor}
\begin{proof} Take any $\xi, \eta$ such that $11/12< \xi< \xi' \tau' $ and $ 1/2< \eta< \eta' \tau'$, which is possible by the assumptions on the parameters.  
Consider any IET $T$ in $RCD (\tau, \xi ,\eta)$. The set of such $T$ has full measure, see Proposition \ref{RDCfullmeasure}. Then, by Lemma \ref{equiv:DC}, $T$ satisfies the Summability Condition with exponents $(\tau', \xi', \eta')$. This finishes the proof.
\end{proof}
\begin{rem}\label{DC12} While the Ratner DC as formulated  is  useful when trying to prove that the set of IETs satisfying the Ratner DC for suitably chosen parameters has full measure (see Proposition \ref{RDCfullmeasure}), 
the Summability Condition  is useful for computations concerning quantitative divergence of nearby points (since it allows to throw sets which are tails of the converging series given by Corollary \ref{summabilitySigmas}).  
From now on, in the rest of the paper we will only use the Summability Condition and hence we will consider only IETs which satisfies the Ratner DC for a range of parameters which imply the Summability Condition. 
By Corollary \ref{RatvsPar}, this set of IETs has full measure for some choice of exponents. 
\end{rem}

From now on, in order to  to simplify notation, we will use exponents $(\tau, \xi, \eta)$ (instead of $(\tau', \xi', \eta'$) in the definition of the Summability Condition.

\section{Proof of the switchable Ratner property}\label{sec:proof}

In this section we will prove that special flows over IETs under functions with logarithmic asymmetric singularities have the  switchable Ratner property when the base IET satisfies the Ratner DC with an appropriate choice of parameters (Theorem \ref{thm:Ratner_special_flows}). In  Section~\ref{sec:backward_forward}, we first use balanced Rauzy-Veech times to show that one can control the distance of orbits of most points from the singularities either in the past or in the future. The proof of Theorem \ref{thm:Ratner_special_flows} is then given in \cref{sec:deducingRatner}. Finally, in \cref{sec:conclusions} we deduce the other results stated in the introduction. 

\subsection{Control of either backward or forward orbits distance from singularities}\label{sec:backward_forward}
In this section we show that balanced positive times of the Rauzy Veech induction allows us to control the distance of most  orbits from the discontinuities either \emph{backward} or \emph{forward}, i.e. points who get too close to a discontinuity in the future, do not get too close in the past (where too close is quantified in Proposition \ref{forbac}). This will provide a key step to prove the switchable Ratner property, since according to whether backward or forward orbits stay far from singularities, we will be able to verify the parabolic divergence estimates  either in the future of in the past. The main proposition that we prove in this section is Proposition \ref{forbac} stated here below.

Let us first recall that  $I_\alpha = [l_\alpha, r_\alpha)$ denote the intervals exchanged by $T=(\lambda, \pi)$. 
Given two sets $E, F \subset [0,1]$ let us denote by $d(E, F)$ the distance between the two sets, given by
\begin{equation}\label{dsets}
d(E, F) = \inf \{ |x-y|, \quad x \in E, \ y \in F\}.
\end{equation}
\begin{prop}[Backward or forward control]\label{forbac}
Let $T$ be an IET which satisfies the Keane condition and  let  $\{n_\ell\}_{\ell\in\N}$ be a sequence of $\nu$-balanced  induction times for $T$ such that $\{n_{\overline{\ell}k}\}_{k\in\N}$ is a positive sequence of times for some $\ov{\ell}\in\N$. Let $q_\ell$ denote the maximal height of towers of step $n_\ell$ (see \eqref{CF_like_notation}). 

There exists an integer $L\geq 1$, explicitly given by $L=\overline{\ell} (1+ \left[ \log_d (2 \nu^2) \right])$ (where $[\cdot]$ denotes the integer part and $\log_d$ the logarithm in base $d$), and $c>0$ such that for any $\vep>0$, there exists $\ell'=\ell'(\vep)\geq 1$ such that for $\ell\geq \ell'$ and $x\not \in [0,\vep/8)\cup (1-\vep/8,1)$, at least one of the following holds:
\begin{align}
d(\{l_\alpha,r_\alpha : \alpha\in\cA \}, \{T^i x : 0\leq i <q_\ell \})>\frac{c}{q_{\ell+L}},\label{j41}\\
d(\{l_\alpha,r_\alpha : \alpha\in\cA\}, \{T^i x : -q_\ell \leq i< 0 \})>\frac{c}{q_{\ell+L}}, \label{j42}
\end{align}
where $d(\cdot, \cdot)$ denotes the distance between sets defined in \eqref{dsets}.
\end{prop}

We will first state some auxiliary definitions and two Lemmas and then use them to prove Proposition \ref{forbac}. 
Given $T=(\lambda, \pi)$, remark that the interval endpoints $0$ and $1$ can be written as $0=l_{\alpha_{1,t}}$ and $1=r_{\alpha_{d,t}}$ where  $\alpha_{1,t}$ and $\alpha_{d,t}$ are defined respectively to be the letters in $\mathcal{A}$ such that $\pi_t(\alpha_{1,t})=1$ and $\pi_t(\alpha_{d,t})=d$,  so that $I_{\alpha_{1,t}}$  and $I_{\alpha_{d,t}}$ are respectively the first and last interval before the exchange. Moreover, if  $\alpha_{1,b}$, $\alpha_{d,b}$  are such that $\pi_b(\alpha_{1,b})=1$ and $\pi_b(\alpha_{d,b})=d$,  $I_{\alpha_{1,b}}$  and $I_{\alpha_{d,b}}$ are such that their image under $T$ are respectively the first and last interval after the exchange.
 
\begin{rem}\label{singularitiesT}
We remark that  for any $\alpha \in \mathcal{A}$ such that $\alpha \neq \alpha_{1,t}$, $l_\alpha = r_{\beta}$ where $\beta$ is such that $\pi_t(\beta)= \pi_t(\alpha)-1$, so that $I_\beta$ is the interval which preceeds $I_\alpha$ before the exchange. Similarly, for any $\alpha \in \mathcal{A}$ such that $\alpha \neq \alpha_{d,t}$, $r_\alpha = l_{\beta}$ where $\beta$ is such that $\pi_t(\beta)= \pi_t(\alpha)+1$, so that $I_\beta$ is the interval which follows $I_\alpha$ before the exchange. 
\end{rem}

Given an IET $T^{(n)}= (\lambda^{(n)}, \pi^{(n)})$ in the Rauzy-Veech orbit of $T$, we will denote by $I^{(n)}_\alpha = \left[ l_{\alpha,t}^{( n)}, r_{\alpha, t}^{(n)} \right)$ for $\alpha \in \mathcal{A}$  the intervals  exchanged by $T^{(n)}$ and we will denote their images under $T^{(n)}$ by $\left[ l_{\alpha,b}^{( n)}, r_{\alpha, b}^{(n)} \right)$. Explicitely, the endpoints are given by 
\begin{align}\label{disc_Tn}
&l_{\alpha,t}^{(n)}: = \sum_{\pi^{(n)}_t(\beta)<\pi_t^{(n)}(\alpha)}\lambda^{(n)}_\beta ,  &r_{\alpha,t}^{(n)}: = \sum_{\pi^{(n)}_t(\beta)\leq \pi_t^{(n)}(\alpha)}\lambda^{(n)}_\beta; \\
&  l_{\alpha,b}^{(n)}: = \sum_{\pi^{(n)}_b(\beta)<\pi_b^{(n)}(\alpha)}\lambda^{(n)}_\beta , & r_{\alpha,b}^{(n)}: = \sum_{\pi^{(n)}_b(\beta)\leq \pi_b^{(n)}(\alpha)}\lambda^{(n)}_\beta. 
\end{align}
We will also use the notation $\alpha^{(n)}_{1,t}$, $\alpha^{(n)}_{d,t}$, $\alpha^{(n)}_{1,b}$, $\alpha^{(n)}_{d,b}$  for the letters such that 
\begin{equation}\label{firstlastdef}
\pi_t^{(n)}(\alpha^{(n)}_{1,t})=1, \qquad \pi_t^{(n)}(\alpha^{(n)}_{d,t})=d, \qquad \pi_b^{(n)}(\alpha^{(n)}_{1,b})=1, \qquad \pi_b^{(n)}(\alpha^{(n)}_{d,b})=d. 
\end{equation}


A crucial step is given by the following Lemma, which is a small modification of  Corollary C.2 (see also Lemma C.1) in the Appendix of \cite{HMU:lag}. For completeness, we  include its short proof in the Appendix~\ref{appendix:IETs}.

\begin{lemma}[see Corollary C.2 in \cite{HMU:lag} and Appendix \ref{appendix:IETs}]\label{lemmaHMU}
Let $T$, $\{n_\ell\}_{\ell\in\N}$ and $q_\ell$ be as in Proposition \ref{forbac} and let $\alpha^{(n_\ell)}_{1,b}$ and $\alpha^{(n_\ell)}_{d,t}$ be as in 
 and \eqref{firstlastdef}. 
Then
\begin{align}\label{lemmaHMU1}
& \min \left\{ \ |l_{\alpha,t}^{(n_\ell)} -l_{\beta,b}^{(n_\ell)}| , \qquad \alpha \in \mathcal{A},\quad  \beta \in \mathcal{A} \backslash \{ \alpha^{(n_\ell)}_{1,b} \} \,  \right\} \geq \frac{1}{\nu }\lambda^{(n_{\ell+\overline{\ell}})},
\\ & \label{lemmaHMU2}
\min \left\{ \ |r_{\alpha,t}^{(n_\ell)} -r_{\beta,b}^{(n_\ell)}| , \qquad \alpha \in \mathcal{A} \backslash \{ \alpha^{(n_\ell)}_{d,t} \} ,\quad  \beta \in \mathcal{A}  \,  \right\} \geq \frac{1}{\nu }\lambda^{(n_{\ell+\overline{\ell}})}.
\end{align}
\end{lemma}

\noindent Using  Lemma \ref{lemmaHMU}, we can then prove the following result.
\begin{lemma}\label{backforw}
Suppose that $\{n_\ell\}_{\ell\in\N}$ is a sequence of $\nu$-balanced acceleration times for $T$ such that $\{n_{\overline{\ell}k}\}_{k\in\N}$ is positive for some $\ov{\ell}\in\N$. Set $ L:\overline{\ell}(1+ \left[ \log_d (2 \nu^2) \right])$. Then, for each $\ell$ sufficiently large,
\begin{align}\label{backforw-a}
&\bigcup_{0\leq i< 2q_\ell}{T^i  \left( \left[ l_\alpha, l_\alpha + \frac{1}{3\nu q_{\ell+L}} \right] \right) } \cap 
\{ l_\beta, r_\beta : \beta\in\cA\} = \emptyset, \ \ \text{for  all} \ l_\alpha \ \text{s.t.} \ T(l_\alpha)\neq 0, \ {\text i.e.\ s.t.\ } \alpha \neq \alpha_b^1;\\
\label{backforw-b}
&\bigcup_{0\leq i<2 q_\ell}{T^i  \left( \left[ r_\alpha - \frac{1}{3\nu q_{\ell+L}}, r_\alpha \right) \right) } \cap 
\{ l_\beta, r_\beta : \beta\in\cA\} = \emptyset, \ \  \text{for all}\ r_\alpha \neq 1, \ {\text i.e.\ s.t. }\ \alpha \neq \alpha_t^d.
\end{align}
\end{lemma}
We recall that the characterization of $0,1$ follows from Remark~\ref{singularitiesT}.

\begin{proof}
Recall first some basic properties of the towers $Z_\alpha^{(n_\ell)}$, $\alpha\in\cA$, $\ell\geq 1$ in the Rauzy-Veech induction (defined by \eqref{def:towers} in Section~\ref{Se:RV}). Recall that the discontinuities of $\mathcal{R}^{n_\ell}(T)$ are determined by the first visit via $T^{-1}$ to $I^{(n_\ell)}$ of the discontinuities for $T$. Therefore (see Figure~\ref{fig1}):
\begin{enumerate}[(a)]
\item
for each $\alpha\in\cA$, in the tower $Z_\alpha^{(n_\ell)}$ there is exactly one interval $T^{i_\alpha}I_\alpha^{(n_\ell)}$, for $0\leq i_\alpha\leq h_\alpha^{(n_\ell)}-1$, whose left endpoint belongs to $\{l_\beta : \beta\in\cA\}$; more precisely,  this left endpoint is $l_\alpha$;
\item
for each $\alpha\in\cA$, with the exclusion of $\alpha^{(n_{\ell+\ell_0})}_{1,t}$ (see \eqref{firstlastdef}), there is  one interval in the tower $Z_\alpha^{(n_\ell)}$, that we will denote by $T^{j_\alpha}I_\alpha^{(n_\ell)}$, $0\leq j_\alpha\leq h_\alpha^{(n_\ell)}-1$, whose right endpoint belongs to $\{r_\beta : \beta\in\cA \}$; more precisely,  this right endpoint is $r_{\alpha'}$, where $\alpha'$ is such that 
$r_{\alpha'} = l_\alpha$ (see Remark \ref{singularitiesT}, which gives that  explicitly $\alpha'\doteqdot  {\pi_t^{-1}(\pi_t(\alpha)-1)}$); 

\item the endpoint $1$ of the original interval (which is also equal to 
$r_{\alpha_{d,t}}$, see just before the  Remark \ref{singularitiesT}), is the image of $r_{ \alpha_{d,b}}$ and hence,  by (b), is the right endpoint of a floor $T^{i}I^{(n_\ell)}_\beta$ of $Z^{(n_\ell)}_\beta$, for some $0\leq i \leq h_\beta^{(n_\ell)}-1$,  then $1=r_{\alpha_d^t} $ is the right endpoint of the floor $T^{i+1}I^{(n_\ell)}_\beta$. In particular,  $Z^{(n_\ell)}_\beta$ contains two elements of  $\{r_\alpha : \alpha\in\cA\}$ one above the other (see Figure~\ref{fig1});
\item the point $T^{-1}(0)$ is the left endpoint of the top floor of the tower $Z_{\alpha^{(n)}_{1,b}}^{(n_\ell)}$.
\end{enumerate}

\begin{figure}[h!]
  \subfigure[Towers at level $n_\ell$.  \label{fig1} Here, $d=4$ and $\pi_b(\alpha_3')=d$.]{
  \includegraphics[width=0.39\textwidth]{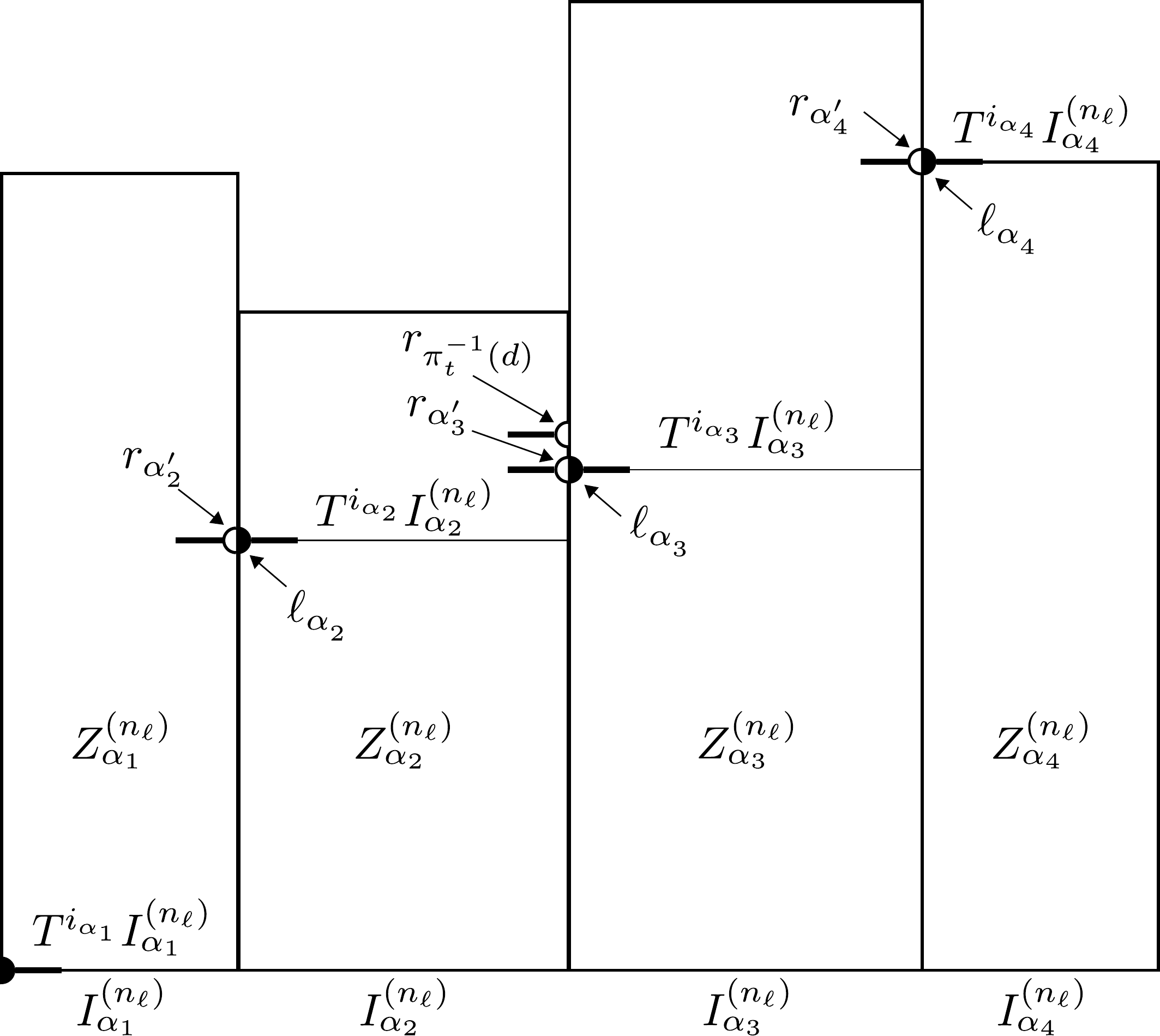}}
		\hspace{1.3cm}
\subfigure[Subtowers in \eqref{subtowers1a}, \eqref{subtowers1b} (lighter shade) and \eqref{subtowers2a}, \eqref{subtowers2b} (darker shade).\label{fig3}]{
    \includegraphics[width=0.39\textwidth]{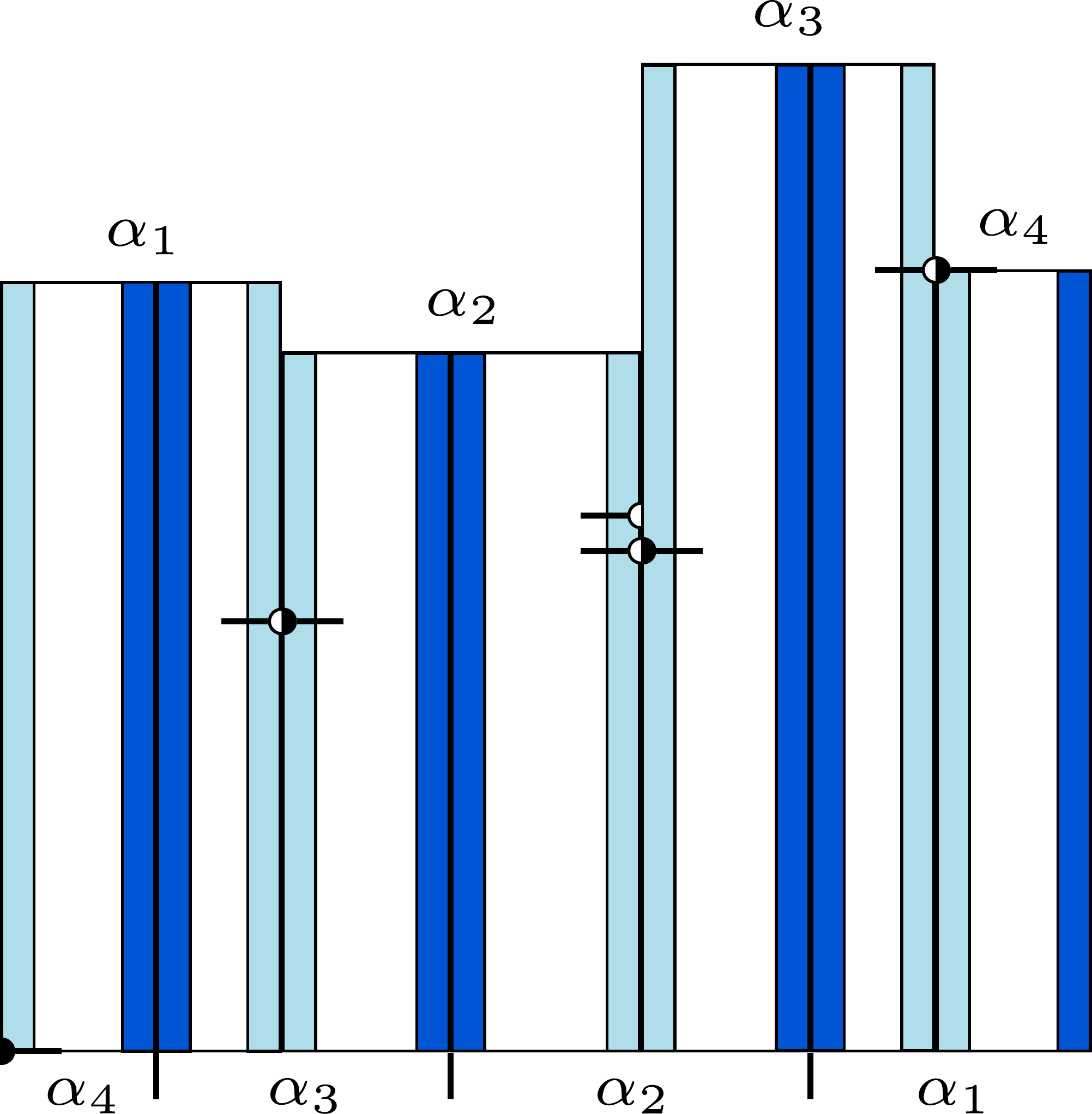}}
 \caption{The subtowers used in the proof of Lemma \ref{backforw}: the towers in \eqref{subtowers1a}, \eqref{subtowers1b} are in lighter shade, the ones given by \eqref{subtowers2a}, \eqref{subtowers2b} are in darker shade.}
\end{figure}


Let us choose a step $n_{\ell+\ell_0}$ whose towers are all taller than twice the shortest tower of step $n_{\ell}$. We claim that we have that
\begin{equation}\label{heightsrel}
 \min_{\alpha \in \mathcal{A}}  h^{(n_{\ell+\ell_0})}_\alpha \geq 2 q_\ell\doteqdot  2 \max_{\alpha \in \mathcal{A}}  h^{(n_{\ell+{\ell_0}})}_\alpha, \qquad \text{where} \ \ell_0\doteqdot  \left[ \log_d (2 \nu^2) \right] \overline{\ell} .
\end{equation}
Indeed, since $(n_\ell)_{\ell}$ is by assumption $\nu$-balanced (see Remark \ref{balancedcomparisons}), $(n_{k\overline{\ell}})_{k}$ is a positive sequence of induction times  and by positivity (see Remark \ref{positvegrowth}) it follows that for any $k\geq 1$, we have
\begin{equation*}
 \min_{\alpha \in \mathcal{A}}  h^{(n_{\ell+ k\overline{\ell}})}_\alpha \geq  \frac{1}{\nu\lambda^{(n_{\ell+k \overline{\ell}})}} \geq \frac{d^k}{\nu \lambda^{(n_{\ell})}} 
 \geq \frac{d^k}{\nu^2}  \max_{\alpha \in \mathcal{A}} h^{(n_\ell)}_\alpha, 
\end{equation*}
so for $k=\ell_0/\overline{\ell}$, we have that $d^k  >2  \nu^2 $ and hence  get \eqref{heightsrel} as desired.

Set $L\doteqdot  \ell_0 + \overline{\ell} = \overline{\ell} (1+ \left[ \log_d (2 \nu^2) \right])$,  where $\nu$ is the balance constant and $\overline{\ell}$ such that $(n_{k \overline{\ell}})_k$ is a positive sequence of induction times.    
Consider the following full height subtowers of the towers $Z_\alpha^{(n_{\ell+\ell_0})}$ (shown in Figure~\ref{fig3} in lighter shade), which have widths $1/(3\nu q_{\ell+L})$ and whose bases contain as endpoints the discontinuities of  $T^{(n_{\ell+\ell_0})}$:
\begin{align}\label{subtowers1a}
& \bigcup_{0\leq i< h^{(n_{\ell+\ell_0})}_\alpha}  T^i  \left( \left[ l_\alpha^{ n_{\ell+\ell_0},t} , l_\alpha^{\ n_{\ell+\ell_0},t} + \frac{1}{3\nu q_{\ell+L}} \right] \right), &   \alpha \in \mathcal{A}; \\
& \bigcup_{0\leq i< h^{(n_{\ell+\ell_0})}_\alpha}  T^i \left( \left[ r_\alpha^{\ n_{\ell+\ell_0},t} - \frac{1}{3\nu q_{\ell+L}}, r_\alpha^{\ n_{\ell+\ell_0},t}  \right)\right) , &     \alpha \in \mathcal{A}\backslash \{ \alpha^{(n_{\ell+\ell_0})}_{d,t}\}. 
\label{subtowers1b}
\end{align}


Consider also the following full height subtowers of the same width (shown in darker shade in Figure~\ref{fig3}), whose bases contain as endpoints the images of the discontinuities of $T^{(n_{\ell+\ell_0})}$:
\begin{align}\label{subtowers2a}
& \bigcup_{0\leq i< h^{(n_{\ell+\ell_0})}_{\beta(\alpha)}}   T^i \left(  \left[ l_\alpha^{\ n_{\ell+\ell_0},b} , l_\alpha^{\ n_{\ell+\ell_0},b} + \frac{1}{3\nu q_{\ell+L}} \right] \right)   , &  \alpha \in \mathcal{A}\backslash \{ \alpha^{(n_{\ell+\ell_0})}_{1,b}\},
  &\quad  \beta(\alpha) \ \text{s.t.}\ l_\alpha^{\ n_{\ell+\ell_0},t} \in I^{(n_{\ell+\ell_0})}_{\beta(\alpha)} ;\\
& \bigcup_{0\leq i< h^{(n_{\ell+\ell_0})}_{\beta(\alpha)}} T^i  \left( \left[ r_\alpha^{\ n_{\ell+\ell_0},b} - \frac{1}{3\nu q_{\ell+L}}, r_\alpha^{\ n_{\ell+\ell_0},b}  \right)  \right) , &  \alpha \in \mathcal{A}, & \quad \beta(\alpha) \ \text{s.t.}\ r_\alpha^{\ n_{\ell+\ell_0},t} \in I^{(n_{\ell+\ell_0})}_{\beta(\alpha)} . \label{subtowers2b}
\end{align}
The key remark that follows from Lemma \ref{lemmaHMU} is that the subtowers in \eqref{subtowers1a}, \eqref{subtowers1b}, \eqref{subtowers2a} and \eqref{subtowers2b} are all pairwise disjoint, since their width ${1}/3\nu{q_{\ell+L}}$ is less than half the distance  between the endpoints in their base floors (which is ${1}/(\nu{q_{\ell+L}})$ by Lemma \ref{lemmaHMU}). Thus, since by properties (a), (b) and (c) recalled at the beginning all elements of $\{l_\alpha,r_\alpha : \alpha\in\cA \}$ are endpoints of floors of the subtowers in \eqref{subtowers1a} and \eqref{subtowers1b},  the closure of the subtowers in \eqref{subtowers2a} and \eqref{subtowers2b} does not intersect $\{ l_\alpha, r_\alpha : \alpha\in\cA \}$.

Let us now  prove \eqref{backforw-a}. 
Take any $\alpha \neq \alpha_b^1$, so that $T(l_\alpha) \neq 0$. By Property (a) recalled at the beginning, $\left[ l_\alpha, l_\alpha + \frac{1}{(3\nu q_{\ell+L})} \right] $ is a floor of   one of the subtower in \eqref{subtowers1a}, the one indexed by the same $\alpha$. Moreover, from the assumption  $T(l_\alpha) \neq 0$ and property (d), it follows that $\pi_b^{(n_{\ell+\ell_0})}(\alpha)\neq 1$ and that the image of the last floor of the $\alpha$ subtower is the base of a subtower in \eqref{subtowers2a}. Thus,  since by \eqref{heightsrel} $2q_\ell \leq  \min_{\alpha \in \mathcal{A}}  h^{(n_{\ell+\ell_0})}_\alpha$, the images of the interval $\left[ l_\alpha, l_\alpha + \frac{1}{(3\nu q_{\ell+L})} \right] $ under $T^i$ for $0\leq i< 2q_\ell$ are contained in the union of the  $\alpha$ subtower in \eqref{subtowers1a} and of a subtower in \eqref{subtowers2a} (more precisely, the subtower which has as endpoint $l_{\alpha,n}^{(n_{\ell+\ell_0})}$, which is the image of $l_{\alpha,t}^{( n_{\ell+\ell_0})}$ under $T^{(n_{\ell+\ell_0})}$). Thus, to prove \eqref{backforw-a} holds it is enough to show that neither of these two subtowers contain other elements of $\{ l_\alpha, r_\alpha : \alpha\in\cA\}$. This is the case since by the properties (a), (b) and (c), the $\alpha$ subtower in \eqref{subtowers1a}  does not contain any other element of $\{ l_\alpha, r_\alpha : \alpha\in\cA\}$ apart $l_\alpha$ as left endpoint and, as remarked above, the closure of the subtowers in \eqref{subtowers2a}  does not intersect $\{ l_\alpha, r_\alpha : \alpha\in\cA\}$. Thus,  \eqref{backforw-a} holds.
 
Let us now  prove \eqref{backforw-b}. 
Take any $\alpha \neq \alpha_d^t$, so that $r_\alpha \neq 1$. By Property (b) recalled at the beginning, $\left[ r_\alpha - \frac{1}{(3\nu q_{\ell+L})} , r_\alpha \right) $ is a floor  of one of the subtowers \eqref{subtowers1a}. By all properties (a)-(d), since  $r_\alpha \neq 1$, the floors of this subtower \emph{above} the floor which contains $r_\alpha$ as right endpoint do not contain any other discontinuity in $\{ l_\alpha, r_\alpha : \alpha\in\cA\}$ in their closure. As before, by  \eqref{heightsrel}  the images of this interval under $T^i$ for $0\leq i< 2q_\ell$ are contained in the original subtower of \eqref{subtowers2a} and one of the subtower of \eqref{subtowers2b} (the one  which has as endpoint $r_{\alpha'}^{\ n_{\ell+\ell_0},b}$, which is the image of $r_{\alpha'}^{\ n_{\ell+\ell_0},t}$ under $T^{(n_{\ell+\ell_0})}$). Thus, since the closure of the subtowers in \eqref{subtowers2b}  also does not intersect $\{ l_\alpha, r_\alpha : \alpha\in\cA\}$,  \eqref{backforw-b} holds.
\end{proof}

\begin{proof}[Proof of \cref{forbac}] Let $L=\overline{\ell} (1+ \left[ \log_d (2 \nu^2) \right])$ and set $c\doteqdot  1/(6 \nu)$. Given $\vep>0$, let  $\ell'=\ell'(\vep)\geq 1$ such that $\lambda^{(n_{\ell'})}\leq \vep/8$. Fix $\ell\geq \ell'$. 
Assume that neither \eqref{j41} nor \eqref{j42} hold. Then there exists two continuity intervals endpoints $e_1, e_2 \in \{l_\alpha,r_\alpha : \alpha\in\cA \}$ and  $i_1, i_2 \in \mathbb{N}$ with 
\begin{equation}\label{times}
-q_\ell  \leq -i_1 < 0, \qquad 
0\leq i_2 < q_{\ell},  
\end{equation}
such that

\begin{equation}\label{distances}
|T^{-i_1} (x) - e_1| < \frac{c}{ q_{\ell + L}}, \qquad  |T^{i_2} (x) - e_2| < \frac{c}{ q_{\ell + L}}
\end{equation}
(see Figure \ref{configurations}  for a schematic picture).  
Without loss of generality, we can assume that $i_1, i_2$ are the smallest natural numbers which satisfy this property.  
Let us consider first the case in which $e_1 \leq  T^{-(i_1+i_2)} e_2 $ (two such configurations of points are shown in Figure \ref{a} and \ref{b}). By Remark \ref{singularitiesT}, we can assume that $e_1= l_\alpha$ for some $\alpha \in \mathcal{A}$. Let us show that the assumption that $x \notin [0, \varepsilon/8)$ guarantees that 
 $\alpha \neq \alpha_{1,b}$.  Indeed, if $\alpha = \alpha_{1,b}$, then $T(l_\alpha) =0$. In this case, since $0$ also belongs to $ \{l_\alpha,r_\alpha : \alpha\in\cA \}$, we must have $-i_1=-1, i_2=0$ and $e_2=0$. Thus, $|x|= |T^{0} x - 0| \leq c/({ q_{\ell + L}})$, which, since $c<1$ and $\ell+L \geq \ell'$, by Remark \ref{balancedcomparisons} gives $|x| \leq \lambda_{\ell' }$.  This,  by definition of $ \ell'$,  implies that $x \in [0, \varepsilon/8)$, which we are excluding by assumption.

 \begin{figure}[h!]
  \subfigure[ \label{a}]{
  \includegraphics[width=0.23\textwidth]{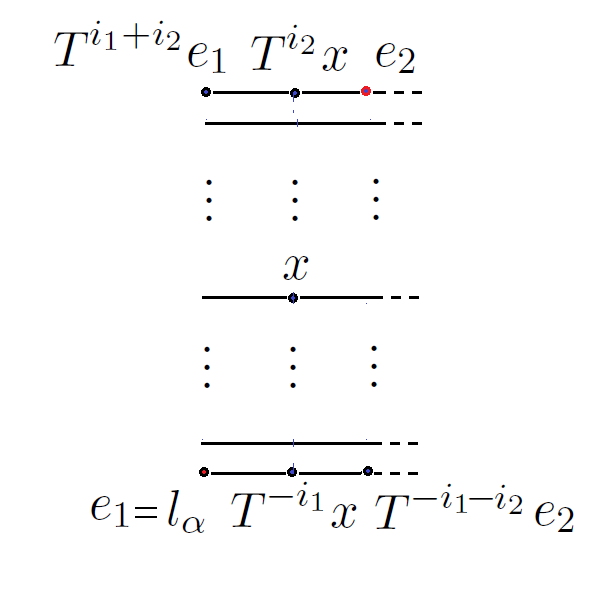} 	}  \subfigure[ \label{b}]{ \includegraphics[width=0.23\textwidth]{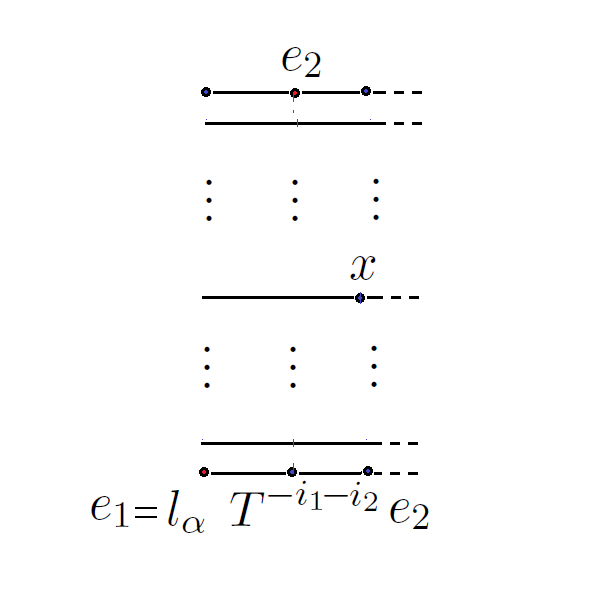}}
\subfigure[\label{c}]{
\includegraphics[width=0.23\textwidth]{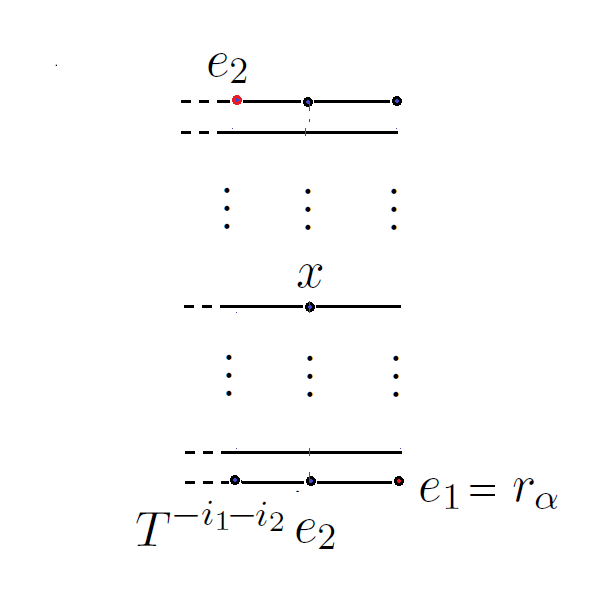} }	  \subfigure[ \label{d}]{	\includegraphics[width=0.23\textwidth]{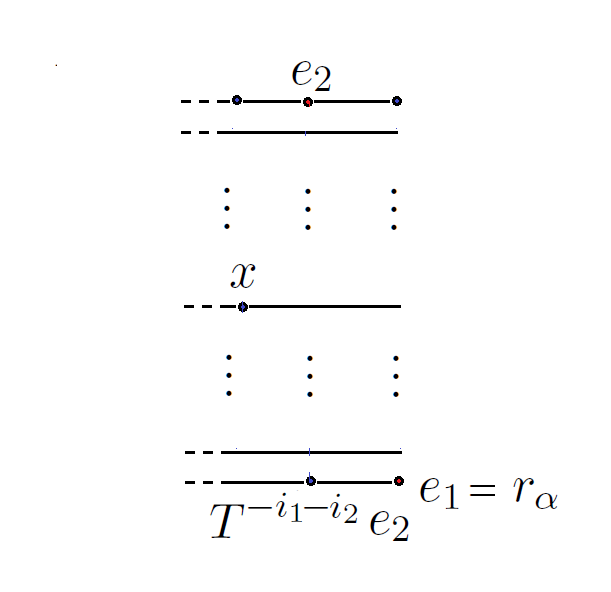}}
 \caption{
 in (a) and (b)  $e_1 \leq  T^{-(i_1+i_2)} e_2 $ and $e_1=l_\alpha$; in (c) and (d)  $e_1 \geq  T^{-(i_1+i_2)} e_2 $ and $e_1=r_\alpha$.\label{configurations}}
\end{figure}

We remark that by the choice of $i_1, i_2$, $T^j$ acts continuously on $[l_\alpha, l_\alpha + 1/(3\nu q_{\ell+L})]$ for any $0\leq j \leq i_1+i_2$ and hence it is an isometry. Thus,  using that $T^{i_1+i_2 }$ is an isometry and  \eqref{distances} twice and recalling the definition of $c$, we have that (see Figure \ref{a} for reference):
$$
| e_2 - T^{i_1+i_2 }( e_1)  | = | e_2 - T^{i_2 }( x)  | + | T^{i_2 }( x) - T^{i_1+i_2 }( e_1)  | \leq \frac{c}{ q_{\ell + L}} + | T^{-i_1 }( x) -  e_1  | < \frac{2c}{   q_{\ell + L}}= \frac{1}{  3\nu q_{\ell + L}}.
$$
Thus, by \eqref{times}, this shows that 
$$
e_2 \in T^{i_1+i_2}  \left( \left[ l_\alpha, l_\alpha + \frac{1}{3\nu q_{\ell+L}} \right] \right) ,\qquad \text{where} \ 0\leq i_1+i_2 < 2q_\ell,
$$
which contradicts  \eqref{backforw-a} in Lemma \ref{backforw}.

Similarly, if $e_1 \geq T^{-(i_1+i_2)} e_2 $ (two such configurations of points are shown in Figure \ref{c} and \ref{d}),  by Remark \ref{singularitiesT}, we can assume that $e_1= r_\alpha$ for some $\alpha \in \mathcal{A}$. Moreover, the assumption $x \notin (1- \varepsilon/8,1)$ guarantees that  $\alpha \neq \alpha_{t,d} $, since otherwise, since  $r_{\alpha_{t,d}}=1$ and $T(r_{\alpha_{t,d},t}) = r_{\alpha_{t,d}}$, we must have $e_2=r_{\alpha_{t,d}}$, $e_1=r_{\alpha_{t,d}}$ and hence $i_1=0, i_2=-1$. Thus, reasoning as before this yields that $|x-1|= |T^{0} x - r_{\alpha_{t,d}}| \leq 1/({3 \nu q_{\ell + L}})\leq \lambda_{\ell' } \leq \varepsilon/8$ and hence $x \in (1- \varepsilon/8,1)$, which we are excluding.  Reasoning in a similar way as in the previous case, we conclude that 
$$
e_2 \in T^{i_1+i_2}  \left( \left[ r_\alpha - \frac{1}{3\nu q_{\ell+L}}, r_\alpha \right) \right), \qquad \text{where} \ 0\leq i_1+i_2 < 2q_\ell,
$$
which this times  contradicts \eqref{backforw-b} in Lemma \ref{backforw}.

\end{proof}


\subsection{Proof that the Ratner DC implies the SR-property}\label{sec:deducingRatner}
In this section we will prove that if an IET is such that  the Ratner Summability Condition holds, one can prove the SR-property for special flows over $T$ with asymmetric logarithmic singularities. More precisely, we will prove the following. 


\begin{prop}\label{prop:Ratproof}
Let $T\colon I \to I$ be an IET and $f\colon I \to \mathbb{R}^+$  a roof function $f\in AsymLogSing(T)$ (see Definition \ref{def_LogAsym}). If $T$ satisfies the Ratner Summability Condition with exponents $(\tau, \eta, \xi)$ such that 
\begin{equation}\label{para}
 \tau\in (1,\frac{16}{15}),\;\;\;    \tau'\in (\frac{15}{16},1),\;\;\; 
	 \eta\in (3/4,2\tau '-\tau),\;\;\;
	 \xi\in(\max(99/100,\tau ' \eta),\tau ').
\end{equation}
  the special flow $(\varphi_t)_{t\in\mathbb{R}}$ over $T$ and under $f$ has the SR-property.
\end{prop}
We will assume for the rest of the section that the parameters $\tau, \tau', \xi , \eta$ are chosen as in \eqref{para}, that $T$ satisfies the assumptions of Proposition \ref{prop:Ratproof} and that $f\in AsymLogSing(T)$. Let us denote by $C^-$, $C^+$ the constants in the Definition \ref{def_LogAsym} of asymmetric logarithmic singularities. We will also  assume without loss of generality that $C^--C^+>0$. 
Let us first give an outline of the structure of the proof.

\smallskip
{\bf Outline of the proof.} We will prove the SR-property using Lemma \ref{njz} on Birkhoff sums to verifying  the assumptions of Proposition \ref{cocy}. 
Consider $x<y$ close. To use Lemma \ref{njz} we need two verify that the following two properties hold:
\begin{enumerate}
   \item[(i)] there is no discontinuity in $[T^ix,T^iy]$ for $i\in [M,M+L]$;
   \item[(ii)] we have good control of $S_n(f')(\theta)$ for $n\in[M,M+L]$ and $\theta\in[x,y]$.
\end{enumerate}  
In order to verify (ii), we use Proposition \ref{mpd} which guarantees that the term $n\log n$ in $S_n(f')(\theta)$ is well controlled. The problematic terms are $U(n,\theta)$ and $V(n,\theta)$, which depend on the distance of $T^i\theta$ from the singularities.  We define $\ell$ to be such that $\frac{1}{q_{\ell+1}\log q_{\ell+1}}\leq |x-y|<\frac{1}{q_\ell\log q_\ell}$ and consider two cases: $\ell\in K_T$ or $\ell\notin K_T$ (see Definition \ref{RatDC2}). When $\ell\in K_T$, we  use Lemma \ref{prty}, but to do so, we need good estimates on $U(q_{\ell+1},x),V(q_{\ell+1},x)$. This control is given by Proposition \ref{forbac}, which tells us that either going forward or backward in time, one can control the distance from singularities and hence that (i) will be satisfied either for positive or negative iterates. Moreover, by the same reasons, \eqref{fra} or \eqref{fra1} will be satisfied and by Lemma \ref{prty} and hence we will have good control of Birkhoff sums of the derivative, thus showing (ii). In the second case, namely when $\ell\notin K_T$,  both (i) and (ii) in Proposition \ref{cocy} will hold for most of the points, i.e. outside the set of points which go too close to some  of the singularities (which is defined in (\eqref{def.jlk} and \eqref{jlk} below). For $r\in [q_\ell,q_{\ell+1}]$ we want the main term in $S_r(f')$ (which is $r\log r$) to dominate the terms $U(r,x), V(r,x)$. Notice that if $r$ gets larger the main term is also larger, so the danger zones (in which $U(r,x), V(r,x)$ are too large) are getting smaller, so that one can control the measure of the set of points which are removed (see  \eqref{estimatejlk}). The assumption  that the Summability Condition is satisfied for $T$ implies that the set of $\ell\notin K_T$ is small and hence that we can throw away the union over $\ell$ of the sets of bad points (see \eqref{z2}) and still end up with a set of arbitrarly large measure, whose points stay sufficiently far from all the singularities and hence satisfy (i) and (ii). Thus, we can apply Lemma \ref{njz} and conclude the proof. 

\begin{proof}[Proof  of Proposition \ref{prop:Ratproof}]
We remark that $T$ by definition of Ratner DC also satisfies the Mixing DC, which implies in particular that $T$ is ergodic. Thus, in order to prove that $(\varphi_t)_{t\in\mathbb{R}}$ has SR-property it is enough to verify the assumptions of  Proposition \ref{cocy}. We will do this by using Lemma  \ref{njz} and Proposition \ref{forbac} on  backward or forward control of distance of orbits from singularities.
 
 Fix $\vep>0$ (small) and $N\in\N$. To verify Proposition \ref{cocy}, we need to define a $\kappa=\kappa(\vep)$ that we set to be $\kappa\doteqdot \vep^5$, a  $\delta=\delta(\vep,N)$ and a set of ``good'' points $X'$ with  $\lambda(X')>1-\vep$. Since $T$ satisfies the Ratner Summability Condition  with $\tau,\xi,\eta$ satisfying \eqref{para}, by  the Definition \ref{RatDC2} of Summability Condition and by   Corollary \ref{summabilitySigmas},     we have that the following two series are summable:
\begin{equation}\label{sersum2}
\sum_{\ell\notin K_T}\sigma_\ell^\eta<\infty, \qquad \sum_{\ell \notin K_T}\lambda(\Sigma_\ell(T)) <\infty, 
\end{equation}
where $K_T=\{\ell\in\N :  q_{\ell+L}\leq \frac{q_\ell}{\sigma_\ell^{\xi'}}\}$ (where $\frac{\xi}{\tau'}+10^{-3}>\xi'>\frac{\xi}{\tau'}$ is a small). 
Hence,   there exists $l_0(\vep)$ such that if we set 
\begin{equation}\label{lees}
Z_1: =\bigcup_{l\notin K_T, \ell\geq \ell_0}2\Sigma_\ell(T), \text{ then } \qquad  \lambda(Z_1)<  \frac{\vep}{3}, \quad \textrm{and} \quad 
\sum_{l\notin K_T,\;\ell \geq \ell_0 }\sigma_\ell^\eta<\frac{\vep}{6\nu^2|\cA|}.
\end{equation}
 
Fix $l\notin K_T, l\geq l_0$ and $k\in\{0,\ldots, \left[\frac{q_{l+1}}{q_\ell}\right]\}$. Define the set 
\begin{equation}\label{def.jlk}
J_\ell^k\doteqdot \bigcup_{\alpha\in\cA}\bigcup_{i=kq_\ell}^{(k+1)q_\ell-1}
T^{-i}[-\frac{1}{(k+1)q_\ell(\log (k+1)q_\ell)^{\xi}}
+l_\alpha,l_\alpha+\frac{1}{(k+1)q_\ell (\log (k+1) q_\ell)^{\xi}}].
\end{equation}
Notice that 
\begin{equation}\label{pot}\lambda(J_\ell^k)\leq  \frac{2|\cA|}{(k+1) (\log q_\ell )^{\xi}}.
\end{equation}
Moreover by \eqref{nr3} for $\ell$ sufficiently large we have 
\begin{equation}\label{nvc}
\frac{6|\cA|}{(\log q_\ell)^{\xi}}\log\|A_\ell \|\leq 
\sigma_l^{\eta}.
\end{equation}
Now define
\begin{equation}\label{jlk}
J_\ell\doteqdot \bigcup_{k=0}^{\left[\frac{q_{\ell+1}}{q_\ell }\right]+1}
J_\ell^k.
\end{equation}
We have by \eqref{pot}
\begin{equation}\label{estimatejlk}
\lambda(J_\ell)\leq \frac{2|\cA|}{(\log q_\ell)^{\xi}}\sum_{k=0}^
{\left[\frac{q_{\ell+1}}{q_\ell}\right]+1}\frac{1}{k}
\leq \frac{2|\cA|}{(\log q_\ell)^{\xi}}2\log(\left[\frac{q_{\ell+1}}{q_\ell}\right]+1)\leq
\frac{6|\cA|}{(\log q_\ell)^{\xi}}\log\|A_\ell\|\stackrel{\eqref{nvc}}{\leq} 
\sigma_\ell^{\eta}
\end{equation}
We define
\begin{equation}\label{z2}
Z_2\doteqdot \bigcup_{l\notin K_T,l\geq l_0}J_\ell. 
\end{equation}
Notice that by the above computations and \eqref{lees} we have $\lambda(Z_2)\leq \frac{\vep}{3}$. Finally, we define

$$X'\doteqdot Z_1^c\cap Z_2^c\cap (\frac{\vep}{8},1-\frac{\vep}{8}).
$$
Notice that since $\lambda(Z_1),\lambda(Z_2)<\vep/3$, we have $\lambda(X')>1-\vep$. We will prove Ratner (see Proposition \ref{cocy}) for pairs of points from $X'$.
Let now 
\begin{equation}\label{ela}\ell_a=\max(\frac{N^2+1}{\vep^4},1/\vep, l_0,\ell_1+1,\ell') ,
\end{equation}
where $\ell_1(\vep^2)$ is such that the estimates in Proposition \ref{mpd} and Lemma \ref{prty} hold for $\ell\geq \ell_1$ and $\ell'$ comes from Proposition \ref{forbac}. 
Define 
\begin{equation}\label{sdel}\delta\doteqdot \min(\frac{1}{\ell_a^2},\vep^2).
\end{equation}
 We will show that any $x,y\in X'$ with $|x-y|<\delta$ satisfy \eqref{posi} or \eqref{nega} of Proposition \ref{cocy}.\\
 \indent Let $r\in \N$ be the unique number such that 
\begin{equation}\label{dist}
\frac{1}{(C^--C^+)(r+1)\log (r+1)}<y-x\leq \frac{1}{(C^--C^+)r \log r}.
\end{equation}
Let now $\ell\in \N$ be the unique number such that $q_\ell\leq 
r<q_{\ell+1}$ (note that by \eqref{sdel} $\ell_a<\ell$).

 We will consider the following two cases:\\

\textbf{Case 1.} $\ell\in K_T$ (in particular $q_{\ell+L}\leq \frac{c}{10}q_\ell (\log q_\ell)^\xi$ where $c$ is comming from Proposition \ref{forbac}). \\
In this case, since $x\in X'$, we can use Proposition \ref{forbac} (with $\ell+1$) to get (WLOG we can assume that $L\geq 1$) 
\begin{equation}\label{op1}
d(\{l_\alpha,r_\alpha : \alpha\in\cA\}, \{T^i x : 0\leq i<q_{\ell+1}\})>\frac{c}{q_{\ell+L}}\geq
\frac{2}{q_\ell(\log q_\ell)^\xi},
\end{equation}
or
\begin{equation}\label{op2}
d(\{l_\alpha,r_\alpha : \alpha\in\cA\}, \{T^i x : -q_{\ell+1}\leq i< 0 \})>\frac{c}{q_{\ell+L}}\geq
\frac{2}{q_\ell(\log q_\ell)^\xi}.
\end{equation}
If \eqref{op1} holds, we show \eqref{posi}, if \eqref{op2} holds we show \eqref{nega}.  Since the proofs in both cases are analogous, we will conduct the proof assuming \eqref{op1} holds.\\

Let 
\begin{equation}\label{soundof}M\doteqdot \min (r,(1-\vep^4)q_{\ell+1})\text{ and }L=[\vep^5 M]+1,
\end{equation}
 (so that $L/M\geq \kappa$ and $M+L<q_{\ell+1}$). Notice that $\|x-y\|<\delta\stackrel{\eqref{sdel}}{<}\vep$. Moreover $$M\geq L>\vep^4 M\geq \vep^4 q_\ell>\vep^4\ell>\vep^4\ell_a>N,$$ (the last inequality by \eqref{ela}). Therefore, the assumptions of Lemma \ref{njz} are satisfied for $x,y,M,L$. Hence, to show \eqref{posi} in Proposition \ref{cocy}, it is enough to verify that \eqref{betai},\eqref{cont},\eqref{der} in Lemma \ref{njz} are satisfied.  

To show \eqref{betai} notice first that by \eqref{dist} and $r\geq q_\ell$,  $|y-x|<\frac{1}{(C^--C^+)q_\ell\log q_\ell}$. We have for every $\theta\in[x,y]$
\begin{multline}\label{betd} d\left(\{l_\alpha : \alpha\in\cA\},\{T^i \theta : 0\leq i<q_{\ell+1}\}\right)\geq\\
d\left(\{l_\alpha : \alpha\in\cA\},\{T^i x : 0\leq i<q_{\ell+1}\}\right)- |\theta-x|\stackrel{\eqref{op1}}{>}
\frac{2}{q_\ell(\log q_\ell)^\xi}-\frac{1}{(C^--C^+)q_\ell\log q_\ell}>\frac{1}{q_\ell(\log q_\ell)^\xi}
\end{multline}
the last inequality follows since $\xi<1$ and $\ell$ is large. This and $M+L<q_{\ell+1}$ gives \eqref{betai}.


Notice that by \eqref{betd}, for every $\theta\in[x,y]$ \eqref{fra} holds. So using \eqref{est.ex} ($[M,M+L]\subset [q_\ell,q_{\ell+1})$), we get that for every $(r,\theta)\in[M,M+L]\times[x,y]$
\begin{equation}\label{czo}
 0<(C^--C^+-\vep^2)r\log r\leq S_r(f')(\theta)\leq (C^--C^++\vep^2)r\log r,
\end{equation}
so the left  hand side automatically gives \eqref{cont}.

Now we show \eqref{der}. By \eqref{czo} we get for every $(r,\theta)\in[M,M+L]\times [x,y]$
\begin{equation}\label{sodo}
(C^--C^+-\vep^2)M\log M\leq S_r(f')(\theta)\leq (C^--C^++\vep^2)(M+L)\log (M+L).
\end{equation}
But by \eqref{soundof} and \eqref{dist} we get
$$
 M\geq (1-\vep^3)(r+1)\text{ and } M+L\leq (1+\vep^4)r.
$$
Pluging this into \eqref{sodo} and using \eqref{dist} gives
\eqref{der}. So by Lemma \ref{njz} (i) is satisfied. This finishes the proof in \textbf{Case 1.}\\

\textbf{Case 2.} $\ell\notin K_T$.\\
Notice first that for every $j\in\{0,...,[\sigma_\ell q_\ell+1]\}$, we have
\begin{equation}\label{sfr} 
l_\alpha\notin [T^jx,T^jy] \text{ for every } \alpha\in\cA. 
\end{equation}
 Indeed, notice that $\sigma_\ell\log q_\ell\to 0$ as $\ell\to+\infty$. This, $n_\ell$ being a balanced time and \eqref{dist} give
\begin{equation}\label{cra}
\sigma_\ell I^{(n_\ell)}\geq \frac{\nu\sigma_\ell}{q_\ell}\geq \frac{100}{q_\ell\log q_\ell}>2|y-x|.
\end{equation}
But $x\in X'\subset Z_1^c\subset (2\Sigma_\ell^+(T))^c$, so 
 for every $j\in\{0,...,[\sigma_\ell q_\ell+1]\}$ and every $\alpha\in\cA$, 
\begin{equation}\label{hgf}
d(T^jx,l_\alpha)\geq 2\sigma_\ell I^{(n_\ell)}>|y-x|,
\end{equation}
which gives \eqref{sfr}. We claim that for every $\theta\in [x,y]$
\begin{equation}\label{kt1}
\theta\notin \Sigma_\ell^+(T).
\end{equation}
Let us prove \eqref{kt1} by contradiction: by the definition of $\Sigma_\ell^+(T)$, if  \eqref{kt1} fails, it would mean that there exist $i\in \{0,...,[\sigma_\ell q_\ell+1]\}$ and $\alpha\in \cA$, such that $d(T^i\theta,l_\alpha)\leq \sigma_\ell I^{(n_\ell)}$ (choose $i$ to be the smallest one with this property).  So
$$
d(T^ix,l_\alpha)\leq d(T^i\theta,l_\alpha)+d(T^ix,T^i\theta)\leq \sigma_\ell I^{(n_\ell)}+ |y-x|\stackrel{\eqref{cra}}{<}2\sigma_\ell I^{(n_\ell)},
$$
a contradiction with \eqref{hgf}. So \eqref{kt1} holds.
Moreover, the following holds, for every $\theta\in[x,y]$
\begin{equation}\label{uvm}
U(m,\theta),V(m,\theta)\leq 2m(\log m)^{\xi},
\end{equation}
for every $m\in[q_\ell, (1+\vep)r]$. Let us show that for every $m\in[q_\ell,(1+\vep)r]$, 
$U(m,\theta)> 2 m(\log m)^{\xi}$ (the proof of $V(m,\theta)>2 m(\log m)^{\xi}$ is analogous). 

This follows by the fact that $x\in X'\subset Z_2^c\stackrel{\eqref{z2}}{\subset} J_\ell$. Indeed, 
let $1 \leq k\in \leq \left[\frac{q_{\ell+1}}{q_\ell}\right] +1$ be such that $kq_\ell\leq m <(k+1)q_\ell$. Then, by \eqref{jlk}, \eqref{def.jlk}, the fact that $x\in J_\ell^k$, 
we get
$$
\min_{\alpha\in\cA} d(\{x,...,T^mx\},l_\alpha)\geq \frac{1}{(k+1)q_\ell(\log(k+1)q_\ell)^\xi}\geq \frac{1}{2m(\log 2m)^\xi}.
$$
Therefore, and since $m\leq (1+\vep)r$ and $\xi<1$, we have 

$$
\min_{\alpha\in\cA}d(\{\theta,...,T^m\theta\},l_\alpha)
\stackrel{\eqref{dist}}{\geq} 
\min_{\alpha\in\cA}d(\{x,...,T^mx\},l_\alpha)- \frac{1}{(C^--C^+)r \log r}\geq
\frac{1}{2m (\log m)^\xi},
$$
so \eqref{uvm} holds.
 
Now define $M\doteqdot \max(r,(1-\vep^4)q_{\ell+1})$, $L\doteqdot [\vep^5 M]+1$.
From this point the proof is analogous to the proof of \textbf{Case 1.} : We verify assumptions \eqref{betai}, \eqref{cont} and \eqref{der} in Lemma \ref{njz}. 
By \eqref{uvm} we get that \eqref{betai} holds.

 Moreover, by \eqref{kt1} and \eqref{uvm} and Proposition \ref{mpd} it follows that for every $r,\theta\in [M,M+L]\times[x,y]$
 $$
(C^--C^+-\vep^2)r\log r\leq S_r(f')(\theta)\leq (C^--C^++\vep^2)r\log r.
 $$
Now since $M, L$ are the same as in \textbf{Case 1.} and the above estimate is the same as \eqref{czo} we verify \eqref{cont} and \eqref{der} repeating the rest of the proof of \textbf{Case 1.} This finishes the proof in \textbf{Case 2.}

\end{proof}

\subsection{Conclusions}\label{sec:conclusions} 
In this final section we conclude by giving the proofs of all the results stated in the introduction. 
The results proved so far immediately give  the proof of Theorem \ref{thm:Ratner_special_flows}, namely show that special flows under functions with logarithmic asymmetric singularities have the SR-property for a.e. IET:
\begin{proof}[Proof of Theorem \ref{thm:Ratner_special_flows}]
 Let $\tau\in (1,16/15)$, $\tau'\in(15/16,1)$,  $\xi'>99/100$, $\eta'>3/4$.  For each irreducible combinatorial datum $\pi$ and for Lebesgue a.e.\  $\underline{\lambda} \in  \Delta_{d}$, the corresponding IET $T= (\underline{\lambda}, \pi) $ satisfies the Summability Condition with exponents $(\tau',\xi',\eta')$ by Corollary \ref{RatvsPar}.  
Hence, by Proposition \ref{prop:Ratproof},  the  special flow $(\varphi_t)_{t\in\mathbb{R}}$ over $T$ and under $f\in AsymLogSing(T)$  has the SR-property. 
\end{proof} 

Let us now prove the corresponding result on the switchable Ratner property in the context of locally Hamiltonian flows, namely Corollary \ref{cor:genFK}.
\begin{proof}[Proof of Corollary \ref{cor:genFK}]
For any fixed genus $g\geq 1$, consider the open set $\mathcal{U}_{\neg min}$ of locally Hamiltonian flows on a surface $S$ of genus $g$ with  non-degenerate fixed points which have a saddle loop homologous to zero. Equivalently, these are locally Hamiltonian flows with  non-degenerate fixed points which have at least one periodic component. As explained by Ravotti in \cite{Ra:mix} (see Section~2 and Section~3), there exists an open and dense set 
$\mathcal{U}_{\neg min}' \subset \mathcal{U}_{\neg min}$ (this is denoted by $\mathcal{A}_{s,l}'$ in \cite{Ra:mix}, see Notation 3.3) such that any minimal component of a locally Hamiltonian flow $(\varphi_t)_{t \in \mathbb{R}}$ in $\mathcal{U}_{\neg min}'$ can be represented as a special flow over an IET $T=(\lambda, \pi)$ under a roof $f\in AsymLogSing(T)$ where  $\pi$ is irreducible. Furthermore, by Remark~\ref{rk:coordinates},  a property which holds for a full measure set of IETs on any number of intervals also holds for the special flow representation of \emph{each} minimal component of a full measure of flows in $\mathcal{U}_{\neg min}'$. Thus,  by  Theorem \ref{thm:Ratner_special_flows},  each minimal component of a typical flow in $\mathcal{U}_{\neg min}'$ admits a representation as a special flow which has the SR-property. Since the special flow representations are metrically isomorphic to the restrictions of $(\varphi_t)_{t \in \mathbb{R}}$ to the corresponding minimal component and the SR-property is an isomorphism invariant (see Lemma \ref{lemma:iso_inv} in Appendix \ref{appendix:Ratner}), it follows for a full measure set of flows in $\mathcal{U}_{\neg min}'$, each restriction of the flow to a minimal component has the SR-Ratner property.
\end{proof}
%

From the SR-property, we an now deduce the results on mixing of all orders in the set up of special flows  (Theorem \ref{prop:main_special_flows}) and finally locally Hamiltonian flows (Theorem \ref{thm:main}).

\begin{proof}[Proof of Theorem~\ref{prop:main_special_flows}]
Fix any irreducible permutation $\pi$. Consider the special flow $(\varphi_t)_{t\in\mathbb{R}}$ over $T= (\underline{\lambda}, \pi) $ and under a roof $f \in AsymLogSing(T)$. 
By Theorem \ref{thm:Ratner_special_flows}, for Lebesgue a.e.\  $\underline{\lambda} \in  \Delta_{d}$,  $(\varphi_t)_{t\in\mathbb{R}}$  has the SR-property and hence in particular also the SWR-property (which is weaker, recall the Definitions \ref{def:SWR} and \ref{def:SR}).  
 On the other hand, by Proposition \ref{existencebalancedtimes} proved in \cite{Ul:mix} and Theorem \ref{DCmixing} (see \cite{Ul:mix, Ra:mix}),  for Lebesgue a.e.\  $\underline{\lambda} \in  \Delta_{d}$ we also have that $(\varphi_t)_{t\in\mathbb{R}}$ is mixing. Thus, for a full measure set of $\underline{\lambda} \in  \Delta_{d}$, $(\varphi_t)_{t\in\mathbb{R}}$ is mixing and has the SWR-property, which, by Theorem \ref{thm:Ratmix}, implies that $(\varphi_t)_{t\in\mathbb{R}}$ is also mixing of all orders.
\end{proof} 



\begin{proof}[Proof of Theorem~\ref{thm:main}]
By Corollary  \ref{cor:genFK}, for any $g\geq 1$ there exists an open and dense set $\mathcal{U}_{\neg min}'$  in the open set  $\mathcal{U}_{\neg min}$ of locally Hamiltonian flows on a surface $S$ of genus $g\geq 1$ (which we recall consists of locally Hamiltonian with  non-degenerate fixed points which have a saddle loop homologous to zero) such that  any minimal component of a  typical locally Hamiltonian flow $(\varphi_t)_{t \in \mathbb{R}}$ in $\mathcal{U}_{\neg min}'$ has the SR-property. Furthermore, by Proposition \ref{existencebalancedtimes}, Theorem \ref{DCmixing}  and Remark \ref{rk:coordinates}, one can also assume by the same arguments in the proof of Corollary  \ref{cor:genFK} that for typical $(\varphi_t)_{t \in \mathbb{R}}$ in $\mathcal{U}_{\neg min}'$ the restriction to each minimal component is also mixing. Thus, by Theorem \ref{thm:Ratmix}, $(\varphi_t)_{t\in\mathbb{R}}$ is  mixing of all orders on each minimal component.
\end{proof}

We conclude by proving Corollary \ref{cor:genFK} which is a streghthening of the main result by Fayad and the first author \cite{FK:mul}. The proof  is based on the  observation that we can add the  singularities of the function as marked points, so that we get a function with asymmetric logarithmic singularities at the discontinuities of a IET  with fake singularities and on a Fubini argument.  
\begin{proof}[Proof of Corollary \ref{cor:genFK}]
Assume by contradiction that the conclusion is false; then there exists a set of positive measure $A \subset [0,1]$ and  a set of positive measure $X \subset [0,1]^d$ such that the special flow over $R_\alpha$ with $\alpha \in A$ under a roof $f$ with singularities (in the sense of Definition \ref{def_LogAsym}) at $x_0\doteqdot 0< x_1< \dots < x_d<x_{d+1}\doteqdot 1$ given by $\underline{x}\doteqdot (x_1, \dots, x_d) \in X$   does not satisfy the SR-property. We can choose  $(\alpha^0$, $\underline{x}^0) \in A \times X$ to be a Lebesgue density point such that $\alpha \neq x_i$ for any $0\leq i \leq d+1$, since both are full measure conditions.  Say that $x^0_{i-1} < 1-\alpha^0 < x_{i}^0$ for  $1\leq i \leq d+1$. This relation will also be true for $(\alpha, \underline{x})$ sufficiently close to $(\alpha^0, \underline{x}^0)$.  
For all these parameters $(\alpha, \underline{x})$, 
 by thinking of the singularities $x_i$ as as maked points, we can think of the rotations  $R_\alpha$ as  IETs on $d+1$ intervals, whose lengths and combinatorial data are explicitly
 given by 
 \begin{equation}\label{IETlenghts}
\underline{\lambda}\doteqdot  \left( x_1, x_2-x_1, \dots, x_{i-1}-x_{i-2}, 1-\alpha - x_{i-1}, x_i-(1-\alpha), x_{i+1}-x_{i}, \dots, 1-x_d\right),
\end{equation}
\begin{align*}  \pi^{rot,i}_t & = (1,2, \dots,i, i+1, \dots, d+1), \\ \pi^{rot,i}_b &= (i, \dots, d+1, 1, 2, \dots, i -1).
\end{align*}
We remark that $\pi^{rot,i} = (\pi^{rot,i}_t, \pi^{rot,i}_b) $ is irreducible.
 By Lebesgue density Theorem, one can find an set $E \subset A \times X \subset [0,1]^{d+1}$ of positive measure such that \eqref{IETlenghts} for all $(\alpha, \underline{x}) $ in $E$. 
Since the map $(\alpha, \underline{x}) \to \underline{\lambda}$ given by \eqref{IETlenghts} is linear, this gives a positive measure set of $\lambda \in [0,1]^{d+1}$ such that 
 the special flow over $T = (\lambda, \pi_{rot,i})$ with $f \in AsymLog (T)$   does not satisfy the SR-property, hence contradicting  Theorem \ref{thm:Ratner_special_flows}. 
This concludes the proof. \end{proof}

\appendix

\section{}
In this Appendix we include, for convenience of the reader, the proofs of two results used in the previous sections, namely the proof that the Switchable Ratner property is an isomorphism invariant (in Section~\ref{appendix:Ratner}) and the proof of the
Lemma from \cite{HMU:lag} which allows to control distances among discontinuities of an IET in terms of the lenght of the inducing interval of the next balanced Rauzy-Veech time (in Section~\ref{appendix:IETs}).

\subsection{Ratner properties are an isomorphism invariant}\label{appendix:Ratner}
In this Appendix we include for completeness the proof that the SR-property is an isomorphism invariant (the same holds for other Ratner properties with the set $P$ being finite).


\begin{lemma}\label{lemma:iso_inv}
Let $(X, (T_t), \mathscr{A}, \mu, d_T)$ and  $(Y, (S_t), \mathscr{B}, \nu,d_S)$ two  {measurably isomorphic}  measure preserving flows. Then, if  $(T_t)$  has the SR-property, also $(S_t)$ does.
\end{lemma}
\begin{proof}
Let us denote by $\psi: X \to Y$ the measurable isomorphism.  Since $(T_t)$ and $(S_t)$ are isomorphic we have $\psi T_t=S_t \psi$ for $t\in\mathbb{R}$. Let $t_0\in \mathbb{R}$ be such that $(T_t)$ has the $sR(t_0,\{-1,1\})$ property and such that $T_{t_0}$ and $S_{t_0}$ are ergodic. We will show that $(S_t)$ also has the $sR(t_0,\{-1,1\})$ property. For simplicity of notation assume that $t_0=1$ (we have $\psi T=S \psi$). Fix $\epsilon>0$ and $N\in \N$. By Egorov's theorem there exists a set $B_\psi\subset X$  $\mu(B_\psi)>1-\epsilon^3$ and $\epsilon'=\epsilon'(\epsilon)>0$ such that 
\begin{equation}\label{yeg1}
\text{ for every } x,y\in B_\psi, d_T(x,y)<\epsilon' \text{ we have } d_S(\psi x,\psi y)<\epsilon^2.
\end{equation}

  Let $\kappa=\kappa(S)=\kappa(T)(\epsilon')$ ($\kappa(T)$  coming from
SR-property for $T$ with  $\epsilon'$). 
By Luzin's lemma, there exists $N_0\in \mathbb{N}$ and a set $C_T\in X$, $\mu(C_T)\geq 1-\epsilon^2$ such that for every $x\in C_T$ and $M,L\geq N_0$, $\frac{L}{M}\geq \kappa$
\begin{equation}\label{luz}
\frac{1}{L}\sum_{i=M}^{M+L}\chi_{B_\psi}(T^ix)\geq 1-\epsilon^2.
\end{equation}
Denote $\tilde{N}=\max(N,N_0)$.
Let $Z_T=Z_T(\epsilon',\tilde{N})$, $\mu(Z_T)\geq 1-\epsilon'$ and $\delta_T=\delta_T(\epsilon',\tilde{N})$ be the SR-parameters for $\epsilon'$ and $\tilde{N}$. 
Using Egorov's theorem, there exists a set $V_\psi$, $\nu(V_\psi)\geq 1-\epsilon^3$ and $\delta'=\delta'(\delta_T)$ such that  
\begin{equation}\label{yeg2}
\text{ for every } x,y\in V_\psi, d_S(x, y)<\delta' \text{ we have } d_S(\psi^{-1}x,\psi^{-1}y)<\delta_T.
\end{equation}
Define $\delta_S=\delta_S(\epsilon,N)=\delta'$ and 
$$
Z_S=Z_S(\epsilon,N)\doteqdot \psi(Z_T\cap C_T)\cap V_\psi.
$$
Notice that by the definition of $Z_T,C_T$ and $V_\psi$ we get that $\nu(Z_S)\geq 1-\epsilon$. We will show that $Z_S$ and $\delta_S$ satisfy the assumptions of SR-property. For this aim let's take $x,y\in Z_S$ such that $d_S(x,y)<\delta_S$. Then by the definition of $Z_S$ and \eqref{yeg2} it follows that $\psi^{-1}x,\psi^{-1}y \in Z_T\cap C_T$ and $d_T(\psi^{-1}x,\psi^{-1}y)<\delta_T$. By the SR-property for $T$ it follows that there exist $M,L\geq \tilde{N}$, $L/M\geq \kappa$ and $p\in\{-1,1\}$ 
such that
\begin{equation}\label{rat.large}
\frac{1}{L}\{i\in[M,M+L] :  d_T(T^i(\psi^{-1}x),T^{i+p}(\psi^{-1}y))<\epsilon'\}>1-\epsilon^2.
\end{equation}
But since $\psi^{-1}x,\psi^{-1}y\in C_T$ it follows that 
$$
\frac{1}{L}\{i\in[M,M+L] : T^i(\psi^{-1}x),T^{i+p}(\psi^{-1}y)\in B_\psi\}>1-\epsilon^2.
$$
Therefore, by \eqref{yeg1}
\begin{equation}\label{rear}
\frac{1}{L}\{i\in[M,M+L] :  d_S(\psi(T^i(\psi^{-1}x)),\psi(T^{i+p}(\psi^{-1}y)))<\epsilon\}>1-\epsilon^2.
\end{equation}
By \eqref{rat.large}, \eqref{rear} and the fact that $\psi T^i\psi^{-1}=S^i$ for $i\in \mathbb{Z}$ we get 
$$
\frac{1}{L}\{i\in[M,M+L] : d_S(S^ix,S^{i+p}y)<\epsilon\}>1-\epsilon.
$$
Therefore $(S_t)$ indeed has the $sR(1,\{1,-1\})$ property. This finishes the proof.


\end{proof}

\subsection{Singularities distances control by positive Rauzy-Veech times}\label{appendix:IETs}
In this section we include the proof of Lemma \ref{lemmaHMU} used in Section~\ref{sec:backward_forward}. The proof is a minor modification of the proof of Lemma C.1 and Corollary C.2 in the paper \cite{HMU:lag} by Hubert, Marchese and the third author, rewritten with the notation used in this paper for convenience of the reader. 

For the rest of the section, we will assume that $T$ is an IET which satisfies the Keane condition and  that  $\{n_\ell\}_{\ell\in\N}$ is a sequence of $\nu$-balanced  induction times for $T$ such that $\{n_{\overline{\ell}k}\}_{k\in\N}$ is a positive sequence of times for some $\ov{\ell}\in\N$ (as in the assumptions of Proposition \ref{forbac} and Lemma \ref{lemmaHMU}).  
Let us remark that, using the notation introduced in \eqref{disc_Tn} and \eqref{firstlastdef}, the sets  
\begin{align*}
D_\ell & \doteqdot  \left\{ l_{\alpha,t}^{( n_\ell)}, \quad \alpha \in \mathcal{A}\backslash \{\alpha^{(n_\ell)}_{1,t} \} \right\} =  \left\{ r_{\alpha,t}^{( n_\ell)}, \quad \alpha \in \mathcal{A}\backslash \{\alpha^{(n_\ell)}_{d,t} \} \right\}, \\  D_{\ell}^{-1} & \doteqdot   \left\{ l_{\alpha,b}^{( n_\ell)}, \quad \alpha \in \mathcal{A}\backslash \{\alpha^{(n_\ell)}_{1,b} \} \right\}=\left\{ r_{\alpha,b}^{( n_\ell)} , \quad \alpha \in \mathcal{A}\backslash \{\alpha^{(n_\ell)}_{d,b} \} \right\}
\end{align*}
consist  respectively of  the discontinuties of $T^{(n_\ell)}$ and  its inverse $(T^{(n_{\ell})})^{-1}$. Recall that we write $B>0$ if all the entries of the matrix $B$ are strictly positive.

To understand the details of the proof, it is useful to keep in mind the main idea behind it, which is based on the analysis of the effect of Rauzy-Veech induction:   the  $(n+1)^{th}$ steps of Rauzy-Veech induction $T^{(n+1)}$ is obtained  by inducing $T^{(n)}$ on an interval $I^{(n+1)}$ whose right endpoint is the discontinuity of either  $T^{(n)}$ or its inverse which is closest to the right endpoint of $I^{(n)}$. In particular, this implies that the distance between two discontinuities (the  endpoint  of $I^{(n)}$, which is a discontinuity of $T^{(n-1)}$,  and the closest discontinuity of $T^{(n)}$) is bounded below by the lenght of an interval exchanged by $T^{(n+1)}$. 
 Thus, starting from $T^{(n_\ell)}$, since discontinuities of $T^{(n_\ell)}$ (and its inverse) appear as discontinuities of $T^{(n)}$ (and its inverse) for $n\geq n_\ell$,  considering the induction steps up to the next balanced step  $T^{(n_{\ell+1})}$ guarantees that the distance between all pairs of discontinuities in $D_\ell$ and  $D_{\ell}^{-1}$ can be controlled by lenghts of an interval of some $T^{(n)}$ with $n_\ell \leq n \leq n_{\ell+1}$, and hence (by monotonicity and balance) in terms of the lenght of $I^{(n_{\ell+1})}$.

\begin{proof}[Proof of Lemma \ref{lemmaHMU}]
Let us first show that, since $B^{(n_\ell, n_{\ell+1})}>0$,  
\begin{equation}\label{auxlemma} \left[0, \lambda^{(n_{\ell+1})}\right) \cap \left( D_{\ell} \cup D_{\ell}^{-1}\right) = \emptyset
\end{equation}
 (compare with Lemma C.1 in \cite{HMU:lag}). Recall (see \eqref{lengthsrelation} in Section~\ref{Se:RV} and the notation thereafter) that we have
$
\lambda^{(n_\ell)}_\alpha=
\sum_{\chi\in\cA}B^{(n_\ell, n_{\ell+1})}_{\alpha\chi}\lambda^{(n_{\ell+1})}_\chi
$
for any letter $\alpha\in\cA$. Therefore, since all the entries of the matrix $B^{(n_\ell, n_{\ell+1})}$ are positive and hence, being integers, are greater than $1$, 
 we have that
$
\min_{\alpha\in\cA}\lambda^{(n_{\ell})}_\alpha 
\geq  \sum_{\chi\in\cA}\lambda^{(n_{\ell+1})}_\chi = \lambda^{(n_{\ell+1})}  
$. 
Thus, \eqref{auxlemma} follows since the elements of $D_\ell$ or $D_{\ell}^{-1}$ are all right endpoints of union of intervals whose lengths all belong to the set $\{\lambda^{(n_{\ell})}_\chi;\chi\in\cA\}$ and hence each of them is greater than a non trivial sum of these lenghts.

We can now finish the proof of Lemma \ref{lemmaHMU}. Assume first that $\alpha = \alpha^{(n_{\ell})}_{t,1}$, so that $l_{\alpha,t}^{(n_\ell)}=0$. In this case, since we are assuming that $\beta \neq \alpha^{(n_\ell)}_{1,b}$ and hence $l_{\beta,b}^{(n_\ell)} \neq 0$, using that  $n_{\ell}$ is $\nu$-balanced by assumption (recall Remark~\ref{balancedcomparisons}), we have that  $|l_{\beta,b}^{(n_\ell)}| \geq \min_\chi \lambda^{(n_{\ell})}_\chi \geq \frac{1}{\nu} \lambda^{(n_{\ell})}  \geq\frac{1}{\nu}   \lambda^{(n_{\ell+\overline{\ell}})}$ and hence  \eqref{lemmaHMU1} holds trivially in this case. Assume next that  $\alpha \neq \alpha^{(n_{\ell})}_{1,t}$ and $\beta \neq \alpha^{(n_\ell)}_{1,b}$, so that $\alpha \in D_\ell$ and $\beta \in D_{\ell}^{-1}$.   Consider the minimum $n\geq n_\ell$ such that both $l_{\alpha,t}^{(n_\ell)}$ and  $l_{\beta,b}^{(n_\ell)}$ do not belong to the interior of $I^{(n)}$. By \eqref{auxlemma}, $ n \leq n_{\ell+1}$. By definition of Rauzy-Veech induction and $n$,  if  $l_{\alpha,t}^{(n_\ell)} > l_{\beta,b}^{(n_\ell)}$, $l_{\alpha,t}^{(n_\ell)}$ is the closest discontinuity of  $T^{(n-1)}$  to the right endpoint of $I^{(n-1)} $ and $I^{(n)} = [0, l_{\alpha,t}^{(n_\ell)})$, or,  if  $l_{\alpha,t}^{(n_\ell)} < l_{\beta,b}^{(n_\ell)}$, then $l_{\beta,b}^{(n_\ell)}$ is the closest discontinuity of the inverse of $T^{(n-1)}$  to the right endpoint of $I^{(n-1)} $ and $I^{(n)} = [0, l_{\beta,b}^{(n_\ell)})$. In the first case, $I^{(n)}$ is obtained by removing from $I^{(n-1)}$ an interval of lenght $\lambda^{(n-1)}_\alpha$ (since in this case $\alpha = \alpha^{(n-1)}_{t,d}$), while in the second of length $\lambda^{(n-1)}_\beta$ (since in that case  $\beta = \alpha^{(n-1)}_{b,d}$). 
In both cases, using that for any $\chi \in \mathcal{A}$ the sequence of lenghts $(\lambda^{(k)}_\alpha)_k$ is non-increasing in $k$ and recalling that  the step $n_{\ell+1}$ is $\nu$-balanced by assumption (recall Remark~\ref{balancedcomparisons}), we have that
$$
|l_{\alpha,t}^{(n_\ell)} - l_{\beta,b}^{(n_\ell)}|  \geq 
\min_{\chi \in \mathcal{A}}\lambda^{(n-1)}_\chi \geq \min_{\chi \in \mathcal{A}}\lambda^{(n_{\ell+1})}_\chi \geq \frac{1}{\nu }\lambda^{(n_{\ell+1})}.
$$
 This concludes the proof of \eqref{lemmaHMU1}. To prove \eqref{lemmaHMU2}, it is enough to remark that if $\beta \neq \alpha^{(n_\ell)}_{d,b}$, since by assumption $\alpha  \neq \alpha^{(n_{\ell+1})}_{d,t}$, \eqref{lemmaHMU2} reduces to \eqref{lemmaHMU1} by Remark \ref{singularitiesT}. On the other hand, if $\beta = \beta^{(n_\ell)}_{d,b}$ we have  that $r_{\beta,b}^{(n_\ell)}=1$, and,   by the  assumption that $\alpha  \neq \alpha^{(n_{\ell+1})}_{d,t}$, $r_{\alpha,t}^{(n_\ell)}$ is not the endpoint of the last interval exchanged by $T^{(n_\ell)}$. Thus, using  again $\nu$-balance of $n_\ell$, we have that
 $ |r_{\alpha,t}^{(n_\ell)} -1| \geq \min_\chi  \lambda^{(n_{\ell})}_\chi  \geq  \frac{1}{\nu }   \lambda^{(n_{\ell})} \geq \frac{1}{\nu } \lambda^{(n_{\ell+\overline{\ell}})}.$
This concludes the proof of  \eqref{lemmaHMU2} and hence of the Lemma.
  \end{proof}


\section*{Acknowledgments}

We would like to thank M.~Lema{\'{n}}czyk and J.P.~Thouvenot for their interest in the questions here addressed. J.P.~Thouvenot asked C.U. several years ago whether the flows she proved to be mixing in \cite{Ul:mix} are mixing of all orders and suggested to try to prove the Ratner property for them; M.~Lema{\'{n}}czyk has inspired and motivated the authors, in particular A.K., to look for suitable variations of the Ratner property.  We also thank him for useful discussions. We are thankful to J.~Chaika for his comments on the first version of this paper and to the referee of the paper for his/her careful reading and corrections. The collaboration that led to this paper was started in occasion of the {\it Ergodic Theory and Dynamical Systems} conference held in Toru{\'n} in May 2014; we thank the organizers and the funding bodies for providing us the opportunity to begin this work. 
J. K.-P. is supported by Narodowe Centrum Nauki grant 
UMO-2014/15/B/ST1/03736; C. U. is supported by the ERC grant \emph{ChaParDyn} and by the {\it Leverhulm Trust} through a Leverhulme Prize.  The research leading
to these results has received funding from the European Research Council under the European
Union Seventh Framework Programme (FP/2007-2013) / ERC Grant Agreement n.
335989.

\small
\bibliography{biblio_RatnerLogAsym}

\bigskip
\footnotesize

\noindent
Adam Kanigowski\\
\textsc{Penn State University Mathematics Department,  University Park, \\ State College, PA 16802, USA}\\
\noindent
\texttt{adkanigowski@gmail.com}

\medskip

\noindent
Joanna Ku\l aga-Przymus\\
\textsc{Institute of Mathematics, Polish Acadamy of Sciences, \'{S}niadeckich 8, 00-956 Warszawa, Poland}\\
\textsc{Faculty of Mathematics and Computer Science, Nicolaus Copernicus University, Chopina 12/18, 87-100 Toru\'{n}, Poland}\par\nopagebreak
\noindent
\texttt{joanna.kulaga@gmail.com}

\medskip

\noindent
Corinna Ulcigrai\\
\textsc{School of Mathematics, University of Bristol, Howard House, Queens Ave, BS8 1SN \\ Bristol, United Kingdom}\\
\noindent
\texttt{corinna.ulcigrai@bristol.ac.uk}

\end{document}